\documentclass[11pt]{amsart}
\allowdisplaybreaks

\usepackage{amssymb
,amsthm
,amsmath
,amscd
,mathtools	
}
\usepackage[all]{xy}
\usepackage{url}
\usepackage[dvipdfmx]{graphicx}
\usepackage[top=30truemm,bottom=30truemm,left=25truemm,right=25truemm]{geometry}

\newcommand{\A}{\mathbb{A}}
\newcommand{\Z}{\mathbb{Z}}
\newcommand{\Q}{\mathbb{Q}}
\newcommand{\R}{\mathbb{R}}
\newcommand{\C}{\mathbb{C}}

\renewcommand{\AA}{\mathcal{A}}
\newcommand{\DD}{\mathcal{D}}

\newcommand{\FF}{\mathcal{F}}

\newcommand{\MM}{\mathcal{M}}
\newcommand{\PP}{\mathcal{P}}
\newcommand{\Sc}{\mathcal{S}}
\newcommand{\VV}{\mathcal{V}}
\newcommand{\ZZ}{\mathcal{Z}}

\newcommand{\Ha}{\mathfrak{H}}

\newcommand{\fin}{\mathrm{fin}}

\newcommand{\vol}{\mathrm{vol}}
\newcommand{\Ik}{\mathrm{Ik}}
\newcommand{\Ind}{\mathrm{Ind}}
\newcommand{\ind}{\mathrm{ind}}
\newcommand{\End}{\mathrm{End}}
\newcommand{\Irr}{\mathrm{Irr}}
\newcommand{\FJ}{\mathrm{FJ}}

\newcommand{\im}{\mathrm{Im}}
\newcommand{\Hom}{\mathrm{Hom}}

\newcommand{\st}{\mathrm{st}}
\newcommand{\Tr}{\mathrm{Tr}}
\newcommand{\Ad}{\mathrm{Ad}}
\newcommand{\Sym}{\mathrm{Sym}}

\newcommand{\sgn}{\mathrm{sgn}}
\newcommand{\St}{\mathrm{St}}

\newcommand{\SL}{\mathrm{SL}}
\newcommand{\GL}{\mathrm{GL}}
\newcommand{\Mat}{\mathrm{Mat}}
\newcommand{\SO}{\mathrm{SO}}
\newcommand{\Sp}{\mathrm{Sp}}
\newcommand{\Mp}{\widetilde{\mathrm{Sp}}}

\newcommand{\WD}{\mathit{WD}}

\newcommand{\lam}{\lambda}

\newcommand{\bm}{\mathbf{m}}
\newcommand{\bn}{\mathbf{n}}
\newcommand{\bv}{\mathbf{v}}
\newcommand{\bi}{\mathbf{i}}
\renewcommand{\1}{{\bf 1}}
\newcommand{\I}{\sqrt{-1}}
\newcommand{\pair}[1]{\langle #1 \rangle}
\newcommand{\half}[1]{\frac{#1}{2}}
\newcommand{\cl}[1]{\widetilde{#1}}
\newcommand{\iif}{&\quad&\text{if }}

\newcommand{\resp}{resp.~}
\newcommand{\bs}{\backslash}

\newcommand{\ep}{\varepsilon}

\newcommand{\semi}{\mathrm{s.s.}}

\newcommand{\kk}{\mathfrak{k}}
\newcommand{\oo}{\mathfrak{o}}
\newcommand{\p}{\mathfrak{p}}
\renewcommand{\sp}{\mathfrak{sp}}

\newtheorem{thm}{Theorem}[section]
\newtheorem{lem}[thm]{Lemma}
\newtheorem{prop}[thm]{Proposition}
\newtheorem{cor}[thm]{Corollary}
\newtheorem{rem}[thm]{Remark}
\newtheorem{conj}[thm]{Conjecture}
\newtheorem{ex}[thm]{Example}
\newtheorem{hypo}[thm]{Hypothesis}
\newtheorem{defi}[thm]{Definition}

\makeatletter
\def\iddots{\mathinner{\mkern1mu\raise\p@
	\hbox{.}\mkern2mu\raise4\p@\hbox{.}\mkern2mu
	\raise7\p@\vbox{\kern7\p@\hbox{.}}\mkern1mu}}
\def\adots{\mathinner{\mkern2mu\raise\p@\hbox{.}
 \mkern2mu\raise4\p@\hbox{.}\mkern1mu
 \raise7\p@\vbox{\kern7\p@\hbox{.}}\mkern1mu}}
\makeatother

\title{A theory of Miyawaki liftings: The Hilbert--Siegel case}
\author{Hiraku Atobe}
\date{}
\address{
Graduate School of Mathematical Sciences, The University of Tokyo, 3-8-1 Komaba, Meguro-ku, Tokyo, 153-8914, Japan
}
\email{
atobe@ms.u-tokyo.ac.jp
}

\begin{document}
\maketitle
\begin{abstract}
The Miyawaki liftings are defined by the pullbacks of Ikeda liftings.
Recently, Ikeda and Yamana extended the theory of Ikeda liftings.
In this paper, using their results, 
we establish a theory of Miyawaki liftings, both locally and globally.
In the local theory, we describe the Miyawaki liftings for almost tempered unitary representations explicitly.
In the global theory, we discuss the non-vanishing of the Miyawaki liftings
using seesaw identities and the global Gan--Gross--Prasad conjecture.
As an application of local Miyawaki liftings, 
we prove a new case of the local Gan--Gross--Prasad conjecture.
\end{abstract}

\tableofcontents

\section{Introduction}
In 1992, Miyawaki \cite{M} predicted the existence of certain Siegel modular forms.
Let $S_k(\Sp_n(\Z))$ be the space of Siegel cusp forms of degree $n$, weight $k$, and level one.

\begin{conj}[Miyawaki \cite{M}]\label{miyawaki}
For normalized Hecke eigenforms $f \in S_{2k-4}(\SL_2(\Z))$ and $g \in S_k(\SL_2(\Z))$, 
there should exist a Hecke eigenform $F_{f,g} \in S_k(\Sp_3(\Z))$
whose standard $L$-function is given by
\[
L(s, F_{f,g}, \st) = L(s, g, \Ad) L(s+k-2, f) L(s+k-3, f).
\]
\end{conj}

For the spinor $L$-function of $F_{f,g}$, see \cite{H1, H2}.
\vskip 10pt

In 2006, to approach Miyawaki's conjecture, Ikeda \cite{I2} constructed certain liftings, 
which are now called the Miyawaki liftings, as follows: 
For positive even $k$,
a normalized Hecke eigenform $f \in S_{2k-2(n+r)}(\SL_2(\Z))$
gives the Ikeda lift $F^{(2n+2r)} \in S_{k}(\Sp_{2n+2r}(\Z))$ (defined up to a constant).
For the (classical) Ikeda lifting, see \cite{I1}.
For a Hecke eigenform $g \in S_{k}(\Sp_{r}(\Z))$, 
Ikeda \cite{I2} defined the Miyawaki lift $\MM^{(2n+r)}(g, F^{(2n+2r)})$ by the integral
\begin{align*}
&\MM^{(2n+r)}(g, F^{(2n+2r)})(Z_{2n+r}) 
\\&=
\int_{\Sp_r(\Z) \bs \Ha_r}
F^{(2n+2r)}(
\begin{pmatrix}
Z_{2n+r} & 0 \\ 0 & Z_{r}
\end{pmatrix}
)
\overline{g^c(Z_r)} (\det \im Z_r)^{k-r-1} dZ_r, 
\end{align*}
where $\Ha_r$ is the Siegel upper half space of genus $r$ 
and we set $g^c(Z_r) = \overline{g(-\overline{Z_r})}$.
It is easy to see that $\MM^{(2n+r)}(g, F^{(2n+2r)}) \in S_{k}(\Sp_{2n+r}(\Z))$.
Ikeda proved the following:

\begin{thm}[Ikeda {\cite[Theorem 1.1]{I2}}]
If $\MM^{(2n+r)}(g, F^{(2n+2r)})$ is not identically zero, 
then it is a Hecke eigenform with standard $L$-function
\[
L(s, \MM^{(2n+r)}(g, F^{(2n+2r)}), \st)
=
L(s, g, \st)\prod_{i=1}^{2n}L(s+k-r-i,f).
\]
\end{thm}

Therefore Miyawaki's conjecture (Conjecture \ref{miyawaki}) was reduced to 
the non-vanishing of $\MM^{(3)}(g, F^{(4)})$.
Ikeda gave a conjectural formula for the Petersson norm of $\MM^{(2n+r)}(g, F^{(2n+2r)})$, 
and predicted the following conjecture.
\begin{conj}[Ikeda {\cite[Conjecture 5.1]{I2}}]\label{ik}
\begin{enumerate}
\item
When $n=0$, the Miyawaki lift $\MM^{(r)}(g, F^{(2r)})$ is nonzero 
if and only if the central value of the tensor product $L$-function $L(s, \st(g) \boxtimes f)$ is nonzero. 
\item
When $n>0$,  the Miyawaki lift $\MM^{(2n+r)}(g, F^{(2n+2r)})$ is always nonzero.
\end{enumerate}
\end{conj}

Ichino \cite{Ic} and Xue \cite{X2} proved Conjecture \ref{ik} (1) for the case where $r=1$ independently. 
Garrett--Heim \cite{GH} established a Hecke duality of Ikeda liftings and gave 
a preliminary answer to Conjecture \ref{ik} (1).
\vskip 10pt

Nowadays Miyawaki's conjecture (Conjecture \ref{miyawaki}) follows from 
Arthur's multiplicity formula (\cite[Theorem 1.5.2]{Ar}), which was established in 2013.
However, this formula tells us only the existence of modular forms (or automorphic representations).
The non-vanishing of Miyawaki liftings (Conjecture \ref{ik}) is still open and interesting.
The integral representations of Miyawaki liftings would imply several properties, 
which do not follow from Arthur's multiplicity formula.
\vskip 10pt

Recently, Ikeda--Yamana \cite{IY} extended the theory of Ikeda liftings.
The purpose of this paper is to establish a theory of Miyawaki liftings using the extended Ikeda liftings.
We will define the Miyawaki liftings more generally, and give their several properties.
Moreover, we will approach the non-vanishing problem of Miyawaki liftings
using the Gan--Gross--Prasad conjecture.
\vskip 10pt

To describe our results, 
let $F$ be a totally real number field, and $\A$ be the ring of adeles of $F$. 
The ring of finite adeles of $F$ is denoted by $\A_\fin$.
For a place $v$ of $F$, we write $v < \infty$ (\resp $v \mid \infty$) 
if $v$ is a finite place (\resp if $v$ is an infinite place).
Fix a non-trivial unitary character $\psi$ of $\A/F$ such that 
for $v \mid \infty$, the local component $\psi_v$ is of the form
$\psi_v(x_v) = \exp(2\pi a_v \I x_v)$ for $x_v \in F_v \cong \R$ with fixed $a_v > 0$.
Let $\tau = \otimes'_v \tau_v$ be an irreducible cuspidal automorphic representation 
of $\mathrm{PGL}_2(\A)$
satisfying the following conditions: 
\begin{itemize}
\item[(A1)]
For $v < \infty$, $\tau_v$ is a principal series $\mu_v \times \mu_v^{-1}$. 
\item[(A2)]
For $v \mid \infty$, $\tau_v$ is a discrete series representation 
with lowest weight $\pm2 k_v$, where $k_v > 0$. 
\item[(A3)]
The root number $\ep(1/2, \tau)$ is equal to $1$.
\end{itemize}
Let $\Sp_n$ be the symplectic group of rank $n$.
For each place $v$, 
we denote the metaplectic double cover of $\Sp_n(F_v)$ by $\Mp_n(F_v)$.
Identify $\Mp_n(F_v) = \Sp_n(F_v) \times \{\pm1\}$ as sets.
Let $P_n$ be the Siegel parabolic subgroup of $\Sp_n$, 
and $\cl{P}_n(F_v)$ be the inverse image of $P_n(F)$ in $\Mp_n(F_v)$.
When $v < \infty$, we set 
\[
I_{\psi_v}^{(n)}(\tau_v) = \Ind_{\cl{P}_n(F_v)}^{\Mp_n(F_v)}(\mu_v^{(n)})
\]
to be a degenerate principal series.
Here, 
\[
\mu_v^{(n)}(
\begin{pmatrix}
A & 0 \\ 0 & {}^tA^{-1}
\end{pmatrix}, 
\zeta
)
= \zeta^n \left(\frac{\alpha_{\psi_v}(1)}{\alpha_{\psi_v}(\det A)}\right)^n \mu_v(\det A)
\]
for $A \in \GL_n(F_v)$ and $\zeta \in \{\pm1\}$, 
where $\alpha_{\psi_v}(a)$ is the Weil constant (see \S \ref{weil}).
It is known that $I_{\psi_v}^{(n)}(\tau_v)$ is irreducible (\cite{KR1, Sw}).
Set $k+(n/2) = (k_v+(n/2))_v \in \prod_{v \mid \infty}\Z$.
Let $\Mp_n(\A)$ be the metaplectic double cover of the adele group $\Sp_n(\A)$.
We denote the set of holomorphic cusp forms of weight $k+(n/2)$
by $\Sc_{k+(n/2)}(\Sp_n(F) \bs \Mp_n(\A))$.
For the definition, see \S \ref{sec.global}.
Ikeda--Yamana \cite{IY} showed that the irreducible representation 
$I_{\psi}^{(n)}(\tau) = \otimes'_{v < \infty}I_{\psi_v}^{(n)}(\tau_v)$ of $\Mp_n(\A_\fin)$
appears in $\Sc_{k+(n/2)}(\Sp_n(F) \bs \Mp_n(\A))$ with multiplicity one.
We denote the unique subrepresentation of $\Sc_{k+(n/2)}(\Sp_n(F) \bs \Mp_n(\A))$ 
which is isomorphic to $I_{\psi}^{(n)}(\tau)$ by $\Ik_{\psi}^{(n)}(\tau)$, 
and call it the Ikeda lift of $\tau$.
\vskip 10pt

Now let $n, r$ be non-negative integers.
Then we have an embedding $\iota \colon \Sp_n \times \Sp_r \rightarrow \Sp_{n+r}$ given by
\[
\iota\left(
\begin{pmatrix}
A_1 & B_1 \\ C_1 & D_1
\end{pmatrix}, 
\begin{pmatrix}
A_2 & B_2 \\ C_2 & D_2
\end{pmatrix}
\right)
=
\left(
\begin{array}{cc|cc}
A_1 & 0 & B_1 & 0 \\
0 & A_2 & 0 & B_2 \\
\hline
C_1 & 0 & D_1 & 0 \\
0 & C_2 & 0 & D_2 
\end{array}
\right).
\]
For an admissible representation $\pi$ of $\Mp_{r}(\A_\fin)$
occurring in $\Sc_{k+(n+r)/2}(\Sp_r(F) \bs \Mp_r(\A))$, 
we define the (global) Miyawaki lift $\MM_{\psi, \tau}^{(n)}(\pi)$ of $\pi$ by 
the representation of $\Mp_n(\A_\fin)$ generated by the integrals
\[
\MM^{(n)}((g_n, \zeta_n); \varphi, \FF) 
= \int_{\Sp_{r}(F) \bs \Sp_{r}(\A)} \FF(\iota(g_n, g_r), \zeta_n\zeta_r) \overline{\varphi(g_r, \zeta_r)} dg_r
\]
for $\varphi \in \pi$, $\FF \in \Ik_{\psi}^{(n+r)}(\tau)$ and $(g_n, \zeta_n) \in \Mp_n(\A)$. 
This is a subrepresentation of $\Sc_{k+(n+r)/2}(\Sp_n(F) \bs \Mp_n(\A))$.
We summarize the properties of Miyawaki liftings.

\begin{thm}\label{main1}
Let $\pi$ be an irreducible representation of $\Mp_{r}(\A_\fin)$
occurring in $\Sc_{k+(n+r)/2}(\Sp_r(F) \bs \Mp_r(\A))$, 
and $\MM_{\psi, \tau}^{(n)}(\pi)$ be its Miyawaki lift.
Suppose that $\MM_{\psi, \tau}^{(n)}(\pi) \not= 0$ and $n \geq r$.
\begin{enumerate}
\item
If $\pi$ has an $A$-parameter $\Psi$, 
then $\MM_{\psi, \tau}^{(n)}(\pi)$ has an $A$-parameter 
\[
\Psi \boxplus \tau\chi_{-1}^{[(n+r)/2]}[n-r], 
\]
where $\chi_{-1}$ be the quadratic character of $\A^\times/F^\times$ corresponding to $F(\I)/F$.
For the $A$-parameters, see \S \ref{Ap}, or Appendix \ref{app.AMF} for more precision.

\item
Suppose that $\pi$ has a tempered $A$-parameter. 
Then $\MM^{(n)}_{\psi, \tau}(\pi)$ is irreducible, and 
$\MM^{(n)}_{\psi, \tau}(\pi) \cong \otimes'_{v < \infty} \MM^{(n)}_{\psi_v, \tau_v}(\pi_v)$.
Here, $\MM_{\psi_v, \tau_v}^{(n)}(\pi_v)$ is the local Miyawaki lift of $\pi_v$ described below.

\item
We have 
\[
\pi \subset \MM_{\psi, \tau}^{(r)}\left(\MM_{\psi, \tau}^{(n)}(\pi)\right).
\]
If $\pi$ has a tempered $A$-parameter and if $r \leq n \leq r+1$ or $n > 2r$, 
then the inclusion is in fact equality.

\end{enumerate}
\end{thm}

The statements in Theorem \ref{main1} are proven in 
Proposition \ref{Apara}, Proposition \ref{dual}, and Theorem \ref{irred}.
Remark that 
Arthur's multiplicity formula (\cite[Theorem 1.5.2]{Ar}) 
should imply the existence of an irreducible subrepresentation of $\Sc_{k+(n+r)/2}(\Sp_r(F) \bs \Sp_r(\A))$
satisfying Theorem \ref{main1} (1)
at least when $n+r$ is even.
\vskip 10pt

The definition of Miyawaki liftings is similar to the one of theta liftings.
As Miyawaki liftings are given by the pullbacks of Ikeda liftings, 
theta liftings are defined by the pullbacks of theta functions.
One of the most important properties of theta liftings is the seesaw identities, 
which have several applications.
For instance, by using seesaw identities, 
Ichino \cite{Ic} computed the pullbacks of Saito--Kurokawa liftings, 
which is regarded as a special case of Ikeda's conjectural formula for 
the Petersson norm of a certain Miyawaki lift.
For another example, Xue \cite{X} reduced the refined version of 
the Gan--Gross--Prasad conjecture for the symplectic-metaplectic case
to the one for the special orthogonal case.
\vskip 10pt

Miyawaki liftings also satisfy certain seesaw identities.
To state these identities, we introduce the Fourier--Jacobi periods.
Fix a totally positive element $\xi \in F^\times$ 
and set $\psi_\xi(x) = \psi(\xi x)$ for $x \in \A$.
Let $n' = n$ or $n' = n-1$.
For $\varphi \in \Sc_{l/2}(\Sp_{n}(F) \bs \Mp_{n}(\A))$, 
$\varphi' \in \Sc_{l'/2}(\Sp_{n'}(F) \bs \Mp_{n'}(\A))$, and $\phi \in \Sc(\A^{n'})$, 
we define the Fourier--Jacobi period $\PP_{n,n', \psi_\xi}(\varphi, \overline{\varphi'}, \phi)$ 
by the integral
\[
\left\{
\begin{aligned}
&\int_{\Sp_{n}(F) \bs \Sp_{n}(\A)} \varphi(g, \zeta) \overline{\varphi'(g,\zeta)} 
\overline{\Theta_{\psi_\xi}^{\phi}(g,\zeta)} dg, \iif n' = n, \\
&\int_{V_{n-1}(F) \bs V_{n-1}(\A)} \int_{\Sp_{n-1}(F) \bs \Sp_{n-1}(\A)} 
\varphi(v(g, \zeta)) \overline{\varphi'(g,\zeta)} 
\overline{\Theta_{\psi_\xi}^{\phi}(v(g,\zeta))} dgdv, \iif n' = n-1.
\end{aligned}
\right.
\]
Here, $\Theta_{\psi_\xi}^{\phi}$ is the theta function associated to $\phi \in \Sc(\A^{n'})$, 
which is a genuine automorphic form on a Jacobi group $\cl{J}_{n'}(\A) = \Mp_{n'}(\A) \ltimes V_{n'}(\A)$, 
where $V_{n'} \cong F^{2n'} \oplus F$ is a Heisenberg group.
Let $\Sc(\A^{n'})_\xi$ be the subspace of $\Sc(\A^{n'})$ consisting of lowest weight vectors
of the Weil representation of $\cl{J}_{n'}(\A)$ with respect to $\psi_\xi$.

\begin{prop}[Seesaw identity (Proposition \ref{GSS})]\label{seesaw}
Let $\pi$ and $\pi'$ be irreducible representations of $\Mp_{r}(\A_\fin)$ and $\Mp_{n-1}(\A_\fin)$
occurring in $\Sc_{k+(n+r)/2}(\Sp_{r}(F) \bs \Mp_{r}(\A))$ and 
$\Sc_{k+(n-1+r)/2}(\Sp_{n-1}(F) \bs \Mp_{n-1}(\A))$, respectively.
Fix a totally positive element $\xi \in F^\times$. 
\begin{enumerate}
\item
If there exist $\MM^{(n)}(\varphi_1, \FF_1) \in \MM^{(n)}_{\psi, \tau}(\pi)$, 
$\varphi'_1 \in \pi'$, and $\phi_1 \in \Sc(\A^{n-1})_\xi$ such that 
\[
\PP_{n,n-1,\psi_\xi}(\MM^{(n)}(\varphi_1, \FF_1), \overline{\varphi'_1}, \phi_1) \not= 0, 
\]
then 
there exist $\varphi_2 \in \pi$, $\MM^{(r)}(\varphi_2', \FF'_2) \in \MM^{(r)}_{\psi, \tau\chi_\xi}(\pi')$, 
and $\phi_2 \in \Sc(\A^{r})_\xi$ such that
\[
\PP_{r,r,\psi_\xi}(\varphi_2, \overline{\MM^{(r)}(\varphi_2', \FF'_2)}, \phi_2) \not= 0.
\]
Moreover, we can take $\varphi_2 = \varphi_1$ and $\varphi_2' = \varphi_1'$.

\item
Assume that $n+r \geq 2$. 
If there exist $\varphi_2 \in \pi$, $\MM^{(r)}(\varphi_2', \FF'_2) \in \MM^{(r)}_{\psi, \tau\chi_\xi}(\pi')$, 
and $\phi_2 \in \Sc(\A^{r})_\xi$ such that
\[
\PP_{r,r,\psi_\xi}(\varphi_2, \overline{\MM^{(r)}(\varphi_2', \FF'_2)}, \phi_2) \not= 0, 
\]
then 
there exist $\MM^{(n)}(\varphi_1, \FF_1) \in \MM^{(n)}_{\psi, \tau}(\pi)$, 
$\varphi'_1 \in \pi'$, and $\phi_1 \in \Sc(\A^{n-1})_\xi$ such that 
\[
\PP_{n,n-1,\psi_\xi}(\MM^{(n)}(\varphi_1, \FF_1), \overline{\varphi'_1}, \phi_1) \not= 0.
\]
Moreover, we can take $\varphi_1 = \varphi_2$ and $\varphi_1' = \varphi_2'$.
\end{enumerate}
\end{prop}
We shall write these properties as the following seesaw diagram:
\[
\xymatrix{
\Mp_r(F) \times \Mp_r(F) \ar@{-}[d] \ar@{-}[dr]& \Mp_n(F) \ar@{-}[d]\\
\Mp_r(F) \ar@{-}[ur]& \Mp_{n-1}(F) \ltimes V_{n-1}(F).
}
\]
Remark that the proof of Proposition \ref{seesaw} uses 
the non-vanishing of Fourier--Jacobi coefficients of Ikeda liftings (Proposition \ref{FJcoeff2} (3)).
In particular, the seesaw identities would not follow from Arthur's multiplicity formula.
\vskip 10pt

Finally, following Ikeda's conjecture (Conjecture \ref{ik}), 
we formulate a conjecture on the non-vanishing of Miyawaki liftings.

\begin{conj}[Conjecture \ref{M0}]\label{nonvanishing}
Let $\pi$ be an irreducible representation of $\Mp_r(\A_\fin)$
occurring in $\Sc_{k+(n+r)/2}(\Sp_r(F) \bs \Mp_r(\A))$.
\begin{enumerate}
\item
When $n=r$, 
the Miyawaki lift $\MM^{(r)}_{\psi, \tau}(\pi)$ is nonzero 
if and only if the central value of the Rankin--Selberg $L$-function 
$L(s, \pi \times \tau\chi_{-1}^r)$ is nonzero.
\item
When $n > r$, 
the Miyawaki lift $\MM^{(n)}_{\psi, \tau}(\pi)$ is always nonzero.
\end{enumerate}
\end{conj}

Using the seesaw identities (Proposition \ref{seesaw}), 
we can relate Conjecture \ref{nonvanishing} with 
the Gan--Gross--Prasad conjecture (see Conjecture \ref{GGP-S} below).
The following theorem summarizes Theorem \ref{rrr}, Corollary \ref{11}, and Theorem \ref{n>r}.

\begin{thm}
\begin{enumerate}
\item
Assume the Gan--Gross--Prasad conjecture (Conjecture \ref{GGP-S})
and Hypothesis \ref{hypo} below.
Then Conjecture \ref{nonvanishing} for $n = r, r+1$ holds when $\pi$ has a tempered $A$-parameter.

\item
In particular, Conjecture \ref{nonvanishing} for $n=r=1$ holds unconditionally.

\item
When $n \geq r+2$, 
Conjecture \ref{nonvanishing} for $\MM^{(n-1)}_{\psi, \tau'}(\pi')$ implies
Conjecture \ref{nonvanishing} for $\MM^{(n)}_{\psi, \tau}(\pi)$.
\end{enumerate}
\end{thm}
\vskip 10pt

There is a local analogue of Miyawaki liftings.
Now let $F$ be a non-archimedean local field of characteristic zero, 
and $\psi$ be a non-trivial additive character of $F$.
Let $\chi_{-1}$ be the quadratic character of $F^\times$ corresponding to $F(\I)/F$.
A local analogue of the Ikeda lifting is the degenerate principal series 
$I_{\psi}^{(n+r)}(\tau) = \Ind_{\cl{P}_{n+r}(F)}^{\Mp_{n+r}(F)}(\mu^{(n+r)})$, 
where $\tau = \mu \times \mu^{-1}$ is a principal series of $\GL_2(F)$ 
with $\mu$ being a unitary character of $F^\times$.
Recall that there is an embedding $\iota \colon \Sp_n \times \Sp_r \hookrightarrow \Sp_{n+r}$.
For an irreducible representation $\pi$ of $\Mp_{r}(F)$, 
on which the kernel $\{\pm1\}$ of the covering map $\Mp_{r}(F) \twoheadrightarrow \Sp_{r}(F)$
acts by $(\pm1)^{n+r}$, 
the maximal $\pi$-isotypic quotient of $I^{(n+r)}_\psi(\tau)$ is of the form
\[
\MM_{\psi, \tau}^{(n)}(\pi) \boxtimes \pi
\]
for some smooth representation $\MM_{\psi, \tau}^{(n)}(\pi)$ of $\Mp_{n}(F)$, 
on which the kernel $\{\pm1\}$ of the covering map $\Mp_{n}(F) \twoheadrightarrow \Sp_{n}(F)$
acts by $(\pm1)^{n+r}$.
We call $\MM_{\psi, \tau}^{(n)}(\pi)$ the local Miyawaki lift of $\pi$.
The following is the local main theorem.

\begin{thm}[Theorem \ref{howe}]
Let $\mu$ be a unitary character of $F^\times$, and 
$\pi$ be an irreducible representation of $\Mp_{r}(F)$ on which $\{\pm1\}$ acts by $(\pm1)^{n+r}$.
Suppose that $n \geq r$.

\begin{enumerate}
\item
The local Miyawaki lift $\MM_{\psi, \tau}^{(n)}(\pi)$ is nonzero and of finite length. 

\item
If $\pi$ is almost tempered and unitary, then 
\begin{itemize}
\item
$\MM_{\psi, \tau}^{(n)}(\pi)$ is irreducible; 
\item
$\MM_{\psi, \tau}^{(n)}(\pi) \cong (\mu' \circ {\det}_{n-r}) \rtimes \pi$ with $\mu' = \mu\chi_{-1}^{[(n+r)/2]}$; 
\item
$\MM_{\psi, \tau}^{(n)}(\pi)$ is isomorphic to the unique irreducible quotient of the induced representation
\[
\left\{
\begin{aligned}
&\tau'|\cdot|^{\half{n-r-1}} \times \tau'|\cdot|^{\half{n-r-3}} \times \dots \times \tau'|\cdot|^{\half{1}} \rtimes \pi
\iif n+r \equiv 0 \bmod 2, \\
&\tau'|\cdot|^{\half{n-r-1}} \times \tau'|\cdot|^{\half{n-r-3}} \times \dots \times \tau'|\cdot|^{1} 
\times \mu' \rtimes \pi
\iif n+r \equiv 1 \bmod 2, 
\end{aligned}
\right.
\]
where $\tau' = \tau \otimes \chi_{-1}^{[(n+r)/2]} = \mu' \times \mu'^{-1}$.
\end{itemize}

\item
For any irreducible almost tempered unitary representations $\pi_1$ and $\pi_2$, 
we have
\[
\MM_{\psi, \tau}^{(n)}(\pi_1) \cong \MM_{\psi, \tau}^{(n)}(\pi_2)
\implies 
\pi_1 \cong \pi_2. 
\]

\item
Set
\[
\pi' = \MM_{\psi, \tau}^{(r)}\left( \MM_{\psi, \tau}^{(n)}(\pi) \right).
\]
Assume one of the following:
\begin{itemize}
\item
$\pi$ is almost tempered and unitary, 
and $r \leq n \leq r+1$ or $n > 2r$; 
\item
$\pi$ is discrete series.
\end{itemize}
Then all irreducible subquotients of $\pi'$ are isomorphic to $\pi$, 
and the maximal semisimple quotient of $\pi'$ is irreducible.

\item
Suppose that $\mu$ is unramified, 
and set $\alpha = (\mu\chi_{-1}^{[(n+r)/2]})(\varpi)$.
If $\pi$ is an irreducible unramified representation of $\Mp_r(F)$ 
with the Satake parameter $\{\beta_1^{\pm1}, \dots, \beta_r^{\pm1}\}$, 
then $\MM_{\psi, \tau}^{(n)}(\pi)$ has a unique irreducible unramified quotient. 
Its Satake parameter is equal to
\[
\{\beta_1^{\pm1}, \dots, \beta_r^{\pm1}\} \cup 
\{\alpha^{\pm1}q^{\half{n-r-1}}, \alpha^{\pm1}q^{\half{n-r-3}}, \dots, \alpha^{\pm1}q^{-\half{n-r-1}}\}
\]
as multisets.

\end{enumerate}
\end{thm}
\vskip 10pt

There is also a local analogue of seesaw identities (Proposition \ref{seesaw}).
For $\xi \in F^\times$, 
we denote the Weil representations of the Jacobi group $\cl{J}_{n-1}(F) = \Mp_{n-1}(F) \ltimes V_{n-1}(F)$ 
and the metaplectic group $\Mp_r(F)$ with respect to $\psi_\xi$
by $\omega_{\psi_\xi}^{(n-1)}$ and $\omega_{\psi_\xi}^{(r)}$, respectively.

\begin{prop}[Seesaw identity (Proposition \ref{LSS})]
Let $\pi$ and $\pi'$ be irreducible representations of $\Mp_r(F)$ and $\Mp_{n-1}(F)$, 
on which $\{\pm1\}$ acts by $(\pm1)^{r+n}$ and $(\pm1)^{r+n-1}$, respectively.
Then 
\[
\Hom_{\cl{J}_{n-1}(F)}(\MM_{\psi, \tau}^{(n)}(\pi)|_{\cl{J}_{n-1}(F)}, \pi' \otimes \omega_{\psi_\xi}^{(n-1)}) 
\not= 0
\iff
\Hom_{\Mp_r(F)}(\MM_{\psi, \tau\chi_\xi}^{(r)}(\pi') \otimes \omega_{\psi_\xi}^{(r)}, \pi) \not= 0.
\]
\end{prop}

As an application, 
the seesaw identity gives a quite new example of 
the local Gan--Gross--Prasad conjecture for a non-generic case
(Theorem \ref{GGP-new}).
\vskip 10pt

This paper is organized as follows. 
In \S \ref{ikeda}, we recall the theory of Ikeda liftings extended by Ikeda--Yamana \cite{IY}.
The local and global theories of Miyawaki liftings are explained 
in \S \ref{local} and \S \ref{global}, respectively.
In \S \ref{relation}, 
we discuss Conjecture \ref{nonvanishing} and its relation with the Gan--Gross--Prasad conjecture.
In Appendices \ref{s.jac}, \ref{langlands} and \ref{sec.GGP}, 
we recall results on Jacquet modules of representations of metaplectic groups, 
the local and global Langlands program, 
and the Gan--Gross--Prasad conjecture, respectively.

\subsection*{Acknowledgments}
The author is grateful to Shunsuke Yamana, Atsushi Ichino, and Tamotsu Ikeda for their helpful comments. 
This work was supported by the Foundation for Research Fellowships of Japan Society for the Promotion of Science for Young Scientists (PD) Grant 29-193. 
\par

\section{Ikeda liftings and their Fourier--Jacobi coefficients}\label{ikeda}
In this section, we recall the theory of Ikeda liftings along with \cite{IY}. 

\subsection{Metaplectic group and its representations}
Let $F$ be a totally real number field.
The symplectic group $\Sp_n$ is an algebraic group defined over $F$ given by
\[
\Sp_n(F) = \left\{
g \in \GL_{2n}(F) \ |\ 
{}^tg 
\begin{pmatrix}
0 & -\1_n \\ \1_n & 0
\end{pmatrix}
g
=
\begin{pmatrix}
0 & -\1_n \\ \1_n & 0
\end{pmatrix}
\right\}.
\]
The set of symmetric matrices of size $n$ with coefficients in $F$ is denoted by $\Sym_n(F)$.
For $A \in \GL_n(F)$ and $B \in \Sym_n(F)$, set 
\[
\bm(A) = \begin{pmatrix}
A & 0 \\ 0 & {}^tA^{-1}
\end{pmatrix},\quad
\bn(B) = \begin{pmatrix}
\1_n & B \\ 0 & \1_n
\end{pmatrix}
\in \Sp_n(F).
\]
For each $k = 1, \dots, n$, 
we define a standard maximal parabolic subgroup $P_k(F)$ of $\Sp_n(F)$ by
\[
P_k(F) = \left\{
\left(
\begin{array}{cc|cc}
a & * & * & * \\
0 & A & * & B \\
\hline
0 & 0 & {}^t a^{-1} & 0 \\
0 & C & * & D
\end{array}
\right)
\ |\ 
a \in \GL_k(F),\ 
\begin{pmatrix}
A & B \\ C & D
\end{pmatrix}
\in \Sp_{n-k}(F)
\right\}.
\]
Let $P_k(F) = M_k(F) N_k(F)$ be the standard Levi decomposition, so that
the Levi $M_k(F)$ is isomorphic to $\GL_k(F) \times \Sp_{n-k}(F)$.
In particular, $P_n(F) = M_n(F)N_n(F)$ is the Siegel parabolic subgroup
with $M_n(F) = \{\bm(A)\ |\  A \in \GL_n(F)\}$ and $N_n(F) = \{\bn(B)\ |\ B \in \Sym_n(F)\}$. 
Let $B_n = \cap_{k=1}^{n}P_k$ be a Borel subgroup of $\Sp_n$.
A parabolic subgroup $P$ of $\Sp_n$ is called standard if $P$ contains $B_n$.
\par

For each place $v$ of $F$, we denote by $\Mp_n(F_v)$ the metaplectic group, 
i.e., the topological double cover of the symplectic group $\Sp_n(F_v)$.
As sets, we identify $\Mp_n(F_v)$ with $\Sp_n(F_v) \times \{\pm1\}$.
Then the group law of $\Mp_n(F_v)$ is given by
\[
(g_1, \zeta_1)(g_2, \zeta_2) = (g_1g_2, c_v(g_1, g_2)\zeta_1\zeta_2)
\]
for $g_1, g_2 \in \Sp_n(F_v)$ and $\zeta_1, \zeta_2 \in \{\pm1\}$, 
where $c_v(g_1, g_2)$ is Rao's $2$-cocycle of $\Sp_n(F_v)$ with values in $\{\pm1\}$.
The double cover $\Mp_n(F_v) \rightarrow \Sp_n(F_v)$ splits over the subgroup $N_n(F_v)$ by
$\bn(B) \mapsto (\bn(B), 1)$.
If $F_v$ is a non-archimedean local field whose residue characteristic is not $2$, 
there is a unique splitting
\[
\Sp_n(\oo_v) \rightarrow \Mp_n(F_v),\ g \mapsto (g, s(g)).
\]
Here, we denote by $\oo_v$ the ring of integers of $F_v$.
We identify $N_n(F_v)$ and $\Sp_n(\oo_v)$ with the images of these splittings.
If $H$ is a subgroup of $\Sp_n(F)$, the inverse image of $H$ in $\Mp_n(F)$ is denoted by $\cl{H}$.
\par

Next, we define the global metaplectic group.
We denote the adele ring of $F$ by $\A$.
Let $\mathfrak{S}$ be a finite set of places of $F$, 
which contains all places above $2$ and $\infty$.
Put
\[
\Sp_n(\A)_{\mathfrak{S}} 
= \prod_{v \in \mathfrak{S}} \Sp_n(F_v) \times \prod_{v \not\in \mathfrak{S}} \Sp(\oo_v).
\]
Then the double cover $\Mp_n(\A)_{\mathfrak{S}} \rightarrow \Sp_n(\A)_{\mathfrak{S}}$
is defined by the $2$-cocycle $\prod_{v \in \mathfrak{S}} c_v(g_{1,v}, g_{2,v})$.
For $\mathfrak{S}_1 \subset \mathfrak{S}_2$, there exists an embedding 
$\Mp_n(\A)_{\mathfrak{S}_1} \hookrightarrow \Mp_n(\A)_{\mathfrak{S}_2}$ given by 
\[
((g_v)_v, \zeta) \mapsto 
\left((g_v)_v, \zeta \prod_{v \in \mathfrak{S}_2 \setminus \mathfrak{S}_1}s_v(g_v)\right).
\]
The global metaplectic group $\Mp_n(\A)$ is defined by the inductive limit
\[
\Mp_n(\A) = \varinjlim_{\mathfrak{S}} \Mp_n(\A)_{\mathfrak{S}}, 
\]
where $\mathfrak{S}$ runs over all finite sets of places of $F$ containing all places above $2$ and $\infty$.
The covering $\Mp_n(\A) \rightarrow \Sp_n(\A)$ splits over $\Sp_n(F)$ uniquely.
We identify $\Sp_n(F)$ with the image of the splitting.
\par

\subsection{Weil representations of Jacobi groups}\label{weil}
We recall the Weil representation on a Jacobi group in the local setting.
Let $F$ be a local field of characteristic zero.
Fix a non-trivial unitary character $\psi$ of $F$. 
For $\xi \in F^\times$, we define a new non-trivial unitary character $\psi_\xi$ by
\[
\psi_\xi(x) = \psi(\xi x)
\]
for $x \in F$.
Let $\pair{,}$ be the quadratic Hilbert symbol.
For $x, \xi \in F^\times$, we set 
$\chi_\xi(x) = \pair{x,\xi}$.
For each Schwartz function $f \in \Sc(F)$, 
the Fourier transform $\hat{f}$ (with respect to $\psi_\xi$) is defined by
\[
\hat{f}(x) = \int_{F} f(y) \psi_{\xi}(xy)dy, 
\]
where $dy$ is the self-dual Haar measure on $F$ with respect to $\psi_\xi$.
For $a \in F^\times$, there exists an $8$-th root of unity $\alpha_{\psi_\xi}(a)$ such that
\[
\int_{F} f(x) \psi_\xi(ax^2) dx
= \alpha_{\psi_\xi}(a) |2a|^{-\half{1}}
\int_{F} \hat{f}(x) \psi_{\xi}(-\frac{x^2}{4a}) dx
\]
for any $f \in \Sc(F)$.
The constant $\alpha_{\psi_\xi}(a)$ is called the Weil constant.
It satisfies that $\alpha_{\psi_\xi}(ab^2) = \alpha_{\psi_\xi}(a)$ and 
\begin{align*}
\frac{\alpha_{\psi_\xi}(a)\alpha_{\psi_\xi}(b)}{\alpha_{\psi_\xi}(1)\alpha_{\psi_\xi}(ab)}
= \pair{a,b}
\end{align*}
for $a, b \in F^\times$.
In particular, 
\[
\left(\frac{\alpha_{\psi_\xi}(1)}{\alpha_{\psi_\xi}(a)}\right)^2
= \chi_{-1}(a), 
\]
where $\chi_{-1} = \pair{\cdot, -1}$ is the quadratic character associated to $F(\I)/F$.
\par

Put
\[
\bv(x,y,z) = \left(
\begin{array}{cc|cc}
1 & x & z & y \\
0 &\1_{n-1} & {}^ty & 0 \\
\hline
0 & 0 & 1 & 0 \\
0 & 0 & -{}^tx & \1_{n-1}
\end{array}
\right)
\in \Sp_n(F),
\]
where $x, y \in F^{n-1}$ are row vectors and $z \in F$.
We set
\begin{align*}
V(F) &= V_{n-1}(F) = \{\bv(x,y,z) \ |\ x, y \in F^{n-1}, z \in F\}, \\
X(F) &= X_{n-1}(F) = \{\bv(x,0,0) \ |\ x \in F^{n-1}\}, \\
Y(F) &= Y_{n-1}(F) = \{\bv(0,y,0) \ |\ y \in F^{n-1}\}, \\
Z(F) &= Z_{n-1}(F) = \{\bv(0,0,z) \ |\ z \in F\}.
\end{align*}
Note that $V(F)$ is a Heisenberg group.
We regard $\Mp_{n-1}(F)$ as a subgroup of $\Mp_n(F)$ by the embedding
\[
(
\begin{pmatrix}
A & B \\ C & D
\end{pmatrix}, \zeta
)
\mapsto
(
\left(
\begin{array}{cc|cc}
1 & 0 & 0 & 0 \\
0 &A & 0 & B \\
\hline
0 & 0 & 1 & 0 \\
0 & C & 0 & D
\end{array}
\right), \zeta
).
\]
By the Stone--von Neumann theorem, 
there is a unique irreducible admissible representation $\omega_{\psi_\xi}$ of $V(F)$ 
on which the center $Z(F)$ acts by $\psi_\xi$.
This representation $\omega_{\psi_\xi}$ extends to the Weil representation of 
the group $\cl{J}_{n-1}(F) = V(F) \rtimes \Mp_{n-1}(F)$.
We call $J_{n-1}(F) = V(F) \rtimes \Sp_{n-1}(F)$ a Jacobi group.
The representation $\omega_{\psi_\xi}$ is realized on the Schwartz space $\Sc(X(F))$
explicitly as follows: 
\begin{align*}
\omega_{\psi_\xi}(\bv(x,y,z))\phi(t) 
&= \psi_\xi(z + 2t \cdot {}^ty + x \cdot {}^ty) \phi(t + x), \\
\omega_{\psi_\xi}(\bm(A), \zeta)\phi(t)
&= \zeta \frac{\alpha_{\psi_\xi}(1)}{\alpha_{\psi_\xi}(\det A)}|\det A|^{\half{1}} \phi(tA),\\
\omega_{\psi_\xi}(\bn(B), \zeta)\phi(t)
&= \zeta \psi_\xi(t \cdot B \cdot {}^tt) \phi(t),\\
\omega_{\psi_\xi}(
\begin{pmatrix}
0 & -\1_{n-1} \\ \1_{n-1} & 0
\end{pmatrix}, \zeta
)\phi(t)
&= \zeta \alpha_{\psi_\xi}(1)^{-n+1}|2|_v^{\half{n-1}}
\int_{X(F)} \phi(u) \overline{\psi_\xi(2t \cdot {}^tu)} du
\end{align*}
for $\zeta \in \{\pm1\}$, $\bv(x,y,z) \in V(F)$, $A \in \GL_{n-1}(F)$, $B \in \Sym_{n-1}(F)$, 
and $\phi \in \Sc(X(F))$.
Here, $du = \prod_i du_i$ is the Haar measure on $X(F)$ 
with $du_i$ being the self-dual Haar measure on $F$ with respect to $\psi_\xi$.
The Weil representation $\omega_{\psi_\xi}$ is unitary with respect to the inner product
\[
(\phi_1, \phi_2) = \int_{X(F)} \phi_1(t)\overline{\phi_2(t)} dt
\]
for $\phi_1, \phi_2 \in \Sc(X(F))$.
\par

\subsection{Non-archimedean case}\label{sec.non-arch}
In this subsection, we assume that $F$ is non-archimedean.
For a smooth representation $\Pi$ of $\cl{J}_{n-1}(F)$, we put
\[
\FJ_{\psi_\xi}(\Pi) = (\Pi \otimes \overline{\omega_{\psi_\xi}})_{V(F)}.
\]
Here, $(\cdot)_{V(F)}$ means the maximal quotient on which $V(F)$ acts trivially.
We call $\FJ_{\psi_\xi}(\Pi)$ the Fourier--Jacobi module of $\Pi$ with index $\psi_\xi$.
We regard $\FJ_{\psi_\xi}(\Pi)$ as a representation of $\Mp_{n-1}(F)$.
Conversely, for a smooth representation $\pi$ of $\Mp_{n-1}(F)$, 
one can consider the tensor product 
\[
\pi \otimes \omega_{\psi_\xi}
\]
which is a smooth representation of $\cl{J}_{n-1}(F)$.
Note that both $\Pi \mapsto \FJ_{\psi_\xi}(\Pi)$ and $\pi \mapsto \pi \otimes \omega_{\psi_\xi}$
are exact functors.
The following proposition seems to be well-known, but we give a proof for readers.
\begin{prop}\label{FJ11}
Suppose that $F$ is non-archimedean.
The map 
\[
\pi \mapsto \pi \otimes \omega_{\psi_\xi}
\]
gives a $1$-$1$ correspondence between irreducible smooth representations of $\Mp_{n-1}(F)$ 
and irreducible smooth representations of $\cl{J}_{n-1}(F)$ on which $Z(F)$ acts by $\psi_\xi$.
The inverse mapping is given by the Fourier--Jacobi module $\Pi \mapsto \FJ_{\psi_\xi}(\Pi)$.
\end{prop}
\begin{proof}
Since any smooth representation of $V(F)$ on which $Z(F)$ acts by $\psi_\xi$
is a direct sum of copies of the Weil representation $\omega_{\psi_\xi}$, 
for any (nonzero) smooth representation $\Pi$ of $\cl{J}_{n-1}$, 
we can write $\Pi \cong \VV \otimes \Sc(X(F))$ as a representation of $V(F)$
for some (nonzero) vector space $\VV$ on which $V(F)$ acts trivially.
For $g \in \Mp_{n-1}(F)$, the operator $\Pi(g) \circ (\1_{\VV} \otimes \omega_{\psi_\xi}(g)^{-1})$
commutes with the action of $V(F)$.
By Schur's lemma, we have 
\[
\Pi(g) \circ (\1_{\VV} \otimes \omega_{\psi_\xi}(g)^{-1}) = \pi(g) \otimes \1_{\Sc(X(F))}
\]
for some $\pi(g) \in \mathrm{Aut}(\VV)$.
Then $\pi$ gives a group homomorphism $\pi \colon \Mp_{n-1}(F) \rightarrow \mathrm{Aut}(\VV)$. 
It is smooth since $\Pi$ and $\omega_{\psi_\xi}$ are smooth.
We conclude that $\Pi \cong \pi \otimes \omega_{\psi_\xi}$ as representations of $\cl{J}_{n-1}(F)$.
Note that $\FJ_{\psi_\xi}(\pi \otimes \omega_{\psi_\xi}) \cong \pi$ as representations of $\Mp_{n-1}(F)$.
In particular, $\FJ_{\psi_\xi}(\Pi) \not= 0$ for any nonzero smooth representation of $\cl{J}_{n-1}(F)$.
Since $-\otimes \omega_{\psi_\xi}$ and $\FJ_{\psi_\xi}$ are exact functors, 
we see that $\pi$ is irreducible as a representation of $\Mp_{n-1}(F)$
if and only if $\pi \otimes \omega_{\psi_\xi}$ is irreducible as a representation of $\cl{J}_{n-1}(F)$.
\end{proof}

For an irreducible smooth representation $\Pi$ of $\Mp_{n}(F)$, 
we denote by $\Pi_{\psi_\xi}$ the maximal quotient of $\Pi$ on which $Z(F)$ acts by $\psi_{\xi}$.
Then $\Pi_{\psi_\xi}$ is a smooth representation of $\cl{J}_{n-1}(F)$. 
By a similar argument to the proof of Proposition \ref{FJ11}, 
we have 
\[
\Pi_{\psi_\xi} \cong \FJ_{\psi_\xi}(\Pi|\cl{J}_{n-1}(F)) \otimes \omega_{\psi_\xi}.
\]
\par

Let $\mu$ be a unitary character of $F^\times$.
For each integer $n$, we define a character $\mu^{(n)}$ of $\cl{M}_n(F)$ by
\[
\mu^{(n)}((\bm(A), \zeta)) = \zeta^n \left(\frac{\alpha_\psi(1)}{\alpha_\psi(\det A)}\right)^n \mu(\det A).
\]
It is also regarded as a character of $\cl{P}_n(F)$.
If 
\[
\tau = \mu \times \mu^{-1} = \Ind_{P_{(1,1)}(F)}^{\GL_2(F)}(\mu \otimes \mu^{-1})
\] 
is an irreducible parabolic induction of $\GL_2(F)$, 
where $P_{(1,1)}$ is the Borel subgroup of $\GL_2$ consisting of upper triangular matrices, 
then the degenerate principal series
\[
I^{(n)}_{\psi}(\tau) = I_{\psi}(\mu^{(n)}) = \Ind_{\cl{P}_n(F)}^{\Mp_n(F)}(\mu^{(n)})
\]
is irreducible by \cite{KR1, Sw}.
See also \cite[Propositions 3.1, 5.1]{IY}.
\par

In general, the Fourier--Jacobi module $\FJ_{\psi_\xi}(\Pi|\cl{J}_{n-1}(F))$ 
is rarely irreducible since the restriction of $\Pi$ to $\cl{J}_{n-1}(F)$ is often reducible.
However, $\FJ_{\psi_\xi}$ sends degenerate principal series of $\Mp_{n}(F)$ to ones of $\Mp_{n-1}(F)$.

\begin{prop}\label{FJ-nonarch}
For any $\xi \in F^\times$, 
we have an isomorphism
\[
\FJ_{\psi_\xi}(I^{(n)}_{\psi}(\tau)) \cong I^{(n-1)}_{\psi}(\tau \chi_\xi), 
\]
where $\tau \chi_\xi = \tau \otimes \chi_\xi$. 
In particular, we have
\[
\left(I^{(n)}_{\psi}(\tau) \right)_{\psi_\xi} \cong I^{(n-1)}_{\psi}(\tau \chi_\xi) \otimes \omega_{\psi_\xi}
\]
as $\cl{J}_{n-1}(F)$-modules.
\end{prop}
\begin{proof}
See \cite[Theorem 3.1]{HS} and \cite{We}.
\end{proof}

\subsection{Archimedean case}
In this subsection, we consider the case where $F = \R$.
We assume that the non-trivial unitary character $\psi$ of $\R$ is of the form
\[
\psi(x) = \exp(2\pi a \I x)
\]
for $x \in \R$ with $a > 0$.
Then the Weil constant $\alpha_\psi(t)$ is given by
\[
\alpha_\psi(t) = \left\{
\begin{aligned}
&\exp(\pi\I/4) \iif t > 0, \\
&\exp(-\pi\I/4) \iif t < 0.
\end{aligned}
\right.
\]
For $\xi \in \R^\times$ with $\xi > 0$, we define $\phi_\xi^0 \in \Sc(X(\R))$ by 
\[
\phi_{\xi}^0(x) = \exp(-2\pi a \xi (x_1^2 + \dots + x_{n-1}^2))
\]
for $x = (x_1, \dots, x_{n-1}) \in X(\R)$.
Let
\[
K_\infty = \left\{
\left.
\begin{pmatrix}
\alpha & \beta \\ - \beta & \alpha
\end{pmatrix}
\right|
\alpha, \beta \in \Mat_n(\R),\ 
{}^t\alpha\beta = {}^t\beta\alpha,\ 
{}^t\alpha\alpha + {}^t\beta\beta = \1_n
\right\}
\]
be the usual maximal compact subgroup of $\Sp_n(\R)$.
For $u =  \begin{pmatrix}
\alpha & \beta \\ - \beta & \alpha
\end{pmatrix}\in K_\infty$, we write $\det(\alpha + \I\beta) = e^{\I \theta}$ with $- \pi < \theta \leq \pi$, 
and we set
\[
{\det}^{1/2}(u,\zeta) = \zeta e^{\I \theta / 2}.
\]
Then $\omega_{\psi_\xi}(u,\zeta) \phi_{\xi}^0 = {\det}^{1/2}(u,\zeta) \cdot \phi_{\xi}^0$
for any $(u,\zeta) \in \cl{K}_\infty \cap \Mp_{n-1}(\R)$.
In particular, ${\det}^{1/2}$ is a genuine character of $\cl{K}_\infty$.
For an integer $l$, we set ${\det}^{l/2}(u,\zeta) = ({\det}^{1/2}(u,\zeta))^l$.
We denote the irreducible lowest weight representation of $\Mp_n(\R)$
with lowest $\cl{K}_\infty$-type ${\det}^{l/2}$ by $\DD_{l/2}^{(n)}$.
\par

Let
\[
\Ha_{n} = \{Z \in \Mat_n(\C)\ |\ {}^tZ = Z,\ \im(Z) > 0\}
\]
be the Siegel upper half space of genus $n$.
Here, for a symmetric matrix $B$, we write $B > 0$ if $B$ is positive definite.
Then $\Sp_n(\R)$ acts on $\Ha_n$ by 
\[
g(Z) = (AZ+B)(CZ+D)^{-1}, 
\quad 
g = \begin{pmatrix}
A & B \\ C & D
\end{pmatrix} 
\in \Sp_n(\R).
\]
Note that the stabilizer of $\bi = \I \cdot \1_n \in \Ha_n$ in $\Sp_n(\R)$ is equal to $K_\infty$.
We set $j(g,Z) = \det(CZ+D)$.
Then there exists a unique automorphy factor $\cl{j}((g,\zeta), Z)$ of $\Mp_n(\R)$
such that $\cl{j}((g,\zeta), Z)^2 = j(g,Z)$ for any $(g,\zeta) \in \Mp_n(\R)$ and $Z \in \Ha_n$.
The following lemma might be well-known, but we give a proof for readers.

\begin{lem}
For $A \in \GL_n(\R)$ with $\det A > 0$, $B \in \Sym_n(\R)$, and $(u, \zeta) \in \cl{K}_\infty$, 
we have
\[
\cl{j}(\bn(B)(\bm(A), 1)(u,\zeta), \bi) = (\det A)^{-\half{1}} {\det}^{-1/2}(u,\zeta).
\]
\end{lem}
\begin{proof}
Since $\cl{j}((g,\zeta), Z)$ is an automorphy factor of $\Mp_n(\R)$, 
for any $(g,\zeta) \in \Mp_n(\R)$, the function 
\[
\Ha_n \rightarrow \C,\ Z \mapsto \cl{j}((g,\zeta), Z)
\]
is holomorphic.
\par

First, we claim that $\cl{j}(\bn(B), Z) = 1$ for any $B \in \Sym_n(\R)$ and $Z \in \Ha_n$.
Note that $\cl{j}(\bn(B), Z)^2 = j(\bn(B), Z) = 1$.
Since $\Ha_n$ is connected, $\cl{j}(\bn(B), Z)$ is independent of $Z$.
In particular, $B \mapsto \cl{j}(\bn(B), Z)$ gives a quadratic character of $\Sym_n(\R)$.
Since $\Sym_n(\R)$ is divisible, it must be the trivial character.
Hence $\cl{j}(\bn(B), Z) = 1$.
\par

Next, we claim that $\cl{j}((\bm(A), 1), Z) = (\det A)^{-1/2}$ 
for $A \in \GL_2(\R)$ with $\det A > 0$ and $Z \in \Ha_n$.
By a similar argument to the first case, 
the function $\Ha_n \ni Z \mapsto (\det A)^{1/2} \cl{j}((\bm(A), 1), Z) \in \{\pm1\}$ is a constant 
for any $A \in \GL_2(\R)$ with $\det A > 0$. 
In particular, the map $A \mapsto (\det A)^{1/2} \cl{j}((\bm(A), 1), Z) \in \{\pm1\}$ 
gives a group homomorphism.
Since $(e^{(1/2)X})^2 = e^{X}$ for $X \in \mathfrak{gl}_n(\R)$, 
we have $(\det A)^{1/2} \cl{j}((\bm(A), 1), Z) = 1$ when $A = e^X$ for some $X \in \mathfrak{gl}_n(\R)$.
Since any $A \in \GL_n(\R)$ with $\det A > 0$ can be written as a product $A = e^{X_1} \cdots e^{X_k}$ 
for some $X_1, \dots, X_k \in \mathfrak{gl}_n(\R)$, 
we have $(\det A)^{1/2} \cl{j}((\bm(A), 1), Z) = 1$ for any $A \in \GL_2(\R)$ with $\det A > 0$. 
\par

Similarly, $\cl{j}((\1_{2n},-1), Z) \in \{\pm1\}$ is independent of $Z \in \Ha_n$.
Choose $A \in \mathrm{O}(n)$ such that $A^2 = \1_n$ and $\det A = -1$.
Then $(\bm(A), 1)^2 = (\1_{2n}, -1)$ so that
$\cl{j}((\1_{2n},-1), \bi) = \cl{j}((\bm(A), 1), \bi)^2 = j(\bm(A), \bi) = \det(A)^{-1} = -1$.
Hence $\cl{j}((\1_{2n},-1), Z) = -1$ for any $Z \in \Ha_n$.

Finally, we see that ${\det}^{1/2}(u,\zeta)\cl{j}((u,\zeta), \bi) \in \{\pm1\}$
gives a group homomorphism $\cl{K}_\infty \rightarrow \{\pm1\}$, 
which factors through $K_\infty$.
Since $K_\infty$ is divisible, it must be the trivial character.
Hence $\cl{j}((u,\zeta), \bi) = {\det}^{-1/2}(u,\zeta)$.
This completes the proof.
\end{proof}
\par

Let $\Sym_n^+(\R) = \{B \in \Sym_n(\R)\ |\ B > 0\}$.
Fix an integer $k > 0$.
For $B \in \Sym_{n}^+(\R)$, we define a function $W_B^0$ on $\Mp_n(\R)$ by 
\begin{align*}
W_B^0(g) &= (\det B)^{(2k+n)/4}\psi(\Tr(Bg(\bi))) \cl{j}(g,\bi)^{-(2k+n)}
\\&= \psi(\Tr(Bz))\det(B[A])^{(2k+n)/4} \exp(-2\pi a \cdot \Tr(B[A])) {\det}^{k+(n/2)}(\cl{u})
\end{align*}
for $g = \bn(z)(\bm(A), 1)\cl{u} \in \Mp_n(\R)$ with $z \in \Sym_n(\R)$, $A \in \GL_n(\R)$, $\det A > 0$, 
and $\cl{u} \in \cl{K}_\infty$.
Here, we put $B[A] = {}^tABA$.
Then $W_B^0$ is a lowest weight vector in $\DD^{(n)}_{k+(n/2)}$.
\par

\begin{prop}[{\cite[Lemma 7.6]{IY}}]\label{FJ-arch}
For $\xi \in \R^\times$ with $\xi > 0$ and $B' \in \Sym_{n-1}^+(\R)$, 
we put
\[
B = \begin{pmatrix}
\xi & 0 \\ 0 & B'
\end{pmatrix} 
\in \Sym_n^+(\R).
\]
Then we have
\[
\int_{X(\R)} W_{B}^0(\bv(x,0,0)g') \overline{\omega_{\psi_\xi}(g')\phi_\xi^0(x)} dx
= |\det B|^{1/4} e^{-2\pi a \xi} W_{B'}^0(g')
\]
for $g' \in \Mp_{n-1}(\R)$.
\end{prop}

The Lie algebras $\sp_n(\R)$ and $\kk$ of $\Sp_n(\R)$ and $K_\infty$ are given by
\begin{align*}
\sp_n(\R) &= \{X \in \Mat_{2n}(\R)\ |\ 
{}^tX 
\begin{pmatrix}
0 & -\1_n \\ \1_n & 0 
\end{pmatrix}
+ 
{}^tX \begin{pmatrix}
0 & -\1_n \\ \1_n & 0 
\end{pmatrix}
X
=
0
\} 
\\&=
\left\{
\begin{pmatrix}
A & B \\ C & D
\end{pmatrix}
\in \Mat_{2n}(\R)
\ |\ 
B = {}^tB,\ C = {}^tC,\ A = -{}^tD
\right\},
\\
\kk &= \left\{
\begin{pmatrix}
A & B \\ -B & A
\end{pmatrix}
\in \Mat_{2n}(\R)\ |\ 
{}^tA = -A,\ {}^tB = B
\right\},
\end{align*}
respectively. 
Note that $\kk$ is the $1$-eigenspace of the Cartan involution $\theta X = -{}^tX$ on $\sp_n(\R)$.
The $(-1)$-eigenspace is given by
\[
\p = \left\{
\begin{pmatrix}
A & B \\ B & -A
\end{pmatrix}
\in \Mat_{2n}(\R)
\ |\ 
A = {}^tA,\ B = {}^tB
\right\}.
\]
Hence $\sp_n(\R) = \kk \oplus \p$.
The homeomorphism 
\[
\rho \colon \Sp_n(\R)/ K_\infty \rightarrow \Ha_n,\ g \mapsto g(\bi)
\]
induces an isomorphism
\[
d\rho \colon \sp_n(\R) /\kk = \p \xrightarrow{\sim} T_{\bi}\Ha_n \cong \Sym_n(\C).
\]
This map is given by
\[
d\rho(
\begin{pmatrix}
A & B \\ B & -A
\end{pmatrix}
)
= 2(B+\I A).
\]
Then the complex structure (i.e., multiplication with $\I$) on $\Sym_n(\C)$ gives 
a map
\[
J \colon \p \rightarrow \p,\ 
\begin{pmatrix}
A & B \\ B & -A
\end{pmatrix}
\mapsto
\begin{pmatrix}
B & -A \\ -A & -B
\end{pmatrix}.
\]
Let $\sp_n(\C)$, $\kk_\C$, and $\p_\C$ be the complexifications of 
$\sp_n(\R)$, $\kk$, and $\p$, respectively.
We denote the $(\pm\I)$-eigenspace of $J$ on $\p_\C$ by $\p_\C^\pm$.
Then 
\begin{align*}
\p_\C^\pm &=
\left\{
\begin{pmatrix}
A & B \\ B & -A
\end{pmatrix}
\otimes 1 
\pm
\begin{pmatrix}
-B & A \\ A & B
\end{pmatrix}
\otimes \I
\ |\ 
A, B \in \Sym_n(\R)
\right\}
\\&= \left\{
\begin{pmatrix}
A & \pm\I A \\ \pm\I A & -A
\end{pmatrix}
\in \Mat_{2n}(\C)
\ |\ A \in \Sym_n(\C)
\right\}.
\end{align*}
The elements of $\p_\C^+$ and $\p_\C^-$ correspond to the linear combinations of 
differential operators 
\[
\left\{\frac{d}{dz_j}\right\}_j
\quad\text{and}\quad
\left\{\frac{d}{d\overline{z_j}}\right\}_j, 
\]
respectively, 
where $z_j = x_j + \I y_j$ are coordinates on $\Ha_n$ at the point $\bi$.
\par

The element $X \in \p$ acts on a smooth function $W$ on $\Mp_n(\R)$ by 
\[
X \cdot W(g) = \left.\frac{d}{dt}\right|_{t=0} W(g (e^{tX},1))
\]
for $g \in \Mp_n(\R)$.
This action is extended to $\p_\C$ linearly.
It is easy to check that $\p_\C^-$ acts on $W_B^0$ by zero for any $B \in \Sym_n^+(\R)$
(cf. \cite[Lemma 7]{AS}).
Similarly, $\p_\C^- \cap \sp_{n-1}(\C)$ acts on $\phi_\xi^0$ by zero.
This fact can be proven by using the Fock model of $\omega_{\psi_\xi}$.

\subsection{Global case}\label{sec.global}
Now we let $F$ be a totally real number field, 
and $\psi$ be a non-trivial unitary character of $\A/F$.
We assume that for each infinite place $v$ of $F$, 
there exists $a_v \in F_v \cong \R$ with $a_v > 0$ 
such that $\psi(x_v) = \exp(2\pi a_v \I x_v)$ for $x_v \in F_v$.
For a place $v$ of $F$, we define a maximal compact subgroup $K_v$ of $\Sp_n(F_v)$ by
\[
K_v = 
\left\{
\begin{aligned}
&\Sp_n(\oo_v) \iif \text{$v$ is non-archimedean}, \\
&K_\infty \iif \text{$v$ is real}.
\end{aligned}
\right.
\]
Here, we denote by $\oo_v$ the ring of integers of $F_v$ when $v$ is non-archimedean.
If $v$ is archimedean, 
we define $(\p_v)_\C^- \subset \mathrm{Lie}(\Sp_n(F_v)) \otimes_\R \C$ as in the previous subsection.
\par

Recall that a function $\varphi \colon \Sp_n(F) \bs \Mp_n(\A) \rightarrow \C$ is cusp form if
\begin{itemize}
\item
$\varphi$ is smooth and of moderate growth; 
\item
$\varphi$ is right $\cl{K}$-finite, 
where $\cl{K} = \prod_{v} \cl{K}_v$; 
\item
$\varphi$ is $\mathfrak{z}$-finite, 
where $\mathfrak{z}$ is the center of the universal enveloping algebra of 
$\mathrm{Lie}(\Sp_n(F \otimes_\Q \R)) \otimes_\R \C$; 
\item
there exists $\delta \in \{0,1\}$ such that $\varphi((g,\zeta)) = \zeta^\delta \varphi((g,1))$
for any $g \in \Sp_n(\A)$; 
\item
For any proper $F$-parabolic subgroup $P$ of $\Sp_n$, 
the constant term along $P$
\[
\int_{N(F) \bs N(\A)} \varphi(ug) du
\]
is zero for any $g \in \Mp_n(\A)$, 
where $N$ is the unipotent radical of $P$.
\end{itemize}
We say that $\varphi$ is genuine if $\delta = 1$.
Let $l = (l_v)_v \in \prod_{v \mid \infty}\Z$.
We say that a cusp form $\varphi$ is holomorphic of weight $l/2$ 
if 
\begin{itemize}
\item
$X_v \cdot \varphi = 0$ for any $X_v \in (\p_v)_\C^-$; 
\item
$\varphi(g\cl{u}_v) = {\det}^{l_v/2}(\cl{u}_v) \varphi(g)$ for $g \in \Mp_n(\A)$ and $\cl{u}_v \in \cl{K}_v$
\end{itemize}
for any infinite place $v$.
We denote the space of holomorphic cusp forms of weight $l/2$ by 
\[
\Sc_{l/2}(\Sp_n(F) \bs \Mp_n(\A)).
\]
The group $\Mp_n(\A_\fin)$ acts on $\Sc_{l/2}(\Sp_n(F) \bs \Mp_n(\A))$ by the right translation.
\par

Let $\tau = \otimes_v' \tau_v$ be an irreducible unitary cuspidal automorphic representation of $\GL_2(\A)$.
Assume that
\begin{itemize}
\item[(A1)]
for any finite place $v$, 
the local factor $\tau_v$ is 
an irreducible principal representation $\mu_v \times \mu_v^{-1}$; 
\item[(A2)]
for any infinite place $v$, 
the local factor $\tau_v$ is a discrete series representation with lowest weight $\pm 2 k_v$
where $k_v > 0$; 
\item[(A3)]
the root numer 
\[
\ep(\tau) = \ep(\half{1}, \tau, \psi) = \prod_{v < \infty}\mu_v(-1) \cdot \prod_{v \mid \infty} (-1)^{k_v}
\]
is equal to $1$.
\end{itemize}
For $v < \infty$, we notice that $|\mu_v(\varpi_v)|_v =1$ 
by the Ramanujan conjecture proven in the case of Hilbert modular forms by Blasius \cite{Bl}, 
where $\varpi_v$ is a uniformizer of $F_v$.
For a more general result on the Ramanujan conjecture, see \cite{C}.
We refer to $\pm 2k = (\pm2k_v)_v \in \prod_{v \mid \infty}\Z$ 
as the weight of $\otimes_{v \mid \infty} \tau_v$.
\par

Put $k+(n/2) = (k_v + n/2)_v \in \prod_{v \mid \infty}\Z$ for each integer $n > 0$.
For $v < \infty$, 
we let $I^{(n)}_{\psi_v}(\tau_v) = \Ind_{\cl{P}_n(F_v)}^{\Mp_n(F_v)}(\mu_v^{(n)})$ 
be the degenerate principal series defined in \S \ref{sec.non-arch}.
Set $I^{(n)}_{\psi}(\tau) = \otimes'_{v < \infty} I^{(n)}_{\psi_v}(\tau_v)$, 
which is an irreducible representation of $\Mp_n(\A_\fin)$.
The following is a part of the main theorem in \cite{IY}.

\begin{thm}[{\cite[Theorems 1.1, 1.2]{IY}}]
Let $\tau = \otimes_v' \tau_v$ be an irreducible unitary cuspidal automorphic representation of $\GL_2(\A)$
satisfying the conditions (A1), (A2) and (A3).
Then the representation $I^{(n)}_{\psi}(\tau)$ 
occurs in $\Sc_{k+(n/2)}(\Sp_n(F) \bs \Mp_n(\A))$ with multiplicity one.
\end{thm}

We denote the unique subrepresentation of $\Sc_{k+(n/2)}(\Sp_n(F) \bs \Mp_n(\A))$ 
which is isomorphic to $I^{(n)}_{\psi}(\tau)$ by $\Ik^{(n)}_{\psi}(\tau)$, 
and call it the Ikeda lift of $\tau$.
\par

Let $\xi \in F^\times$ be a totally positive element.
For $\phi \in \Sc(X(\A))$, 
the theta function $\Theta_{\psi_\xi}^{\phi}(vg')$ is defined by
\[
\Theta_{\psi_\xi}^{\phi}(\bv(x,y,z) g') 
= \sum_{t \in X(F)} \psi_\xi(z + 2t \cdot {}^ty + x \cdot {}^ty) \omega_{\psi_\xi}(g')\phi(t + x)
\]
for $\bv(x,y,z) \in V(\A)$ and $g' \in \Mp_{n-1}(\A)$.
It is a genuine automorphic form on $\cl{J}_{n-1}(\A)$.
One can easily check that
\[
\int_{Z(\A)V(F) \bs V(\A)} \Theta_{\psi_\xi}^{\phi_1}(vg')\overline{\Theta_{\psi_\xi}^{\phi_2}(vg')} dv 
= (\phi_1, \phi_2) = \int_{X(\A)} \phi_1(x) \overline{\phi_2(x)} dx
\]
for $\phi_1, \phi_2 \in \Sc(X(\A))$ and $g' \in \Mp_{n-1}(\A)$.
\par

We denote by $\Sc(X(\A))_\xi$ the subspace of $\Sc(X(\A))$ spanned by $\phi = \otimes_v\phi_v$ 
such that $\phi_v = \phi_{\xi_v}^0$ for each infinite place $v$.
For $\varphi \in \Sc_{l/2}(\Sp_n(F) \bs \Mp_n(\A))$ and $\phi \in \Sc(X(\A))_\xi$, 
the Fourier--Jacobi coefficient associated to $(\varphi, \phi)$ is defined by
\[
\FJ_{\psi_\xi}^\phi(g'; \varphi) 
= \int_{V(F) \bs V(\A)} \varphi(vg') \overline{\Theta_{\psi_\xi}^{\phi}(vg')}dv.
\]
This is a cusp form on $\Mp_{n-1}(\A)$.
(See also \cite[Lemma 2.3]{GRS2} and \cite[Theorem 8]{GRS1}.)
Moreover, by Proposition \ref{FJ-arch}, we conclude that 
\[
\FJ_{\psi_\xi}^\phi(\varphi) \in \Sc_{(l-1)/2}(\Sp_{n-1}(F) \bs \Mp_{n-1}(\A))
\]
for any $\varphi \in \Sc_{l/2}(\Sp_n(F) \bs \Mp_n(\A))$ and $\phi \in \Sc(X(\A))_\xi$.

\begin{prop}\label{FJcoeff1}
For $\FF \in \Ik^{(n)}_\psi(\tau)$ and $\phi \in \Sc(X(\A))_\xi$, we have
$\FJ_{\psi_\xi}^\phi(\FF) \in \Ik^{(n-1)}_\psi(\tau \chi_\xi)$.
\end{prop}
\begin{proof}
We may assume that $\FJ_{\psi_\xi}^\phi(\FF) \not= 0$.
Let $\Pi$ be the representation of $\Mp_{n-1}(\A)$ generated by $\FJ_{\psi_\xi}^\phi(\FF)$
for $\FF \in \Ik^{(n)}_\psi(\tau)$ and $\phi \in \Sc(X(\A))_\xi$.
Since $\Pi$ is cuspidal, it is a direct sum of irreducible representations.
Choose an irreducible direct summand $\pi$ of $\Pi$ and a projection $\Pi \twoheadrightarrow \pi$.
Then the map 
\[
\Ik^{(n)}_\psi(\tau) \otimes \overline{\omega_{\psi_\xi}} \xrightarrow{\FJ_{\psi_\xi}} \Pi \twoheadrightarrow \pi
\]
is a (nonzero) $V(\A_\fin)$-invariant map so that it factors through $\FJ_{\psi_\xi}(\Ik^{(n)}_\psi(\tau))$.
Since $\FJ_{\psi_\xi}(\Ik^{(n)}_\psi(\tau))$ is nonzero, 
we have $\FJ_{\psi_\xi}(\Ik^{(n)}_\psi(\tau)) = \Ik^{(n-1)}_\psi(\tau \chi_\xi)$ 
by Proposition \ref{FJ-nonarch}.
Since it is irreducible, we have $\pi \cong \Ik^{(n-1)}_\psi(\tau \chi_\xi)$.
Hence $\Pi$ is isomorphic to a direct sum of some copies of $\Ik^{(n-1)}_\psi(\tau \chi_\xi)$.
However the Ikeda lift appears in $\Sc_{k+(n-1)/2}(\Sp_{n-1}(F) \bs \Mp_{n-1}(\A))$ with multiplicity one, 
we see that $\Pi$ is irreducible and is equal to $\Ik^{(n-1)}_\psi(\tau \chi_\xi)$.
\end{proof}

For $\varphi \in \Sc_{l/2}(\Sp_n(F) \bs \Mp_n(\A))$, 
we define the $\xi$-th Fourier--Jacobi coefficient $\varphi_{\psi_\xi}$ by 
\[
\varphi_{\psi_\xi}(vg') = \int_{Z(F) \bs Z(\A)}\varphi(zvg') \overline{\psi_{\xi}(z)}dz
\]
for $v \in V(\A)$ and $g' \in \Mp_{n-1}(\A)$.
If $\Pi$ be a subspace of $\Sc_{l/2}(\Sp_n(F) \bs \Mp_n(\A))$, we put
\[
\Pi_{\psi_\xi} = \left\{ \varphi_{\psi_\xi}\ |\ \varphi \in \Pi \right\}.
\]
This is the maximal quotient of $\Pi$ on which $Z(\A_\fin)$ acts by $\psi_\xi$.

\begin{prop}\label{FJcoeff2} 
\begin{enumerate}
\item
For $\FF \in \Ik^{(n)}_\psi(\tau)$, 
there exist $\FF'_1, \dots, \FF'_r \in \Ik^{(n-1)}_\psi(\tau \chi_\xi)$ 
and $\phi_1, \dots, \phi_r \in \Sc(X(\A))_\xi$ such that 
\[
\FF_{\psi_\xi}(vg') = \sum_{i=1}^{r} \FF'_i(g')\Theta_{\psi_\xi}^{\phi_i}(vg')
\]
for $v \in V(\A)$ and $g' \in \Mp_{n-1}(\A)$. 

\item
Suppose that $(\Ik^{(n)}_\psi(\tau))_{\psi_\xi}$ is nonzero.
Then for $\FF' \in \Ik^{(n-1)}_\psi(\tau \chi_\xi)$ and $\phi \in \Sc(X(\A))_\xi$, 
there exists $\FF \in \Ik^{(n)}_\psi(\tau)$ such that 
\[
\FF_{\psi_\xi}(vg') = \FF'(g')\Theta_{\psi_\xi}^{\phi}(vg')
\]
for $v \in V(\A)$ and $g' \in \Mp_{n-1}(\A)$.

\item
If $n \geq 2$, then $(\Ik^{(n)}_\psi(\tau))_{\psi_\xi} \not= 0$ for any totally positive $\xi \in F^\times$.
\end{enumerate}
\end{prop}
\begin{proof}
By \cite[Proposition 1.3]{I0}, 
for $\FF \in \Ik^{(n)}_\psi(\tau)$, 
there exist $\phi_1, \dots, \phi_r \in \Sc(X(\A))_\xi$ with $(\phi_i, \phi_j) = \delta_{i,j}$ such that 
\[
\FF_{\psi_\xi}(vg') = \sum_{i=1}^{r} \FJ_{\psi_\xi}^{\phi_i}(\FF)(g') \cdot \Theta_{\psi_\xi}^{\phi_i}(vg')
\]
for $v \in V(\A)$ and $g' \in \Mp_{n-1}(\A)$.
Since $\FJ_{\psi_\xi}^{\phi_i}(\FF)(g') \in \Ik^{(n-1)}_\psi(\tau \chi_\xi)$ 
by Proposition \ref{FJcoeff1}, we obtain (1).
\par

Hence 
\[
(\Ik^{(n)}_\psi(\tau))_{\psi_\xi} \subset \Ik^{(n-1)}_\psi(\tau \chi_\xi) \otimes \omega_{\psi_\xi}
\]
as $\cl{J}_{n-1}(\A_\fin)$-modules.
If $(\Ik^{(n)}_\psi(\tau))_{\psi_\xi}$ is nonzero, this inclusion must be equal 
since $\Ik^{(n-1)}_\psi(\tau \chi_\xi) \otimes \omega_{\psi_\xi}$ is irreducible.
Hence we obtain (2).
\par

Let $\Sym_n^+(F)$ be the subset of $\Sym_n(F)$ consisting of symmetric matrices 
whose images in $\Sym_n(F_v)$ are positive definite for each infinite place $v$.
By \cite[Lemmas 8.5, 5.4]{IY}, there exists $\FF \in \Ik^{(n)}_{\psi}(\tau)$ such that
for $B \in \Sym_n^+(F)$, 
the $B$-th Fourier coefficient
\[
W_B(g) = \int_{\Sym_n(F) \bs \Sym_n(\A)} \FF(\bn(X)g) \overline{\psi(\Tr(BX))} dX
\]
is not identically zero if and only if $L(1/2, \tau \otimes \chi_{\det(B)}) \not= 0$.
When $B = \mathrm{diag}(\xi, \xi',1, \dots, 1)$ with $\xi' \in F^\times$ totally positive, 
the $B$-th Fourier coefficient $W_B$ is given by
\[
W_B(g) = \int_{Z(\A)\Sym_n(F) \bs \Sym_n(\A)}\FF_{\psi_\xi}(\bn(X)g) \overline{\psi(\Tr(BX))} dX.
\]
Hence the $\xi$-th Fourier--Jacobi coefficient $\FF_{\psi_\xi}$ is not identically zero
if $W_B \not= 0$.
Since there exists a totally positive element $\xi' \in F^\times$ such that
$L(1/2, \tau \otimes \chi_{\xi\xi'}) \not= 0$ by \cite[Th{\'e}or{\`e}m 4]{W}, 
we obtain (3).
\end{proof}

\section{Local Miyawaki liftings}\label{local}
In this section, we define the local Miyawaki lifting, and prove basic properties.

\subsection{Definition}
Let $F$ be a non-archimedean local field of characteristic zero, 
and $\psi$ be a non-trivial additive character of $F$.
For two non-negative integers $n$ and $r$, 
we define an embedding $\iota \colon \Sp_n \times \Sp_r \hookrightarrow \Sp_{n+r}$ by
\[
\iota\left(
\begin{pmatrix}
A_1 & B_1 \\ C_1 & D_1
\end{pmatrix}, 
\begin{pmatrix}
A_2 & B_2 \\ C_2 & D_2
\end{pmatrix}
\right)
=
\left(
\begin{array}{cc|cc}
A_1 & 0 & B_1 & 0 \\
0 & A_2 & 0 & B_2 \\
\hline
C_1 & 0 & D_1 & 0 \\
0 & C_2 & 0 & D_2 
\end{array}
\right), 
\]
and we identify $\Sp_n \times \Sp_r$ with the image.
For $\tau = \mu \times \mu^{-1}$ with $\mu$ being a unitary character of $F^\times$, 
we consider the local Ikeda lift $I^{(n+r)}_{\psi}(\tau) = \Ind_{\cl{P}_{n+r}(F)}^{\Mp_{n+r}(F)}(\mu^{(n+r)})$, 
which is a degenerate principal series of $\Mp_{n+r}(F)$.
For an irreducible representation $\pi$ of $\Mp_{r}(F)$, 
on which the kernel $\{\pm1\}$ of the covering map $\Mp_{r}(F) \twoheadrightarrow \Sp_{r}(F)$
acts by $(\pm1)^{n+r}$, 
the maximal $\pi$-isotypic quotient of $I^{(n+r)}_\psi(\tau)$ is of the form
\[
\MM_{\psi, \tau}^{(n)}(\pi) \boxtimes \pi
\]
for some smooth representation $\MM_{\psi, \tau}^{(n)}(\pi)$ of $\Mp_{n}(F)$, 
on which the kernel $\{\pm1\}$ of the covering map $\Mp_{n}(F) \twoheadrightarrow \Sp_{n}(F)$
acts by $(\pm1)^{n+r}$.
We call $\MM_{\psi, \tau}^{(n)}(\pi)$ the local Miyawaki lift of $\pi$.

\subsection{Preliminary}\label{pre}
In this subsection, we recall some basic terminologies of representations of $\Mp_{r}(F)$.
\par

A parabolic induction
\[
\tau_1|\cdot|^{s_1} \times \dots \times \tau_l|\cdot|^{s_l} \rtimes \pi_0
=
\Ind_{P(F)}^{\Sp_{r}(F)}(\tau_1|\cdot|^{s_1} \otimes \dots \otimes \tau_l|\cdot|^{s_l} \otimes \pi_0), 
\]
is standard if 
\begin{itemize}
\item
$P=MN$ is a standard parabolic subgroup of $\Sp_r$
with $M \cong \GL_{k_1} \times \dots \times \GL_{k_l} \times \Sp_{r_0}$; 
\item
$\tau_i$ (\resp $\pi_0$) is an irreducible tempered representation of $\GL_{k_i}(F)$ (\resp $\Sp_{r_0}(F)$); 
\item
$s_1, \dots, s_l$ are real numbers such that $s_1 > \dots > s_l > 0$.
\end{itemize}
The Langlands classification asserts that 
any irreducible smooth representation $\pi$ of $\Sp_r(F)$ is a unique irreducible quotient of 
a standard module $\tau_1|\cdot|^{s_1} \times \dots \times \tau_l|\cdot|^{s_l} \rtimes \pi_0$, 
which is called the Langlands quotient and is denoted 
by $J(\tau_1|\cdot|^{s_1}, \dots, \tau_l|\cdot|^{s_l}, \pi_0)$.
For any irreducible representation $\pi$, 
the datum $(P, \{\tau_i |\cdot|^{s_i}\}, \pi_0)$ is determined uniquely up to an isomorphism.
\par

When $\pi$ is a genuine irreducible representation of $\Mp_{r}$, 
one should replace the parabolic induction with
\[
\Ind_{\cl{P}(F)}^{\Mp_{r}(F)}
(\tau_{1,\psi}|\cdot|^{s_1} \otimes \dots \otimes \tau_{l, \psi}|\cdot|^{s_l} \otimes \pi_0), 
\]
where 
\[
\tau_{i, \psi}(a, \zeta) = \zeta \frac{\alpha_\psi(1)}{\alpha_\psi(\det a)} \tau_i(a)
\]
is a genuine irreducible representation of the double cover $\cl{\GL}_{k_i}(F)$ of $\GL_{k_i}(F)$
(which is identified with $\GL_{k_i}(F) \times \{\pm1\}$ as sets), 
and $\pi_0$ is a genuine irreducible representation of $\Mp_{r_0}(F)$.
For more precision, see \cite{GS}.
We identify $\tau_{\psi}$ with $\tau$ itself, 
and we use the same notation $\tau_1|\cdot|^{s_1} \times \dots \times \tau_l|\cdot|^{s_l} \rtimes \pi_0$
and $J(\tau_1|\cdot|^{s_1}, \dots, \tau_l|\cdot|^{s_l}, \pi_0)$ as in the non-genuine case.
\par

We say that an irreducible representation $\pi$ of $\Mp_{r}$ is almost tempered 
if $\pi = J(\tau_1|\cdot|^{s_1}, \dots, \tau_l|\cdot|^{s_l}, \pi_0)$ with $0 < s_l < \dots < s_1 < 1/2$.
Then the standard module $\tau_1|\cdot|^{s_1} \times \dots \times \tau_l|\cdot|^{s_l} \rtimes \pi_0$ 
is irreducible by Corollary \ref{at}, so that $\pi$ is equal to this standard module. 
\par

When we consider unramified representations, 
we always assume that the residue characteristic of $F$ is greater than $2$.
Then the covering map $\Mp_r(F) \rightarrow \Sp_r(F)$ splits over 
the maximal compact subgroup $\Sp_r(\oo)$, 
where $\oo$ is the ring of integers of $F$.
Recall that an irreducible smooth representation $\pi$ of $\Mp_r(F)$ is unramified 
if $\pi$ has a nonzero $\Sp_r(\oo)$-fixed vector.
If $\pi$ is unramified, then there exist unramified characters $\chi_1, \dots, \chi_r$ such that
$\pi$ is a unique unramified constituent of the induced representation 
$\Ind_{\cl{B}_r(F)}^{\Mp_r(F)}(\chi_1 \otimes \dots \otimes \chi_r)$.
We call the multiset $\{\chi_1(\varpi)^{\pm1}, \dots, \chi_r(\varpi)^{\pm1}\}$
the Satake parameter of $\pi$, 
where $\varpi$ is a uniformizer of $F$.
\par

Recall that we can associate an irreducible representation $\pi$ of $\Mp_r(F)$ to an $L$-parameter $\phi$, 
which is a self-dual representation of the Weil--Deligne group $\WD_F = W_F \times \SL_2(\C)$ of $F$.
This is symplectic (\resp orthogonal) if $\pi$ is genuine (\resp not genuine).
More precisely, see Appendix \ref{langlands}.
We say that an $L$-parameter $\phi$ is of good parity 
if $\phi$ is a sum of irreducible self-dual representations of the same type as $\phi$.
The unique irreducible algebraic representation of $\SL_2(\C)$ of dimension $d$ 
is denoted by $S_d$.

\subsection{Local main theorem}
Recall that $\tau = \mu \times \mu^{-1}$ with $\mu$ being a unitary character of $F^\times$.
Let $\chi_{-1}$ be the quadratic character associated to the extension $F(\sqrt{-1})/F$.
For a real number $a$, we denote by $[a]$ the maximal integer which is not greater than $a$.
\par

The following is the local main theorem.
\begin{thm}\label{howe}
Let $\mu$ be a unitary character of $F^\times$, and 
$\pi$ be an irreducible representation of $\Mp_{r}(F)$ on which $\{\pm1\}$ acts by $(\pm1)^{n+r}$.
Suppose that $n \geq r$.

\begin{enumerate}
\item
The local Miyawaki lift $\MM_{\psi, \tau}^{(n)}(\pi)$ is nonzero and of finite length. 

\item
If $\pi$ is almost tempered and unitary, then 
\begin{itemize}
\item
$\MM_{\psi, \tau}^{(n)}(\pi)$ is irreducible; 
\item
$\MM_{\psi, \tau}^{(n)}(\pi) \cong (\mu' \circ {\det}_{n-r}) \rtimes \pi$ with $\mu' = \mu\chi_{-1}^{[(n+r)/2]}$; 
\item
$\MM_{\psi, \tau}^{(n)}(\pi)$ is isomorphic to the unique irreducible quotient of the induced representation
\[
\left\{
\begin{aligned}
&\tau'|\cdot|^{\half{n-r-1}} \times \tau'|\cdot|^{\half{n-r-3}} \times \dots \times \tau'|\cdot|^{\half{1}} \rtimes \pi
\iif n+r \equiv 0 \bmod 2, \\
&\tau'|\cdot|^{\half{n-r-1}} \times \tau'|\cdot|^{\half{n-r-3}} \times \dots \times \tau'|\cdot|^{1} 
\times \mu' \rtimes \pi
\iif n+r \equiv 1 \bmod 2, 
\end{aligned}
\right.
\]
where $\tau' = \tau \otimes \chi_{-1}^{[(n+r)/2]} = \mu' \times \mu'^{-1}$.
\end{itemize}

\item
For any irreducible almost tempered unitary representations $\pi_1$ and $\pi_2$, 
we have
\[
\MM_{\psi, \tau}^{(n)}(\pi_1) \cong \MM_{\psi, \tau}^{(n)}(\pi_2)
\implies 
\pi_1 \cong \pi_2. 
\]

\item
Set
\[
\pi' = \MM_{\psi, \tau}^{(r)}\left( \MM_{\psi, \tau}^{(n)}(\pi) \right).
\]
Assume that $\pi$ is almost tempered and unitary, 
and that one of the following conditions holds:
\begin{itemize}
\item
The $L$-parameter $\phi$ does not contain $\mu^{\pm1} S_d$ 
for any $d \geq n-r$ with $d \equiv n-r \bmod2$;

\item
$n=r$ or $n=r+1$.
\end{itemize}
Then all irreducible subquotients of $\pi'$ are isomorphic to $\pi$, 
and the maximal semisimple quotient of $\pi'$ is irreducible.

\item
Suppose that $\mu$ is unramified, 
and set $\alpha = (\mu\chi_{-1}^{[(n+r)/2]})(\varpi)$.
If $\pi$ is an irreducible unramified representation of $\Mp_r(F)$ 
with the Satake parameter $\{\beta_1^{\pm1}, \dots, \beta_r^{\pm1}\}$, 
then $\MM_{\psi, \tau}^{(n)}(\pi)$ has a unique irreducible unramified quotient. 
Its Satake parameter is equal to
\[
\{\beta_1^{\pm1}, \dots, \beta_r^{\pm1}\} \cup 
\{\alpha^{\pm1}q^{\half{n-r-1}}, \alpha^{\pm1}q^{\half{n-r-3}}, \dots, \alpha^{\pm1}q^{-\half{n-r-1}}\}
\]
as multisets.

\end{enumerate}
\end{thm}

As in Remark \ref{rem} (1) below, the first condition of Theorem \ref{howe} (4) holds 
when $n > 2r$ or $\pi$ is discrete series.
\par

The assertion (2) gives the Langlands data for $\MM_{\psi, \tau}^{(n)}(\pi)$ explicitly.
In particular, it deduces (3).
The proof of (5) for the non-genuine case (i.e., the case where $n+r$ is even) is \cite[Proposition 3.1]{I2}.
The genuine case is proven similarly.
\par

We prove Theorem \ref{howe} (1), (2) and (4) in \S \ref{(1)}, \S \ref{(2)} and \S \ref{(4)}, respectively.

\subsection{Miyawaki liftings and degenerate induced representations}
In this subsection, 
we show that $\MM_{\psi, \tau}^{(n)}(\pi)$ is of finite length, 
and that for almost tempered $\pi$, 
there is a surjection
\[
(\mu' \circ {\det}_{n-r}) \rtimes \pi \twoheadrightarrow \MM_{\psi, \tau}^{(n)}(\pi).
\]
\par

We need the following lemma (see \cite{KR2} and \cite[Lemma 2.2]{GT}).
\begin{lem}\label{filt}
Suppose that $n \geq r$. 
Set $\mu' = \mu \chi_{-1}^{[(n+r)/2]}$.
The local Ikeda lift (the degenerate principal series) 
$I^{(n+r)}_\psi(\tau) = \Ind_{\cl{P}_{n+r}(F)}^{\Mp_{n+r}(F)}(\mu^{(n+r)})$
has an $\Mp_{n}(F) \times \Mp_{r}(F)$-equivalent filtration 
\[
0 \subset I_0 \subset I_1 \subset \dots \subset I_{r} = I^{(n+r)}_\psi(\tau)
\]
with successive quotients
\begin{align*}
R_t = I_{t}/I_{t-1}
= \Ind_{\cl{P}_{t+(n-r)}(F) \times \cl{P}_{t}(F)}^{\Mp_{n}(F) \times \Mp_{r}(F)}
\left(
(\mu'|{\det}_{t+(n-r)}|^{\half{t}} \boxtimes \mu'|{\det}_{t}|^{\half{t+(n-r)}})
\otimes C_c^\infty(\Mp_{r-t}(F))
\right).
\end{align*}
Here the induction is normalized, 
${\det}_{t}$ denotes the determinant character of $\GL_{t}(F)$, 
and $\Mp_{r-t}(F) \times \Mp_{r-t}(F)$ acts on $C_c^\infty(\Mp_{r-t}(F))$ by 
\[
((g_1, g_2)\varphi)(x) = \varphi(g_1^{-1} \cdot x \cdot \alpha g_2 \alpha^{-1})
\quad\text{with}\quad
\alpha = \begin{pmatrix}
\1_{r-t} & 0 \\ 0 & -\1_{r-t}
\end{pmatrix}.
\]
In particular, 
\[
R_0 = \Ind_{\cl{P}_{n-r}(F) \times \Mp_r(F)}^{\Mp_{n}(F) \times \Mp_r(F)}
( \mu' \circ {\det}_{n-r} \otimes C_c^\infty(\Mp_r(F))).
\]
\end{lem}

Using this lemma, we have the following.
\begin{prop}\label{det}
Let $\pi$ and $\pi'$ be smooth representations of $\Mp_r(F)$ and $\Mp_n(F)$, respectively.
Set $\mu' = \mu \chi_{-1}^{[(n+r)/2]}$.
Suppose that $\pi$ is irreducible, almost tempered and that
\[
\Hom_{\Mp_{n}(F) \times \Mp_r(F)}(I_{\psi}^{(n+r)}(\tau), \pi' \boxtimes \pi) \not= 0.
\]
Then we have
\[
\Hom_{\Mp_n(F)} \left((\mu' \circ {\det}_{n-r}) \rtimes \pi, \pi' \right) \not= 0.
\]
Moreover, if there is an $\Mp_{n}(F) \times \Mp_r(F)$-equivalent surjection
\[
I_{\psi}^{(n+r)}(\tau) \twoheadrightarrow \pi' \boxtimes \pi, 
\]
then there is an $\Mp_{n}(F)$-equivalent surjection
\[
(\mu' \circ {\det}_{n-r}) \rtimes \pi \twoheadrightarrow \pi'.
\]
\end{prop}
\begin{proof}
By Lemma \ref{filt}, for some $0 \leq t \leq r$, we must have 
$\Hom_{\Mp_{n}(F) \times \Mp_r(F)}(R_t, \pi' \boxtimes \pi) \not= 0$.
By Bernstein's Frobenius reciprocity, this $\Hom$-space is isomorphic to
the space of $(\cl{M}_{t+(n-r)}(F) \times \cl{M}_{t}(F))$-equivalent maps
\begin{align*}
&(\mu'|{\det}_{t+(n-r)}|^{\half{t}} \boxtimes \mu'|{\det}_{t}|^{\half{t+(n-r)}})
\otimes C_c^\infty(\Mp_{r-t}(F))
\\&\rightarrow
R_{\overline{P_{t+(n-r)}(F)}}( (\mu' \circ {\det}_{n-r}) \rtimes \pi ) \boxtimes R_{\overline{P_{t}(F)}}(\pi'), 
\end{align*}
where $R_{\overline{P_{t}(F)}}(\pi)$ is the normalized Jacquet module of $\pi$
with respect to (the double cover of) the opposite parabolic subgroup $\overline{P_{t}(F)}$ to $P_t(F)$.
First, we assume that $t > 0$.
By taking the contragredient, we have 
\[
\Hom_{\cl{\GL}_t(F)}(R_{P_t(F)}(\pi^\vee), \mu'^{-1}|{\det}_{t}|^{\half{-t-(n-r)}}) \not= 0.
\]
Since $\pi^\vee$ is almost tempered, we must have $-t-(n-r) > -1$. 
This contradicts $n \geq r$ and $t > 0$.
Hence we must have $t = 0$, and 
\[
\Hom_{\cl{M}_{n-r}(F) \times \Mp_r(F)}\left(
\mu' \circ {\det}_{n-r} \otimes C_c^\infty(\Mp_{r}(F)), 
R_{\overline{P_{n-r}(F)}}(\pi') \boxtimes \pi
\right) \not= 0.
\]
Taking care of the action of $\Mp_r(F) \times \Mp_r(F)$ on $C_c^\infty(\Mp_{r}(F))$, 
we see that this $\Hom$-space is isomorphic to
\[
\Hom_{\cl{M}_{n-r}(F)}\left(
\mu' \circ {\det}_{n-r} \otimes \pi, R_{\overline{P_{n-r}(F)}}(\pi')
\right).
\]
Using Bernstein's Frobenius reciprocity, we obtain the first assertion.
\par

Suppose that there is an $\Mp_{n}(F) \times \Mp_r(F)$-equivalent surjection 
$I_r \twoheadrightarrow \pi' \boxtimes \pi$.
Then the image of $I_0$ is of the form $\pi'_0 \boxtimes \pi$ 
for some $\Mp_{n}(F)$-subrepresentation $\pi'_0$ of $\pi'$, 
and it induces a surjection $I_r/I_0 \twoheadrightarrow (\pi'/\pi'_0) \boxtimes \pi$.
The above argument implies $(\pi'/\pi'_0) = 0$ 
so that the restriction gives an $\Mp_{n}(F) \times \Mp_r(F)$-equivalent surjection 
$I_0 \twoheadrightarrow \pi' \boxtimes \pi$.
This induces an $\Mp_{n}(F)$-equivalent surjection
$(\mu' \circ {\det}_{n-r}) \rtimes \pi \twoheadrightarrow \pi'$.
\end{proof}

\begin{rem}\label{length}
Let $\pi$ and $\pi'$ be smooth representations of $\Mp_r(F)$ and $\Mp_n(F)$, respectively.
Suppose that $\pi$ is irreducible, and that there exists a surjection
\[
I_{\psi}^{(n+r)}(\tau) \twoheadrightarrow \pi' \boxtimes \pi.
\]
Then Lemma \ref{filt} gives a filtration 
\[
0 \subset \pi'_0 \subset \pi'_1 \subset \dots \subset \pi'_r = \pi'.
\]
More precisely, the restriction of the surjection to $I_t \subset I_{\psi}^{(n+r)}(\tau)$ 
defines a subrepresentation $\pi'_t$ of $\pi'$ so that the image of $I_t$ is $\pi'_t \boxtimes \pi$.
By a similar argument to Proposition \ref{det}, 
one can show that the successive quotients $\pi'_t/\pi'_{t-1}$ are of finite length.
Therefore $\pi'$ must also be of finite length.
\end{rem}

\subsection{Non-vanishing of local Miyawaki liftings}\label{(1)}
Next, we show Theorem \ref{howe} (1).
By Remark \ref{length}, we see that $\MM_{\psi, \tau}^{(n)}(\pi)$ is of finite length.
We show that $\MM_{\psi, \tau}^{(n)}(\pi) \not= 0$. 
In fact, we will prove that 
\[
\Hom_{\Mp_n(F) \times \Mp_r(F)}\left( 
I_{\psi}^{(n+r)}(\tau), ((\mu' \circ {\det}_{n-r}) \rtimes \pi) \boxtimes \pi 
\right) \not= 0
\]
for any irreducible representation $\pi$ of $\Mp_r(F)$.
\par

Using 
\[
\Sp_{2r}(F) \hookrightarrow \Sp_{n+r}(F),\ 
\begin{pmatrix}
A & B \\ C & D
\end{pmatrix}
\mapsto 
\left(
\begin{array}{cc|cc}
\1_{n-r} &&&\\
&A&&B\\
\hline
&&\1_{n-r}&\\
&C&&D
\end{array}
\right), 
\]
we regard $\Mp_{2r}(F)$ as a subgroup of $\Mp_{n+r}(F)$.
Then for $\Phi \in I_{\psi}^{(n+r)}(\tau)$, we have
\[
\Phi|_{\Mp_{2r}(F)} \in \Ind_{\cl{P}_{2r}(F)}^{\Mp_{2r}(F)}(\mu \chi_{-1}^{[\half{n+r}]} |{\det}_{2r}|^{\half{n-r}}).
\]
By the theory of the doubling method \cite{LR}, \cite{PSR}, 
a doubling zeta integral gives a nonzero element
\[
\ZZ_{r,r} \in 
\Hom_{\Mp_r(F) \times \Mp_r(F)}(
\Ind_{\cl{P}_{2r}(F)}^{\Mp_{2r}(F)}(\mu \chi_{-1}^{[\half{n+r}]} |{\det}_{2r}|^{\half{n-r}}) 
\otimes \left(\pi \boxtimes \pi\right), \C
).
\]
Note that the embedding $\iota \colon \Sp_r(F) \times \Sp_r(F) \hookrightarrow \Sp_{2r}(F)$
is not the usual one in this theory. 
For $\Phi \in I_{\psi}^{(n+r)}(\tau)$, $f \in (\mu^{-1}\chi_{-1}^{[-(n+r)/2]} \circ {\det}_{n-r}) \rtimes \pi$, 
and $v \in \pi$, 
we consider the integral 
\[
\ZZ_{r,n}(\Phi, f, v) = \int_{P_{n-r}(F) \bs \Sp_n(F)}
\ZZ_{r,r}\left( (\iota(g,1), \zeta)\Phi |_{\Mp_{2r}(F)} \otimes f(g,\zeta) \otimes v \right)
dg.
\]
\begin{prop}\label{zeta}
The integral $\ZZ_{r,n}(\Phi, f, v)$ is well-defined and gives a nonzero element in 
\[
\Hom_{\Mp_n(F) \times \Mp_r(F)}
\left( 
I_{\psi}^{(n+r)}(\tau) \otimes 
\left(
((\mu^{-1}\chi_{-1}^{[-(n+r)/2]} \circ {\det}_{n-r}) \rtimes \pi) 
\boxtimes \pi
\right), \C 
\right).
\]
\end{prop}
\begin{proof}
Since 
\[
h \cdot \iota(u,1) \cdot h^{-1} = 
\left(
\begin{array}{cc|cc}
\1_{n-r} & * & * & * \\
0 & \1_{2r} & * & 0 \\
\hline
0 & 0 & \1_{n-r} & 0 \\
0 & 0 & * & \1_{2r}
\end{array}
\right)
\]
for $u \in N_{n-r}(F) \subset \Sp_{n}(F)$ and $h \in \Mp_{2r}(F)$, 
we have $\Phi(h \cdot (\iota(ug,1), \zeta)) = \Phi(h \cdot (\iota(g,1), \zeta))$.
Similarly, we have
\begin{align*}
&\Phi(h \cdot \left( \iota(
\left(
\begin{array}{cc|cc}
a & 0 & 0 & 0 \\
0 & \1_{r} & 0 & 0 \\
\hline
0 & 0 & {}^ta^{-1} & 0 \\
0 & 0 & 0 & \1_{r}
\end{array}
\right)
g,1), \zeta \right))
\otimes
f(
\left(
\begin{array}{cc|cc}
a & 0 & 0 & 0 \\
0 & \1_{r} & 0 & 0 \\
\hline
0 & 0 & {}^ta^{-1} & 0 \\
0 & 0 & 0 & \1_{r}
\end{array}
\right)
g,\zeta)
\\&= 
|\det a|^{\half{n+r+1}}\Phi(h \cdot (\iota(g,1), \zeta)) \otimes f(g,\zeta)
\end{align*}
for $a \in \GL_{n-r}(F)$.
Hence we have
\begin{align*}
&\ZZ_{r,r}\left( (\iota(pg,1), \zeta)\Phi |_{\Mp_{2r}(F)} \otimes f(pg,\zeta) \otimes v \right)
\\&= \delta_{P_{n-r}}(p) \cdot 
\ZZ_{r,r}\left( (\iota(g,1), \zeta)\Phi |_{\Mp_{2r}(F)} \otimes f(g,\zeta) \otimes v \right)
\end{align*}
for $p \in P_{n-r}(F) \subset \Sp_{n}(F)$, where $ \delta_{P_{n-r}}$ is the modulus character of $P_{n-r}(F)$.
This shows that $\ZZ_{r,n}(\Phi, f, v)$ is well-defined.
\par

It is easy to see that
\[
\ZZ_{r,n} \in \Hom_{\Mp_n(F) \times \Mp_r(F)}
\left( 
I_{\psi}^{(n+r)}(\tau) \otimes 
\left(
((\mu^{-1}\chi_{-1}^{[-(n+r)/2]} \circ {\det}_{n-r}) \rtimes \pi) \boxtimes \pi
\right), \C 
\right).
\]
Since 
\[
I_{\psi}^{(n+r)}(\tau) \rightarrow 
\Ind_{\cl{P}_{2r}(F)}^{\Mp_{2r}(F)}(\mu \chi_{-1}^{[\half{n+r}]} |{\det}_{2r}|^{\half{n-r}}) ,\ 
\Phi \mapsto \Phi|_{\Mp_{2r}(F)}
\]
is surjective, 
we see that $\ZZ_{r,n} \not= 0$.
\end{proof}

Applying Proposition \ref{zeta} for $\pi^\vee$, 
we conclude that 
\[
\Hom_{\Mp_n(F) \times \Mp_r(F)}(I_{\psi}^{(n+r)}(\tau), ((\mu' \circ {\det}_{n-r}) \rtimes \pi) \boxtimes \pi) 
\not= 0
\]
since 
$((\mu' \circ {\det}_{n-r}) \rtimes \pi)^\vee 
\cong (\mu^{-1}\chi_{-1}^{[-(n+r)/2]} \circ {\det}_{n-r}) \rtimes \pi^\vee$.

\subsection{Irreducibility of degenerate induced representations}\label{(2)}
To show Theorem \ref{howe} (2), 
we need to prove the irreducibility of the induced representation $(\mu \circ {\det}_{n-r}) \rtimes \pi$, 
where $\pi$ is an irreducible almost tempered representation of $\Mp_r(F)$, 
and $\mu$ is a unitary character of $F^\times$.
When $\pi$ is supercuspidal and $\delta=0$, 
the irreducibility of $(\mu \circ {\det}_{n-r}) \rtimes \pi$ was proven by Tadi{\'c} \cite[Theorem 9.1]{T2}, 
and its Langlands data was given by Jantzen \cite{J}.
We imitate their proofs.
\par

We first show the irreducibility of other induced representations.
Let  $\pi$ be an irreducible representation of $\Mp_r(F)$.
We set $\delta = 1$ if $\pi$ is genuine, and $\delta = 0$ otherwise.
For a smooth representation $\Pi$ of $\Mp_n(F)$, 
we write $\semi(\Pi)$ for the semisimplification of $\Pi$.

\begin{prop}\label{lk}
Let $\mu$ be a unitary character of $F^\times$, 
and $\pi$ be an irreducible representation of $\Mp_r(F)$.
Suppose that the $L$-parameter $\phi$ for $\pi$ is of good parity.
Then for $l \geq k-\delta$ with $l \equiv k-\delta \bmod 2$, 
the induced representation
\[
\mu^{-1}|{\det}_k|^{\half{l}} \rtimes \pi
\]
is irreducible.
\end{prop}
\begin{proof}
We prove the proposition by induction on $k$.
When $k=1$ and $l>0$, this is Corollary \ref{cor2}.
When $k=1$ and $l=0$ so that $\delta=1$, 
the irreducibility of $\mu^{-1} \rtimes \pi$ follows from Theorem \ref{LLC} (6).
\par

Suppose that $k>1$
and that $\mu^{-1}|{\det}_k|^{\half{l}} \rtimes \pi$ is reducible.
Let $\sigma_1, \dots, \sigma_t$ be irreducible representations of $\Mp_{k+r}(F)$ such that
$\semi(\mu^{-1}|{\det}_k|^{\half{l}} \rtimes \pi) = \sigma_1 \oplus \dots \oplus \sigma_t$, 
and that $\sigma_1$ is a submodule 
and $\sigma_t$ is a quotient of $\mu^{-1}|{\det}_k|^{\half{l}} \rtimes \pi$:
\[
\sigma_1 \hookrightarrow \mu^{-1}|{\det}_k|^{\half{l}} \rtimes \pi \twoheadrightarrow \sigma_t.
\]
Since the Langlands quotient appears in the standard module with multiplicity one as a subquotients, 
we have $\sigma_i \not\cong \sigma_t$ for $i \not= t$.
\par

By Proposition \ref{jacquet}, we have
\begin{align*}
\semi R_{P_1(F)}(\mu^{-1}|{\det}_k|^{\half{l}} \rtimes \pi)
&= \mu|\cdot|^{-\half{l+k-1}} \boxtimes \left( \mu^{-1}|{\det}_{k-1}|^{\half{l-1}} \rtimes \pi \right)
\\&
\oplus \mu^{-1}|\cdot|^{\half{l-k+1}} \boxtimes \left( \mu^{-1}|{\det}_{k-1}|^\half{l+1} \rtimes \pi \right)
\\&
\oplus \semi\left( 
\bigoplus_{\lam} \chi_\lam |\cdot|^{\alpha_\lam} 
\boxtimes \left( \mu^{-1}|{\det}_k|^{\half{l}} \rtimes \pi_\lam \right) 
\right), 
\end{align*}
where $\semi R_{P_1(F)}(\pi) = \oplus_\lam \chi_\lam |\cdot|^{\alpha_\lam} \boxtimes \pi_\lam$
with $\chi_\lam$ unitary and $\alpha_\lam \in \R$.
By the induction hypothesis, we see that 
$\mu^{-1}|{\det}_{k-1}|^{(l \pm 1)/2} \rtimes \pi$ is irreducible.
By Corollary \ref{cor3}, we have $2\alpha_\lam \in \Z$ and $2\alpha_\lam \equiv l-k \bmod 2$.
In particular, the first two summands appear in $\semi R_{P_1(F)}(\mu^{-1}|{\det}_k|^{\half{l}} \rtimes \pi)$
with multiplicity one.
\par

Since $\sigma_1 \hookrightarrow \mu^{-1}|{\det}_k|^{\half{l}} \rtimes \pi$ 
and $\sigma_t \hookrightarrow \mu|{\det}_k|^{-\half{l}} \rtimes \pi$, 
we have
\begin{align*}
\semi R_{P_1(F)}(\sigma_i) 
\supset \mu^{-1}|\cdot|^{\half{l-k+1}} \boxtimes \left( \mu^{-1}|{\det}_{k-1}|^\half{l+1} \rtimes \pi \right)
& \iff i = 1, \\
\semi R_{P_1(F)}(\sigma_i)
\supset \mu|\cdot|^{-\half{l+k-1}} \boxtimes \left( \mu^{-1}|{\det}_{k-1}|^{\half{l-1}} \rtimes \pi \right)
& \iff i = t.
\end{align*}
\par

On the other hand, by Proposition \ref{jacquet}, we have
\[
\semi R_{P_k(F)}(\mu^{-1}|{\det}_k|^{\half{l}} \rtimes \pi) \supset 
\semi\left( \mu|{\det}_{k-1}|^{-\half{l+1}} \times \mu^{-1}|\cdot|^{\half{l-k+1}} \right) \boxtimes \pi.
\]
Note that $\mu|{\det}_{k-1}|^{-\half{l+1}} = \pair{\mu; -(l+k-1)/2, \dots, (-l+k-3)/2}$
and $\mu^{-1}|\cdot|^{\half{l-k+1}} = \pair{\mu^{-1}; (l-k+1)/2}$
in the notation in Appendix \ref{GL}.
Unless $\delta=1$, $l=k-1$ and $\mu^{-1} = \mu$, 
the representation $\mu|{\det}_{k-1}|^{-\half{l+1}} \times \mu^{-1}|\cdot|^{\half{l-k+1}}$ 
is irreducible by Theorem \ref{zel}.
When it is irreducible, 
there exists $i$ such that 
\[
\semi R_{P_k(F)}(\sigma_i) \supset 
\left( \mu|{\det}_{k-1}|^{-\half{l+1}} \times \mu^{-1}|\cdot|^{\half{l-k+1}} \right) \boxtimes \pi.
\]
Considering the Jacquet module with respect to $P_k \cap P_1$, 
we see that 
\[
\semi R_{P_1(F)}(\sigma_i) 
\supset 
\mu^{-1}|\cdot|^{\half{l-k+1}} \boxtimes \left( \mu^{-1}|{\det}_{k-1}|^\half{l+1} \rtimes \pi \right)
\oplus 
\mu|\cdot|^{-\half{l+k-1}} \boxtimes \left( \mu^{-1}|{\det}_{k-1}|^{\half{l-1}} \rtimes \pi \right).
\]
This implies that $i=1$ and $i=t$.
We obtain a contradiction.
\par

Similarly, when $k \geq 3$, we have
\[
\semi R_{P_k(F)}(\mu^{-1}|{\det}_k|^{\half{l}} \rtimes \pi) \supset 
\left( \mu|{\det}_{k-2}|^{-\half{l+2}} \times \mu^{-1}|{\det}_2|^{\half{l-k+2}} \right) \boxtimes \pi.
\]
The right hand side is irreducible by Theorem \ref{zel}
since $\mu|{\det}_{k-2}|^{-\half{l+2}} = \pair{\mu; -(l+k-1)/2, \dots, (-l+k-5)/2}$
and $\mu^{-1}|{\det}_2|^{\half{l-k+2}} = \pair{\mu^{-1}; (l-k+1)/2, (l-k+3)/2}$.
By the same argument, we obtain a contradiction.
\par

Therefore, the proposition is reduced to the case 
where $\delta=1$, $k=l+1=2$ and $\mu^{-1} = \mu$.
We treat this case in the following lemma.
\end{proof}

\begin{lem}
Let $\pi$ be an irreducible genuine representation of $\Mp_r(F)$ so that $\delta = 1$.
Suppose that $\mu$ is a quadratic character, and that
the $L$-parameter $\phi$ for $\pi$ is of good parity.
Then the induced representation
\[
\mu|{\det}_2|^{\half{1}} \rtimes \pi
\]
is irreducible.
\end{lem}
\begin{proof}
First, we show that $(\mu \circ {\det}_3) \rtimes \pi$ is irreducible in this case.
By Theorem \ref{zel} and Corollary \ref{cor2}, 
we have
\begin{align*}
\tau|\cdot| \times \mu \rtimes \pi
&\cong 
\mu|\cdot| \times \mu|\cdot| \times \mu \rtimes \pi
\\&\twoheadrightarrow
\mu|\cdot| \times \mu|{\det}_2|^{\half{1}} \rtimes \pi 
\\&\cong 
\mu|{\det}_2|^{\half{1}} \times \mu|\cdot| \rtimes \pi
\\&\cong 
\mu|{\det}_2|^{\half{1}} \times \mu|\cdot|^{-1} \rtimes \pi
\\&\twoheadrightarrow
(\mu \circ {\det}_3) \rtimes \pi.
\end{align*}
Since $\pi$ is unitary, $(\mu \circ {\det}_3) \rtimes \pi$ is semisimple. 
Therefore we deduce that $(\mu \circ {\det}_3) \rtimes \pi \cong J(\tau|\cdot|, \mu \rtimes \pi)$.
\par

Now we start to prove the lemma.
Suppose for the sake of contradiction that $\mu|{\det}_2|^{\half{1}} \rtimes \pi$ is reducible.
Take an irreducible submodule of $\sigma$, 
and an irreducible quotient $\sigma'$ of $\mu|{\det}_2|^{\half{1}} \rtimes \pi$: 
\[
\sigma \hookrightarrow 
\mu|{\det}_2|^{\half{1}} \rtimes \pi
\twoheadrightarrow \sigma'.
\]
Then $\sigma \not= \sigma'$.
Moreover, we have
\[
\mu|\cdot|^{1} \rtimes \sigma \hookrightarrow 
\mu|\cdot|^{1} \times \mu|{\det}_2|^{\half{1}} \rtimes \pi
\twoheadrightarrow \mu|\cdot|^{1} \rtimes \sigma'.
\]
Since 
$\tau|\cdot| \times \mu \rtimes \pi 
\twoheadrightarrow \mu|\cdot|^{1} \times \mu|{\det}_2|^{\half{1}} \rtimes \pi$, 
we have $\mu|\cdot|^{1} \rtimes \sigma' \twoheadrightarrow (\mu \circ \det_3) \rtimes \pi$.
Moreover, since any standard module has its Langlands quotient with multiplicity one as a subquotient, 
we have
\[
(\mu \circ {\det}_3) \rtimes \pi \not\subset \semi(\mu|\cdot|^{1} \rtimes \sigma).
\]
\par

Using proposition \ref{jacquet}, we compute the Jacquet modules
$\semi R_{P_1(F)}(\mu|\cdot|^{1} \times \mu|{\det}_2|^{\half{1}} \rtimes \pi)$
and $\semi R_{P_1(F)}((\mu \circ {\det}_3) \rtimes \pi)$.
Then the sums of all irreducible representations of the form $\mu|\cdot|^{-1} \boxtimes \Sigma$
which appear in $\semi R_{P_1(F)}(\mu|\cdot|^{1} \times \mu|{\det}_2|^{\half{1}} \rtimes \pi)$
and $\semi R_{P_1(F)}((\mu \circ {\det}_3) \rtimes \pi)$
are isomorphic to
\begin{align*}
&\semi \left(\mu|\cdot|^{-1} \boxtimes \mu|{\det}_2|^{\half{1}} \rtimes \pi \right)
\oplus 
\semi \left(\mu|\cdot|^{-1} \boxtimes (\mu|\cdot|^{1} \times \mu) \rtimes \pi \right)
\end{align*}
and
\[
\semi \left(\mu|\cdot|^{-1} \boxtimes \mu|{\det}_2|^{-\half{1}} \rtimes \pi \right)
\oplus 
\semi \left(\mu|\cdot|^{-1} \boxtimes \mu|{\det}_2|^{\half{1}} \rtimes \pi \right), 
\]
respectively.
Hence the difference is 
\[
\mu|\cdot|^{-1} \boxtimes \mu|\cdot|^{\half{1}}\St_2 \rtimes \pi, 
\]
where $\mu|\cdot|^{\half{1}}\St_2 = \pair{\mu; 1, 0}$ is a Steinberg representation.
By Proposition \ref{GPR}, we see that $\mu|\cdot|^{\half{1}}\St_2 \rtimes \pi$ is irreducible.
Since
\[
\semi R_{P_1(F)}(\mu|\cdot|^{1} \rtimes \sigma) \supset \mu|\cdot|^{-1} \boxtimes \sigma, 
\]
we deduce that $\sigma = \mu|\cdot|^{\half{1}}\St_2 \rtimes \pi$.
However, we have
\[
\semi R_{P_1(F)}(\mu|\cdot|^{\half{1}}\St_2 \rtimes \pi)
\supset 
\mu|\cdot|^1 \boxtimes \mu \rtimes \pi
\not\subset
\semi R_{P_1(F)}(\mu|{\det}_2|^{\half{1}} \rtimes \pi).
\]
Therefore, $\mu|\cdot|^{\half{1}}\St_2 \rtimes \pi \not\subset \mu|{\det}_2|^{\half{1}} \rtimes \pi$, 
which is a contradiction.
\end{proof}

Now we show the following.
\begin{thm}
Let $\pi$ be an irreducible almost tempered representation of $\Mp_r(F)$, 
and $\mu$ be a unitary character of $F^\times$.
Set $\tau = \mu \times \mu^{-1}$.
Then there exists a surjection
\[
\left\{
\begin{aligned}
\tau|\cdot|^{k-\half{1}} \times \tau|\cdot|^{k-\half{3}} \times \dots \times \tau|\cdot|^{\half{1}} \rtimes \pi
&\twoheadrightarrow (\mu \circ {\det}_{2k}) \rtimes \pi
\iif \delta = 0, \\
\tau|\cdot|^{k-1} \times \tau|\cdot|^{k-2} \times \dots \times \tau|\cdot|^{1}
\times \mu \rtimes \pi 
&\twoheadrightarrow (\mu \circ {\det}_{2k-1}) \rtimes \pi
\iif \delta = 1.
\end{aligned}
\right.
\]
In particular, if $\pi$ is unitary, then $(\mu \circ {\det}_{2k-\delta}) \rtimes \pi$ is irreducible. 
\end{thm}
\begin{proof}
By Theorem \ref{LLC}, 
we have
\[
\pi \cong \tau_1 \times \dots \times \tau_l \rtimes \pi_0, 
\]
where
\begin{itemize}
\item
$\tau_i = |\cdot|^{s_i}\tau_i'$ with
$\tau'_i$ being an irreducible discrete series representation of $\GL_{k_i}(F)$; 
\item
$1/2 > s_1 \geq s_2 \geq \dots \geq s_l \geq 0$; 
\item
when $s_i = 0$, 
the irreducible representation $\phi_i$ of $\WD_F$ corresponding to $\tau_i$
is not orthogonal if $\delta = 0$, and is not symplectic if $\delta = 1$; 
\item
$\pi_0$ is an irreducible representation of $\Mp_{r_0}(F)$ whose $L$-parameter is of good parity.
\end{itemize}
The segments corresponding to $\mu^{-1}|{\det}_k|^{\half{k-\delta}}$ and $\mu|{\det}_k|^{-\half{k-\delta}}$
are $[(1-\delta)/2, k-(1+\delta)/2]$ and $[-k+(1+\delta)/2, -(1-\delta)/2]$, respectively,
which contain $(1-\delta)/2$ or $-(1-\delta)/2$.
Hence by Theorem \ref{zel}, we see that
\begin{align*}
\mu^{-1}|{\det}_k|^{\half{k-\delta}} \times \tau_i &\cong \tau_i \times \mu^{-1}|{\det}_k|^{\half{k-\delta}}, \\
\mu|{\det}_k|^{-\half{k-\delta}} \times \tau_i &\cong \tau_i \times \mu|{\det}_k|^{-\half{k-\delta}}.
\end{align*}
By Proposition \ref{lk}, we have
\[
\mu^{-1}|{\det}_k|^{\half{k-\delta}} \rtimes \pi_0 \cong \mu|{\det}_k|^{-\half{k-\delta}} \rtimes \pi_0.
\]
Therefore, 
\begin{align*}
&\tau|\cdot|^{k-\half{1+\delta}} \times \tau|\cdot|^{k-\half{1+\delta}-1} \times \dots \times \tau|\cdot|^{1+\delta}
\times (\mu \circ {\det}_\delta) \rtimes \pi
\\&\twoheadrightarrow 
\mu|{\det}_{k-\delta}|^{\half{k}} \times \mu^{-1}|{\det}_k|^{\half{k-\delta}} \rtimes \pi
\\&\cong
\mu|{\det}_{k-\delta}|^{\half{k}} \times \tau_1 \times \dots \times \tau_l 
\times \mu^{-1}|{\det}_k|^{\half{k-\delta}} \rtimes \pi_0
\\&\cong
\mu|{\det}_{k-\delta}|^{\half{k}} \times \tau_1 \times \dots \times \tau_l 
\times \mu|{\det}_k|^{-\half{k-\delta}} \rtimes \pi_0
\\&\cong 
\mu|{\det}_{k-\delta}|^{\half{k}} \times \mu|{\det}_k|^{-\half{k-\delta}} \rtimes \pi
\\&\twoheadrightarrow
(\mu \circ {\det}_{2k-\delta}) \rtimes \pi.
\end{align*}
If $\pi$ is unitary, then $(\mu \circ {\det}_{2k-\delta}) \rtimes \pi$ is semisimple.
Therefore, by the uniqueness of the Langlands quotient, 
we conclude that $(\mu \circ {\det}_{2k-\delta}) \rtimes \pi$ is irreducible. 
\end{proof}

This theorem together with Proposition \ref{det} implies Theorem \ref{howe} (2).

\subsection{Going down case}\label{(4)}
Finally, we show Theorem \ref{howe} (4).
\par

Let $\pi$ be an irreducible almost tempered unitary representation of $\Mp_r(F)$, 
and set
\[
\pi' = \MM_{\psi, \tau}^{(r)}\left( \MM_{\psi, \tau}^{(n)}(\pi) \right).
\]
Since there is a surjection $\pi' \twoheadrightarrow \pi$, we have $\pi' \not= 0$.
Moreover, by the definition and Theorem \ref{howe} (2), 
we obtain a surjection 
\[
I^{(n+r)}_\psi(\tau) \twoheadrightarrow \left( (\mu' \circ {\det}_{n-r}) \rtimes \pi \right) \boxtimes \pi'.
\]
\par

By Lemma \ref{filt}, there is $0 \leq t \leq r$ such that 
$\Hom_{\Mp_{n}(F) \times \Mp_{r}(F)}(R_t, ( (\mu' \circ {\det}_{n-r}) \rtimes \pi ) \boxtimes \pi') \not= 0$.
By Bernstein's Frobenius reciprocity, this $\Hom$-space is isomorphic to
the space of $(\cl{M}_{t+(n-r)}(F) \times \cl{M}_{t}(F))$-equivalent maps
\begin{align*}
&(\mu'|{\det}_{t+(n-r)}|^{\half{t}} \boxtimes \mu'|{\det}_{t}|^{\half{t+(n-r)}})
\otimes C_c^\infty(\Mp_{r-t}(F))
\\&\rightarrow
R_{\overline{P_{t+(n-r)}(F)}}( (\mu' \circ {\det}_{n-r}) \rtimes \pi ) \boxtimes R_{\overline{P_{t}(F)}}(\pi'). 
\end{align*}
In particular, we have a nonzero $\cl\GL_{t+(n-r)}(F)$-equivalent map
\[
R_{P_{t+(n-r)}(F)}( (\mu'^{-1} \circ {\det}_{n-r}) \rtimes \pi^\vee ) 
\rightarrow \mu'^{-1}|{\det}_{t+(n-r)}|^{-\half{t}}.
\]
By Proposition \ref{jacquet}, 
$\semi R_{P_{t+(n-r)}(F)}( (\mu'^{-1} \circ {\det}_{n-r}) \rtimes \pi^\vee )$ is the sum of 
\[
\semi\left( \mu'|{\det}_{n-r-a}|^{-\half{a}} \times \mu'^{-1}|{\det}_b|^{-\half{n-r-b}} \times \tau_\lam \right)
\boxtimes
\left( \mu'^{-1}|{\det}_{a-b}|^{-\half{n-r-a-b}} \rtimes \pi_\lam \right), 
\]
where $(a,b)$ runs over the pairs of integers such that $0 \leq b \leq a \leq n-r$, 
and $\tau_\lam \boxtimes \pi_\lam$ runs over all irreducible subquotients of $R_{P_{a-b+t}(F)}(\pi^\vee)$.
Hence we have
\[
\semi\left( \mu'|{\det}_{n-r-a}|^{-\half{a}} \times \mu'^{-1}|{\det}_b|^{-\half{n-r-b}} \times \tau_\lam \right)
\supset \mu'^{-1}|{\det}_{t+(n-r)}|^{-\half{t}}
\]
for some $(a,b)$ and $\lam$.
Note that the segments corresponding to 
$\mu'|{\det}_{n-r-a}|^{-\half{a}}$, $\mu'^{-1}|{\det}_b|^{-\half{n-r-b}}$ 
and $\mu'^{-1}|{\det}_{t+(n-r)}|^{-\half{t}}$
are $[-(n-r-1)/2, (n-r-1)/2-a]$, $[-(n-r-1)/2, -(n-r-1)/2+b-1]$
and $[-(n-r-1)/2-t, (n-r-1)/2]$, respectively.
If $t \geq 1$, then $\semi R_{P_1(F)}(\pi^\vee)$ must contain a nonzero representation on which 
$\cl\GL_1(F)$ acts by $\mu'^{-1}|\cdot|^{-(n-r-1)/2-t}$.
This contradicts that $\pi$ is almost tempered, by Casselman's criterion.
Hence we must have $t=0$. 
\par

Therefore, we obtain a surjection
\[
\Ind_{\cl{P}_{n-r}(F) \times \Mp_r(F)}^{\Mp_n(F) \times \Mp_r(F)}
\left( (\mu' \circ {\det}_{n-r}) \otimes C_c^\infty(\Mp_r(F)) \right)
\twoheadrightarrow 
\left( (\mu' \circ {\det}_{n-r}) \rtimes \pi \right) \boxtimes \pi'.
\]
Now Theorem \ref{howe} (4) follows from the following lemma. 

\begin{lem}\label{msi}
Let $\pi$ be an irreducible almost tempered unitary representation of $\Mp_r(F)$.
Then the maximal $((\mu' \circ {\det}_{n-r}) \rtimes \pi)$-isotypic quotient of
\[
\Ind_{\cl{P}_{n-r}(F) \times \Mp_r(F)}^{\Mp_n(F) \times \Mp_r(F)}
\left( (\mu' \circ {\det}_{n-r}) \otimes C_c^\infty(\Mp_r(F)) \right)
\]
is isomorphic to $( (\mu' \circ {\det}_{n-r}) \rtimes \pi ) \boxtimes \Pi$, 
where $\Pi$ is determined so that
the subrepresentation of $R_{P_{n-r}(F)}( (\mu'^{-1} \circ {\det}_{n-r}) \rtimes \pi )$ 
on which $\cl\GL_{n-r}(F)$ acts by $\mu'^{-1} \circ {\det}_{n-r}$ is isomorphic to 
$(\mu'^{-1} \circ {\det}_{n-r}) \boxtimes \Pi$.
Moreover, if $\pi$ satisfies one of the conditions in Theorem \ref{howe} (4), 
then $\semi(\Pi) = \pi^\alpha$ for some integer $\alpha \geq 1$, and that
\[
\dim\Hom_{\Mp_r(F)}(\Pi, \pi) = 1.
\]
\end{lem}
\begin{proof}
The linear dual $\Pi^*$ of $\Pi$, which is not necessarily smooth, is given by
\[
\Pi^* \cong \Hom_{\Mp_n(F)}
\left(
\Ind_{\cl{P}_{n-r}(F) \times \Mp_r(F)}^{\Mp_n(F) \times \Mp_r(F)}
\left( (\mu' \circ {\det}_{n-r}) \otimes C_c^\infty(\Mp_r(F)) \right),
(\mu' \circ {\det}_{n-r}) \rtimes \pi
\right).
\]
As representations of $\Mp_n(F)$, we have
\begin{align*}
\Ind_{\cl{P}_{n-r}(F) \times \Mp_r(F)}^{\Mp_n(F) \times \Mp_r(F)}
\left( (\mu' \circ {\det}_{n-r}) \otimes C_c^\infty(\Mp_r(F)) \right)
\cong 
\ind_{\cl{P}_{n-r}^\circ(F)}^{\Mp_n(F)}(\chi), 
\end{align*}
where
\begin{itemize}
\item
$P_{n-r}^\circ(F) \subset P_{n-r}(F)$ is generated by $\GL_{n-r}(F)$ and $N_{n-r}(F)$, 
so that $P_{n-r}(F)/P_{n-r}^\circ(F) \cong \Sp_r(F)$; 
\item
$\ind_{\cl{P}_{n-r}^\circ(F)}^{\Mp_n(F)}$ is the (unnormalized) compact induction functor; 
\item
the character $\chi \colon \cl{P}_{n-r}^\circ(F) \rightarrow \C^\times$ is given by
\[
\chi(
\left(
\begin{array}{cc|cc}
a & * & * & * \\
0 & \1_r & * & 0 \\
\hline
0 & 0 & {}^t a^{-1} & 0 \\
0 & 0 & * & \1_r
\end{array}
\right), \zeta
)
= \zeta^{n+r} \left( \frac{\alpha_{\psi}(1)}{\alpha_{\psi}(\det a)} \right)^{n+r} 
\mu(\det a) |\det a|^{\half{n+r+1}}.
\]
\end{itemize}
Moreover, $\Mp_r(F)$ acts on $\ind_{\cl{P}_{n-r}^\circ(F)}^{\Mp_n(F)}(\chi)$ by 
\[
(g_r \cdot \varphi)(g_n) = \varphi((g_r^{-1})^{\alpha} g_n)
\]
for $g_r \in \Mp_r(F)$, $g_n \in \Mp_n(F)$ and $\varphi \in \ind_{\cl{P}_{n-r}^\circ(F)}^{\Mp_n(F)}(\chi)$.
Here, we set
\[
g_r^\alpha = 
\begin{pmatrix}
\1_r & 0 \\ 0 & -\1_r
\end{pmatrix}
g_r 
\begin{pmatrix}
\1_r & 0 \\ 0 & -\1_r
\end{pmatrix}^{-1}, 
\]
and we identify $\Mp_r(F)$ with the image of the composition 
$\Mp_{r}(F) \hookrightarrow \cl{M}_{n-r}(F) \hookrightarrow \Mp_{n}(F)$.
By the Frobenius reciprocity, we have
\begin{align*}
\Pi^* 
&\cong 
\Hom_{\Mp_n(F)}((\mu'^{-1} \circ {\det}_{n-r}) \rtimes \pi^\vee, 
\Ind_{\cl{P}_{n-r}^\circ(F)}^{\Mp_n(F)}(\chi^{-1} \delta_{P_{n-r}}))
\\&\cong 
\Hom_{\cl{P}_{n-r}^\circ(F)} ((\mu'^{-1} \circ {\det}_{n-r}) \rtimes \pi^\vee|_{\cl{P}_{n-r}^\circ(F)}, 
\chi^{-1}\delta_{P_{n-r}}) 
\\&\cong 
\Hom_{\cl\GL_{n-r}(F)} (R_{P_{n-r}(F)}
\left( (\mu'^{-1} \circ {\det}_{n-r}) \rtimes \pi^\vee \right)|_{\cl\GL_{n-r}(F)}, 
\mu'^{-1} \circ {\det}_{n-r}).
\end{align*}
Moreover, $\Mp_r(F)$ acts on this $\Hom$-space by 
\[
(g_r \cdot \Phi)(f) = \Phi((g_r^{-1})^\alpha \cdot f)
\]
for $\Phi$ in this $\Hom$-space, 
and for $f \in R_{P_{n-r}(F)}\left( (\mu'^{-1} \circ {\det}_{n-r}) \rtimes \pi^\vee \right)$.
Since the representation $[g_r \mapsto \pi^\vee(g_r^\alpha)]$ is isomorphic to $\pi$, 
by the definition of the Jacquet module, we see that
the action $f \mapsto g_r^\alpha \cdot f$ is isomorphic 
to $R_{P_{n-r}(F)}\left( (\mu'^{-1} \circ {\det}_{n-r}) \rtimes \pi \right)$.
Hence 
\[
\Pi^* \cong 
\Hom_{\cl\GL_{n-r}(F)} (R_{P_{n-r}(F)}\left( (\mu'^{-1} \circ {\det}_{n-r}) \rtimes \pi \right)|_{\cl\GL_{n-r}(F)}, 
\mu'^{-1} \circ {\det}_{n-r}).
\]
as representations of $\Mp_r(F)$.
This means that $\Pi$ is determined so that 
the subrepresentation of $R_{P_{n-r}(F)}( (\mu'^{-1} \circ {\det}_{n-r}) \rtimes \pi )$ 
on which $\cl\GL_{n-r}(F)$ acts by $\mu'^{-1} \circ {\det}_{n-r}$ is isomorphic to 
$(\mu'^{-1} \circ {\det}_{n-r}) \boxtimes \Pi$
since $R_{P_{n-r}(F)}( (\mu'^{-1} \circ {\det}_{n-r}) \rtimes \pi )$ is admissible.
\par

Note that $\semi R_{P_{n-r}(F)}( (\mu'^{-1} \circ {\det}_{n-r}) \rtimes \pi )$ is the sum of 
\[
I_{(a,b), \lam} = 
\semi\left( \mu'|{\det}_{n-r-a}|^{-\half{a}} \times \mu'^{-1}|{\det}_b|^{-\half{n-r-b}} \times \tau_\lam \right)
\boxtimes
\left( \mu'^{-1}|{\det}_{a-b}|^{-\half{n-r-a-b}} \rtimes \pi_\lam \right), 
\]
where $(a,b)$ runs over the pairs of integers such that $0 \leq b \leq a \leq n-r$, 
and $\tau_\lam \boxtimes \pi_\lam$ runs over all irreducible subquotients of $R_{P_{a-b}(F)}(\pi)$.
We claim that when $\pi$ satisfies one of the conditions in Theorem \ref{howe} (4), 
if 
\[
I_{(a,b), \lam}|_{\cl\GL_{n-r}(F)} \supset \mu'^{-1} \circ {\det}_{n-r}, 
\]
then $I_{(a,b), \lam} \cong \mu'^{-1} \circ {\det}_{n-r} \boxtimes \pi$.
\par

Note that the segments corresponding to 
$\mu'|{\det}_{n-r-a}|^{-\half{a}}$, $\mu'^{-1}|{\det}_b|^{-\half{n-r-b}}$ 
and $\mu'^{-1} \circ {\det}_{n-r}$
are $[-(n-r-1)/2, (n-r-1)/2-a]$, $[-(n-r-1)/2, -(n-r-1)/2+b-1]$
and $[-(n-r-1)/2, (n-r-1)/2]$, respectively.
When the $L$-parameter $\phi$ does not contain $\mu^{\pm1} S_d$ 
for any $d \geq n-r$ with $d \equiv n-r \bmod2$, 
by computing the Jacquet module with respect to the Borel subgroup of $\cl{\GL}_{n-r}(F)$, 
we see that if
\[
\semi\left( \mu'|{\det}_{n-r-a}|^{-\half{a}} \times \mu'^{-1}|{\det}_b|^{-\half{n-r-b}} \times \tau_\lam \right)
\supset 
\mu'^{-1} \circ {\det}_{n-r}, 
\]
then $(a,b) = (0,0)$ or $(a,b) = (n-r, n-r)$.
In these cases, 
we have $I_{(a,b), \lam} = \mu'^{\pm1} \circ {\det}_{n-r} \boxtimes \pi$, as desired.
\par

When $n=r$, there is nothing to prove.
Now we assume that $n=r+1$.
We take the maximal integer $a$ such that 
\[
\pi \cong \underbrace{\mu'^{-1} \times \dots \times \mu'^{-1}}_{a-1} \rtimes \pi_0
\]
for some irreducible representation $\pi_0$.
Then by induction on $a$, 
one can show that 
the subrepresentation of $\semi R_{P_1(F)}( \mu'^{-1} \rtimes \pi )$ 
on which $\cl\GL_{1}(F)$ acts by $\mu'^{\pm1}$ is isomorphic to 
\[
\left( (\mu' \boxtimes \pi) \oplus (\mu'^{-1} \boxtimes \pi) \right)^{\oplus a}.
\]
This proves the claim when $n=r+1$.
\par

When $\pi$ satisfies one of the conditions in Theorem \ref{howe} (4), 
by the claim, we see that $\semi(\Pi)$ is of the form $\pi^{\oplus \alpha}$ for some integer $\alpha \geq 1$.
Moreover, we have
\begin{align*}
&\Hom_{\Mp_r(F)}(\Pi, \pi) \cong (\Pi^* \otimes \pi)^{\Mp_r(F)}
\\&\cong \Hom_{\cl\GL_{n-r}(F)} (R_{P_{n-r}(F)}
\left( (\mu'^{-1} \circ {\det}_{n-r}) \rtimes \pi \right)|_{\cl\GL_{n-r}(F)}, 
(\mu'^{-1} \circ {\det}_{n-r}) \boxtimes \pi)^{\Mp_r(F)}
\\&\cong \End_{\Mp_r(F)}((\mu'^{-1} \circ {\det}_{n-r}) \rtimes \pi), 
\end{align*}
which is one dimensional by Theorem \ref{howe} (2).
\end{proof}

\begin{rem}\label{rem}
\begin{enumerate}
\item
If the $L$-parameter $\phi$ contains one of $\mu S_d$ or $\mu^{-1} S_d$ 
for some $d \geq n-r$ with $d \equiv n-r \bmod2$, 
then $\phi \supset \mu S_d \oplus \mu^{-1} S_d$.
By comparing the dimensions, we have
\[
2r+1-\delta \geq 2d \geq 2(n-r).
\]
Hence we must have $n \leq 2r$.
In other words, when $n > 2r$, the first condition in Theorem \ref{howe} (4) always holds.

\item
In general, $\Pi$ in Lemma \ref{msi} might have irreducible subquotients other than $\pi$.
For example, consider the case where $r=2$, $n=4$ and $\pi = \mu'^{-1}\St_2 \rtimes \1_{\Sp_0(F)}$.
Then one can see that
\[
\semi(\Pi) \supset (\mu'^{-1} \circ {\det}_2) \rtimes \1_{\Sp_0(F)} \not\cong \pi.
\]
\end{enumerate}
\end{rem}

\subsection{Seesaw identities}
In the theory of theta liftings, seesaw identities are useful tools.
The following proposition is an analogy for Miyawaki liftings.
Recall that $J_{n-1}(F) = \Sp_{n-1}(F) \ltimes V_{n-1}(F)$ is a Jacobi subgroup of $\Sp_n(F)$.
The center of the Heisenberg group $V_{n-1}(F)$ is denoted by $Z_{n-1}(F)$.
For $\xi \in F^\times$, 
we denote the Weil representations of $\cl{J}_{n-1}(F)$ and $\Mp_r(F)$ with respect to $\psi_\xi$
by $\omega_{\psi_\xi}^{(n-1)}$ and $\omega_{\psi_\xi}^{(r)}$, respectively.

\begin{prop}[Seesaw identity]\label{LSS}
Let $\pi$ and $\pi'$ be irreducible representations of $\Mp_r(F)$ and $\Mp_{n-1}(F)$, 
on which $\{\pm1\}$ acts by $(\pm1)^{r+n}$ and $(\pm1)^{r+n-1}$, respectively.
Then 
\[
\Hom_{\cl{J}_{n-1}(F)}(\MM_{\psi, \tau}^{(n)}(\pi)|_{\cl{J}_{n-1}(F)}, \pi' \otimes \omega_{\psi_\xi}^{(n-1)}) 
\not= 0
\iff
\Hom_{\Mp_r(F)}(\MM_{\psi, \tau\chi_\xi}^{(r)}(\pi') \otimes \omega_{\psi_\xi}^{(r)}, \pi) \not= 0.
\]
We shall write this property as the following seesaw diagram:
\[
\xymatrix{
\Mp_r(F) \times \Mp_r(F) \ar@{-}[d] \ar@{-}[dr]& \Mp_n(F) \ar@{-}[d]\\
\Mp_r(F) \ar@{-}[ur]& \Mp_{n-1}(F) \ltimes V_{n-1}(F).
}
\]
\end{prop}
\begin{proof}
By proposition \ref{FJ-nonarch}, 
there exists a $\cl{J}_{n+r-1}(F)$-surjection
\[
I_\psi^{(n+r)}(\tau) \twoheadrightarrow I_\psi^{(n+r-1)}(\tau\chi_\xi) \otimes \omega_{\psi_\xi}^{(n+r-1)}.
\]
Note that as $\cl{J}_{n-1}(F) \times \Mp_{r}(F)$-modules, 
we have $\omega_{\psi_\xi}^{(n+r-1)} \cong \omega_{\psi_\xi}^{(n-1)} \boxtimes \omega_{\psi_\xi}^{(r)}$.
Composing a surjective $\Mp_{n-1}(F) \times \Mp_{r}(F)$-map
\[
I_\psi^{(n+r-1)}(\tau\chi_\xi) \twoheadrightarrow \pi' \boxtimes \MM_{\psi, \tau\chi_\xi}^{(r)}(\pi'), 
\]
we obtain a nonzero $\cl{J}_{n-1}(F) \times \Mp_r(F)$-map
\[
I_\psi^{(n+r)}(\tau) \twoheadrightarrow 
(\pi' \otimes \omega_{\psi_\xi}^{(n-1)}) \boxtimes 
(\MM_{\psi, \tau\chi_\xi}^{(r)}(\pi') \otimes \omega_{\psi_\xi}^{(r)}).
\]
If $\Hom_{\Mp_r(F)}(\MM_{\psi, \tau\chi_\xi}^{(r)}(\pi') \otimes \omega_{\psi_\xi}^{(r)}, \pi) \not= 0$, 
then we obtain a nonzero $\cl{J}_{n-1}(F) \times \Mp_r(F)$-map
\[
I_\psi^{(n+r)}(\tau) \twoheadrightarrow 
(\pi' \otimes \omega_{\psi_\xi}^{(n-1)}) \boxtimes \pi.
\]
By the definition of local Miyawaki liftings, it implies a nonzero $\cl{J}_{n-1}(F)$-map
$\MM_{\psi, \tau}^{(n)}(\pi) \rightarrow \pi' \otimes \omega_{\psi_\xi}^{(n-1)}$ so that
\[
\Hom_{\cl{J}_{n-1}(F)}(\MM_{\psi, \tau}^{(n)}(\pi)|_{\cl{J}_{n-1}(F)}, \pi' \otimes \omega_{\psi_\xi}^{(n-1)}) 
\not= 0.
\]
\par

Conversely, if 
$\Hom_{\cl{J}_{n-1}(F)}(\MM_{\psi, \tau}^{(n)}(\pi)|_{\cl{J}_{n-1}(F)}, 
\pi' \otimes \omega_{\psi_\xi}^{(n-1)}) \not= 0$, 
then we have a nonzero $\cl{J}_{n-1}(F) \times \Mp_r(F)$-map
\[
I_\psi^{(n+r)}(\tau) \twoheadrightarrow 
\MM_{\psi, \tau}^{(n)}(\pi) \boxtimes \pi
\rightarrow
(\pi' \otimes \omega_{\psi_\xi}^{(n-1)}) \boxtimes \pi.
\]
It factors through 
\[
\left( I_\psi^{(n+r)}(\tau) \right)_{\psi_\xi},  
\]
that is the maximal quotient of $I_\psi^{(n+r)}(\tau)$ on which $Z_{n+r-1}(F)$ acts by $\psi_\xi$.
By the proof of Proposition \ref{FJ11} together with Proposition \ref{FJ-nonarch}, 
we see that 
\[
\left( I_\psi^{(n+r)}(\tau) \right)_{\psi_\xi} 
\cong I_\psi^{(n+r-1)}(\tau\chi_\xi) \otimes \omega_{\psi_\xi}^{(n+r-1)}
\]
as $\cl{J}_{n+r-1}(F)$-modules.
Hence we have a nonzero $\cl{J}_{n-1}(F) \times \Mp_r(F)$-map
\[
I_\psi^{(n+r-1)}(\tau\chi_\xi) \otimes \omega_{\psi_\xi}^{(n+r-1)}
\rightarrow 
(\pi' \otimes \omega_{\psi_\xi}^{(n-1)}) \boxtimes \pi.
\]
This implies that
\[
\Hom_{\Mp_r(F)}(\MM_{\psi, \tau\chi_\xi}^{(r)}(\pi') \otimes \omega_{\psi_\xi}^{(r)}, \pi) \not= 0.
\]
This completes the proof.
\end{proof}

\section{Global Miyawaki liftings}\label{global}
Now we let $F$ be a totally real number field, 
and $\psi$ be a non-trivial unitary character of $\A/F$.
We assume that for each infinite place $v$ of $F$, 
there exists $a_v \in F_v \cong \R$ with $a_v > 0$ 
such that $\psi(x_v) = \exp(2\pi a_v \I x_v)$ for $x_v \in F_v$.
\par

In this section, we define the global Miyawaki liftings as the pullbacks of Ikeda liftings, 
and establish basic properties.
One may regard the global Miyawaki liftings as an analogue of the global theta liftings, 
which are defined by the pullbacks of theta functions.

\subsection{Definition}
Let $\tau = \otimes_v' \tau_v$ be an irreducible unitary cuspidal automorphic representation of $\GL_2(\A)$
satisfying the conditions (A1), (A2) and (A3) in \S \ref{sec.global}.
We denote the weight of $\otimes_{v \mid \infty}\tau_v$ 
by $\pm2k = (\pm2k_v)_{v}$ with $k_v > 0$.
Then we have the Ikeda lift $\Ik^{(n+r)}_\psi(\tau)$, 
which is an irreducible admissible representation of $\Mp_{n+r}(\A_\fin)$
occurring in the space $\Sc_{k+(n+r)/2}(\Sp_{n+r}(F) \bs \Mp_{n+r}(\A))$
of holomorphic cusp forms on $\Mp_{n+r}(\A)$ of weight $k + (n+r)/2$.
For $\FF \in \Ik^{(n+r)}_\psi(\tau)$ and $\varphi \in \Sc_{k+(n+r)/2}(\Sp_{r}(F) \bs \Mp_{r}(\A))$, 
consider the integral 
\[
\MM^{(n)}((g_n, \zeta_n); \varphi, \FF) 
= \int_{\Sp_{r}(F) \bs \Sp_{r}(\A)} \FF(\iota(g_n, g_r), \zeta_n\zeta_r) \overline{\varphi(g_r, \zeta_r)} dg_r 
\]
for $(g_n, \zeta_n) \in \Mp_{n}(\A)$, 
where $dg_r$ is the Tamagawa measure on $\Sp_{r}(F) \bs \Sp_{r}(\A)$.
Note that this integral does not depend on the choice of $\zeta_r \in \{\pm1\}$, 
and that $\MM^{(n)}(\varphi, \FF)$ is genuine if and only if $n+r$ is odd.

\begin{lem}
We have
\[
\MM^{(n)}(\varphi, \FF) \in \Sc_{k+(n+r)/2}(\Sp_{n}(F) \bs \Mp_{n}(\A)). 
\]
\end{lem}
\begin{proof}
The non-trivial part is the cuspidality.
Note that $\MM^{(n)}(\varphi, \FF)$ has a Fourier expansion of the form
\[
\MM^{(n)}((g_n, \zeta_n); \varphi, \FF) = \sum_{B \in \Sym_n^+(F)} W_B(g_n, \zeta_n), 
\]
where 
\[
W_B(g_n, \zeta_n) = 
\int_{\Sym_n(F) \bs \Sym_n(\A)} \MM^{(n)}(\bn(X)(g_n, \zeta_n); \varphi, \FF) \overline{\psi(\Tr(BX))}dX.
\]
Since 
\[
\bv(0,y,z) \mapsto \psi(\Tr
\left(B
\begin{pmatrix}
z & y \\ {}^ty & 0
\end{pmatrix}
\right)
)
\]
is a non-trivial character of $Y(\A)Z(\A)$ for any $B \in \Sym_n^+(F)$, 
we have 
\[
\int_{Y(F)Z(F) \bs Y(\A)Z(\A)} \MM^{(n)}(\bv(0,y,z)(g_n, \zeta_n); \varphi, \FF)dydz = 0.
\]
Since $YZ$ is a normal subgroup of $N_k$, 
which is the unipotent radical of the standard maximal parabolic subgroup $P_k$, 
the constant term of $\MM^{(n)}(\varphi, \FF)$ along $P_k$ must be zero.
Hence $\MM^{(n)}(\varphi, \FF)$ is cuspidal.
\end{proof}

Let $\pi$ be an irreducible admissible representation of $\Mp_{r}(\A_\fin)$
occurring in $\Sc_{k+(n+r)/2}(\Sp_{r}(F) \bs \Mp_{r}(\A))$.
The (global) Miyawaki lift $\MM^{(n)}_{\psi, \tau}(\pi)$ of $\pi$ is defined by 
the representation of $\Mp_{n}(\A_\fin)$ generated by
\[
\left\{
\MM^{(n)}(\varphi, \FF) \in \Sc_{k+(n+r)/2}(\Sp_{n}(F) \bs \Mp_{n}(\A))\ |\ 
\FF \in \Ik^{(n+r)}_\psi(\tau),\ \varphi \in \pi
\right\}.
\]

\subsection{$A$-parameters}\label{Ap}
Arthur's multiplicity formula established by Arthur \cite{Ar} and Gan--Ichino \cite{GI2} 
describes the discrete spectrum of automorphic forms on
$\Sp_n(F) \bs \Mp_n(\A)$ in terms of global $A$-parameters.
In this subsection, we review Arthur's multiplicity formula for holomorphic cusp forms.
For more precision, see also Appendix \ref{app.AMF}.
\par

A discrete global $A$-parameter for $\Sp_n(F)$ (\resp $\Mp_n(F)$) is 
a symbol
\[
\Psi = \tau_1[d_1] \boxplus \dots \boxplus \tau_t[d_t], 
\]
where $\tau_i$ is an irreducible cuspidal unitary automorphic representation of $\GL_{m_i}(\A)$, 
and $d_i$ is a positive integer such that 
they satisfy several conditions, e.g., 
$\sum_{i=1}^{t}m_id_i$ is equal to $2n+1$ (\resp $2n$).
The precise definition is given in Appendix \ref{app.AMF}.
Two $A$-parameters $\Psi = \boxplus_{i=1}^{t}\tau_{i}[d_{i}]$ and $\Psi' = \boxplus_{i=1}^{t'}\tau'_{i}[d'_{i}]$
are said to be equivalent if $t = t'$ and there exists a permutation $\sigma \in \mathfrak{S}_t$ such that
$d'_i = d_{\sigma(i)}$ and $\tau'_i \cong \tau_{\sigma(i)}$ for each $i$.
We denote the set of equivalence classes of discrete global $A$-parameters 
for $\Sp_n(F)$ (\resp $\Mp_n(F)$) by $\Psi_2(\Sp_n(F))$ (\resp $\Psi_2(\Mp_n(F))$).
We call an $A$-parameter $\Psi = \boxplus_{i=1}^{t}\tau_{i}[d_{i}]$ tempered if $d_i = 1$ for any $i$.
In this case, we write $\Psi = \boxplus_{i=1}^{t}\tau_{i}$ for simplicity.
\par

We state Arthur's multiplicity formula (\cite[Theorem 1.5.2]{Ar} and \cite[Theorems 1.1, 1.3]{GI2}) 
for holomorphic cusp forms.
\begin{thm}[Arthur's multiplicity formula]\label{AMFhol}
Let $l = (l_v) \in \prod_{v \mid \infty}\Z$ with $l_v > 0$ 
and $l_v \equiv l_{v'} \bmod 2$ for any $v,v' \mid \infty$.
\begin{enumerate}
\item
For $\Psi \in \Psi_2(\Sp_n(F))$ if $l_v$ is even, and for $\Psi \in \Psi_2(\Mp_n(F))$ if $l_v$ is odd, 
there exists an $\Mp_{n}(\A_\fin)$-stable subspace $\Sc_{l/2, \Psi}$ of $\Sc_{l/2}(\Sp_n(F) \bs \Mp_n(\A))$
(possibly zero) such that
\[
\Sc_{l/2}(\Sp_n(F) \bs \Mp_n(\A)) = \bigoplus_{\Psi}\Sc_{l/2, \Psi}, 
\]
where $\Psi$ runs over $\Psi_2(\Sp_n(F))$ if $l_v$ is even, and over $\Psi_2(\Mp_n(F))$ if $l_v$ is odd.

\item
Suppose that $\pi = \otimes'_{v < \infty} \pi_v$ is an irreducible subrepresentation of $\Sc_{l/2, \Psi}$ 
with $\Psi = \boxplus_{i=1}^{t}\tau_{i}[d_{i}]$.
Then for almost all $v < \infty$, 
the local factors $\pi_v$ of $\pi$ and $\tau_{i,v}$ of $\tau_i$ are unramified for any $i$.
Moreover, if we denote the Satake parameter for $\tau_{i,v}$ by $\{c_{i,v,1}, \dots, c_{i,v,m_i}\}$, 
then the Satake parameter of $\pi_v$ is equal to
\[
\bigcup_{i=1}^{t} \bigcup_{j=1}^{m_i} 
\left\{ c_{i,v,j}q^{-\half{d_i-1}}, c_{i,v,j}q^{-\half{d_i-3}}, \dots c_{i,v,j}q^{\half{d_i-1}} \right\}
\]
as multisets.

\item
If $\Psi$ is a tempered $A$-parameter,
then $\Sc_{l/2, \Psi}$ is multiplicity-free as a representation of $\Mp_{n}(\A_\fin)$.

\end{enumerate}
\end{thm}
\par

If an admissible representation $\pi$ of $\Mp_n(\A_\fin)$ is contained in $\Sc_{l/2, \Psi}$, 
we say that $\pi$ has an $A$-parameter $\Psi$, and $\Psi$ is the $A$-parameter for $\pi$.
\par

\subsection{Basic properties}
We establish basic properties of Miyawaki liftings.
Let $\pi = \otimes'_{v < \infty} \pi_v$ be an irreducible admissible representation of $\Mp_{r}(\A_\fin)$
occurring in $\Sc_{k+(n+r)/2}(\Sp_{r}(F) \bs \Mp_{r}(\A))$.
\par

First we compute the $A$-parameter for Miyawaki liftings.
\begin{prop}\label{Apara}
Suppose that $\MM^{(n)}_{\psi, \tau}(\pi) \not= 0$ with $n \geq r$.
If $\pi$ has an $A$-parameter $\Psi$, then $\MM^{(n)}_{\psi, \tau}(\pi)$ has an $A$-parameter
\[
\Psi \boxplus \tau\chi_{-1}^{[(n+r)/2]}[n-r].
\]
\end{prop}
\begin{proof}
This follows from the computation of Satake parameters for $\MM^{(n)}_{\psi_v, \tau_v}(\pi_v)$ 
(Theorem \ref{howe} (4)).
\end{proof}

Miyawaki liftings have a duality.
\begin{prop}\label{dual}
Suppose that $\MM^{(n)}_{\psi, \tau}(\pi) \not= 0$.
Then 
\[
\pi \subset \MM^{(r)}_{\psi, \tau} \left( \MM^{(n)}_{\psi, \tau}(\pi) \right).
\]
\end{prop}
\begin{proof}
For $\varphi_1, \varphi_2 \in \Sc_{k+(n+r)/2}(\Sp_{r}(F) \bs \Mp_{r}(\A))$, 
we define the Petersson inner product by
\[
\pair{\varphi_1, \varphi_2} = \int_{\Sp_{r}(F) \bs \Sp_{r}(\A)} 
\varphi_1(g,\zeta) \overline{\varphi_2(g,\zeta)} dg.
\]
If $\MM^{(n)}(\varphi, \FF) \not= 0$ for $\varphi \in \pi$ and $\FF \in \Ik_\psi^{(n+r)}(\tau)$, 
then 
\begin{align*}
\pair{
&\MM^{(r)}\left( \MM^{(n)}(\varphi, \FF), \FF \right), \varphi
}
\\&= \int_{\Sp_{r}(F) \bs \Sp_{r}(\A)} 
\left(\int_{\Sp_{n}(F) \bs \Sp_{n}(\A)} \FF(\iota(g_n, g_r), \zeta_n\zeta_r) 
\overline{\MM^{(n)}((g_n,\zeta_n)\varphi, \FF)} dg_n \right) 
\overline{\varphi(g_r, \zeta_r)} dg_r
\\&= \pair{\MM^{(n)}(\varphi, \FF), \MM^{(n)}(\varphi, \FF)} \not= 0.
\end{align*}
Hence $\pair{, }$ is nonzero on $\MM^{(r)}_{\psi, \tau} ( \MM^{(n)}_{\psi, \tau}(\pi) ) \times \pi$.
This shows that $\pi \subset \MM^{(r)}_{\psi, \tau} ( \MM^{(n)}_{\psi, \tau}(\pi) )$.
\end{proof}

The local irreducibility (Theorem \ref{howe}) implies the global irreducibility.
\begin{thm}\label{irred}
Suppose that $\pi$ has a tempered $A$-parameter, 
and $\MM^{(n)}_{\psi, \tau}(\pi) \not= 0$ with $n \geq r$.
Then:
\begin{enumerate}
\item
$\MM^{(n)}_{\psi, \tau}(\pi)$ is irreducible, and 
$\MM^{(n)}_{\psi, \tau}(\pi) \cong \otimes'_{v < \infty} \MM^{(n)}_{\psi_v, \tau_v}(\pi_v)$.

\item
When $n \leq r+1$ or $n > 2r$, we have
\[
\MM^{(r)}_{\psi, \tau}\left( \MM^{(n)}_{\psi, \tau}(\pi) \right) = \pi.
\]
\end{enumerate}
\end{thm}
\begin{proof}
Since $\pi$ has a tempered $A$-parameter, 
the local factor $\pi_v$ is almost tempered by Lemma \ref{atu}. 
By Theorem \ref{howe} (2), the local Miyawaki lift $\MM^{(n)}_{\psi_v, \tau_v}(\pi_v)$ is irreducible.
Hence the $\pi$-isotypic quotient of $\Ik^{(n+r)}_\psi(\tau)$ is of the form $\pi' \boxtimes \pi$
for some irreducible representation $\pi'$ of $\Mp_{n}(\A_\fin)$.
The $\Mp_r(\A_\fin)$-invariant surjection
\[
\Ik^{(n+r)}_\psi(\tau) \otimes \overline{\pi} \twoheadrightarrow \MM^{(n)}_{\psi, \tau}(\pi), \ 
\FF \otimes \overline{\varphi} \mapsto \MM^{(n)}(\varphi, \FF)
\]
factors through a surjective $\Mp_n(\A_\fin)$-homomorphism
$\pi' \twoheadrightarrow \MM^{(n)}_{\psi, \tau}(\pi)$.
Since $\pi'$ is irreducible, this map must be injective if $\MM^{(n)}_{\psi, \tau}(\pi) \not= 0$.
Hence $\MM^{(n)}_{\psi, \tau}(\pi)$ is irreducible.
Moreover, the above $\Mp_r(\A_\fin)$-invariant surjection 
implies that the local component at $v < \infty$ of $\MM^{(n)}_{\psi, \tau}(\pi)$
is isomorphic to $\MM^{(n)}_{\psi_v, \tau_v}(\pi_v)$.
Hence we obtain (1).
\par

By Proposition \ref{dual}, we know that
\[
\MM^{(r)}_{\psi, \tau}\left( \MM^{(n)}_{\psi, \tau}(\pi) \right) \supset \pi.
\]
By a similar argument to (1), 
when $r \leq n \leq r+1$ or $n > 2r$, 
Theorem \ref{howe} (4) implies that the left hand side is irreducible.
Hence the inclusion is an equality, and we obtain (2).
\end{proof}

\begin{cor}\label{pi}
Suppose that $\pi$ has a tempered $A$-parameter $\Psi$.
\begin{enumerate}
\item
If $n=r$ and $\MM^{(r)}_{\psi, \tau}(\pi) \not= 0$, 
then $\MM^{(r)}_{\psi, \tau}(\pi) = \pi$ as a subrepresentation of $\Sc_{k+r}(\Sp_{r}(F) \bs \Sp_{r}(\A))$.
\item
If $n=r+1$ and $\MM^{(r+1)}_{\psi, \tau}(\pi) \not= 0$, 
then $\MM^{(r+1)}_{\psi, \tau}(\pi)$ is a unique irreducible subrepresentation 
of $\Sc_{k+r+1/2}(\Sp_{r+1}(F) \bs \Mp_{r+1}(\A))$
whose local component at $v < \infty$ is isomorphic to $\mu'_v \rtimes \pi_v$
with $\mu'_v = \mu_v \chi_{-1}^r$.
\end{enumerate}
\end{cor}
\begin{proof}
By Theorem \ref{howe} (3), 
the local Miyawaki lift $\MM^{(n)}_{\psi_v, \tau_v}(\pi_v)$ 
is isomorphic to $\pi_v$ if $n = r$, and to $\mu'_v \rtimes \pi_v$ if $n = r+1$.
By Proposition \ref{Apara}, when $n = r$ (\resp $n = r+1$), 
we see that $\MM^{(n)}_{\psi, \tau}(\pi)$ has a tempered $A$-parameter $\Psi' = \Psi$
(\resp $\Psi' = \Psi \boxplus \tau \chi_{-1}^{r}$).
Since $\Sc_{k+(n+r)/2, \Psi'}$ is multiplicity-free by Theorem \ref{AMFhol} (3), 
we have $\MM^{(r)}_{\psi, \tau}(\pi) = \pi$ if $n=r$, 
and $\MM^{(r+1)}_{\psi, \tau}(\pi)$ satisfies the desired uniqueness property if $n = r+1$.
\end{proof}

Let $n' = n$ or $n' = n-1$.
Fix $l,l' \in \prod_{v \mid \infty} \Z$ such that 
$l_{v} \equiv l_{v'} \not\equiv l'_v \equiv l'_{v'} \bmod 2$ for any $v, v' \mid \infty$.
For $\varphi \in \Sc_{l/2}(\Sp_{n}(F) \bs \Mp_{n}(\A))$, 
$\varphi' \in \Sc_{l'/2}(\Sp_{n'}(F) \bs \Mp_{n'}(\A))$, and $\phi \in \Sc(X_{n'}(\A))$, 
we define a Fourier--Jacobi period $\PP_{n,n', \psi_\xi}(\varphi, \overline{\varphi'}, \phi)$ 
by the integral
\[
\left\{
\begin{aligned}
&\int_{\Sp_{n}(F) \bs \Sp_{n}(\A)} \varphi(g, \zeta) \overline{\varphi'(g,\zeta)} 
\overline{\Theta_{\psi_\xi}^{\phi}(g,\zeta)} dg, \iif n' = n, \\
&\int_{V_{n-1}(F) \bs V_{n-1}(\A)} \int_{\Sp_{n-1}(F) \bs \Sp_{n-1}(\A)} 
\varphi(v(g, \zeta)) \overline{\varphi'(g,\zeta)} 
\overline{\Theta_{\psi_\xi}^{\phi}(v(g,\zeta))} dgdv, \iif n' = n-1.
\end{aligned}
\right.
\]
The local seesaw identity (Proposition \ref{LSS}) is a local analogue of the following result.

\begin{prop}[Seesaw identity]\label{GSS}
Let $\pi$ and $\pi'$ be irreducible representations of $\Mp_{r}(\A_\fin)$ and $\Mp_{n-1}(\A_\fin)$
occurring in $\Sc_{k+(n+r)/2}(\Sp_{r}(F) \bs \Mp_{r}(\A))$ and 
$\Sc_{k+(n-1+r)/2}(\Sp_{n-1}(F) \bs \Mp_{n-1}(\A))$, respectively.
Fix a totally positive element $\xi \in F^\times$. 
\begin{enumerate}
\item
If there exist $\MM^{(n)}(\varphi_1, \FF_1) \in \MM^{(n)}_{\psi, \tau}(\pi)$, 
$\varphi'_1 \in \pi'$, and $\phi_1 \in \Sc(X_{n-1}(\A))_\xi$ such that 
\[
\PP_{n,n-1,\psi_\xi}(\MM^{(n)}(\varphi_1, \FF_1), \overline{\varphi'_1}, \phi_1) \not= 0, 
\]
then 
there exist $\varphi_2 \in \pi$, $\MM^{(r)}(\varphi_2', \FF'_2) \in \MM^{(r)}_{\psi, \tau\chi_\xi}(\pi')$, 
and $\phi_2 \in \Sc(X_{r}(\A))_\xi$ such that
\[
\PP_{r,r,\psi_\xi}(\varphi_2, \overline{\MM^{(r)}(\varphi_2', \FF'_2)}, \phi_2) \not= 0.
\]
Moreover, we can take $\varphi_2 = \varphi_1$ and $\varphi_2' = \varphi_1'$.

\item
Assume that $n+r \geq 2$. 
If there exist $\varphi_2 \in \pi$, $\MM^{(r)}(\varphi_2', \FF'_2) \in \MM^{(r)}_{\psi, \tau\chi_\xi}(\pi')$, 
and $\phi_2 \in \Sc(X_{r}(\A))_\xi$ such that
\[
\PP_{r,r,\psi_\xi}(\varphi_2, \overline{\MM^{(r)}(\varphi_2', \FF'_2)}, \phi_2) \not= 0, 
\]
then 
there exist $\MM^{(n)}(\varphi_1, \FF_1) \in \MM^{(n)}_{\psi, \tau}(\pi)$, 
$\varphi'_1 \in \pi'$, and $\phi_1 \in \Sc(X_{n-1}(\A))_\xi$ such that 
\[
\PP_{n,n-1,\psi_\xi}(\MM^{(n)}(\varphi_1, \FF_1), \overline{\varphi'_1}, \phi_1) \not= 0.
\]
Moreover, we can take $\varphi_1 = \varphi_2$ and $\varphi_1' = \varphi_2'$.

\end{enumerate}
\end{prop}
\begin{proof}
Suppose that there exist $\MM^{(n)}(\varphi, \FF) \in \MM^{(n)}_{\psi, \tau}(\pi)$, 
$\varphi' \in \pi'$, and $\phi \in \Sc(X_{n-1}(\A))_\xi$ such that 
\[
\PP_{n,n-1,\psi_\xi}(\MM^{(n)}(\varphi, \FF), \overline{\varphi'}, \phi) \not= 0.
\]
It is equal to 
\begin{align*}
&\int_{V_{n-1}(F) \bs V_{n-1}(\A)} \int_{\Sp_{n-1}(F) \bs \Sp_{n-1}(\A)} 
\\&\quad\times
\left(
\int_{\Sp_{r}(F) \bs \Sp_{r}(\A)} \FF(\iota(v, 1) \cdot \iota(g_{n-1}, g_r), \zeta_n\zeta_r) 
\overline{\varphi(g_r, \zeta_n)} dg_r
\right)
\\&\quad\times
\overline{\varphi'(g_{n-1},\zeta_{n})} 
\overline{\Theta_{\psi_\xi}^{\phi}(v(g_{n-1},\zeta_{n}))} dg_{n-1}dv.
\end{align*}
First, if we compute the integral on $Z_{n-1}(F) \bs Z_{n-1}(\A)$, then $\FF_{\psi_\xi}$ appears.
By Proposition \ref{FJcoeff2}, 
there exist $\FF'_1, \dots, \FF'_r \in \Ik^{(n+r-1)}_\psi(\tau \otimes \chi_\xi)$ 
and $\phi_1, \dots, \phi_r \in \Sc(X_{n+r-1}(\A))_\xi$ such that 
\[
\FF_{\psi_\xi}(\iota(v, 1) \cdot \iota(g_{n-1}, g_r), \zeta_n\zeta_r) 
= \sum_{i=1}^{r} \FF'_i(\iota(g_{n-1}, g_r), \zeta_n\zeta_r)
\Theta_{\psi_\xi}^{\phi_i}(\iota(v, 1) \cdot \iota(g_{n-1}, g_r), \zeta_n\zeta_r).
\]
Note that $X_{n+r-1}(\A) \cong X_{n-1}(\A) \oplus X_{r}(\A)$.
We may assume that $\phi_i = \phi_i^{(n-1)} \otimes \phi_i^{(r)}$ with 
$\phi_i^{(n-1)} \in \Sc(X_{n-1}(\A))$ and $\phi_i^{(r)} \in \Sc(X_{r}(\A))$.
Then we have
\[
\Theta_{\psi_\xi}^{\phi_i}(\iota(v, 1) \cdot \iota(g_{n-1}, g_r), \zeta_n\zeta_r)
= 
\Theta_{\psi_\xi}^{\phi_i^{(n-1)}}(v (g_{n-1}, \zeta_n))
\Theta_{\psi_\xi}^{\phi_i^{(r)}}(g_r, \zeta_r).
\]
Hence there exists $i$ such that
\begin{align*}
&\int_{\Sp_{n-1}(F) \bs \Sp_{n-1}(\A)} 
\int_{\Sp_{r}(F) \bs \Sp_{r}(\A)}
\FF_i'(\iota(g_{n-1}, g_r), \zeta_n\zeta_r) 
\overline{\varphi(g_r, \zeta_r)} \overline{\varphi'(g_{n-1}, \zeta_n)}
\Theta_{\psi_\xi}^{\phi_i^{(r)}}(g_r, \zeta_r)
\\&\times
\left(
\int_{Z_{n-1}(\A)V_{n-1}(F) \bs V_{n-1}(\A)}
\Theta_{\psi_\xi}^{\phi_i^{(n-1)}}(v (g_{n-1}, \zeta_n))
\overline{\Theta_{\psi_\xi}^{\phi}(v(g_{n-1},\zeta_{n}))} dv
\right) dg_r dg_{n-1} \not= 0.
\end{align*}
The integral on $Z_{n-1}(\A)V_{n-1}(F) \bs V_{n-1}(\A)$ is equal to the inner product $(\phi_i^{(n-1)}, \phi)$, 
which does not depend on $(g_{n-1},\zeta_{n})$.
The other integral is equal to the complex conjugate of 
\[
\PP_{r,r,\psi_\xi}(\varphi, \overline{\MM^{(r)}(\varphi', \FF'_i)}, \phi_i^{(r)}).
\]
Hence we obtain (1).
\par

Next suppose that 
there exist $\varphi \in \pi$, $\MM^{(r)}(\varphi', \FF') \in \MM^{(r)}_{\psi, \tau\chi_\xi}(\pi')$, 
and $\phi_2 \in \Sc(X_{r}(\A))_\xi$ such that
\[
\PP_{r,r,\psi_\xi}(\varphi, \overline{\MM^{(r)}(\varphi', \FF')}, \phi) \not= 0.
\]
Choose a nonzero vector $\phi' \in \Sc(X_{n-1}(\A))_\xi$.
Then by the same calculation above, we have
\begin{align*}
&(\phi', \phi') \overline{\PP_{r,r,\psi_\xi}(\varphi, \overline{\MM^{(r)}(\varphi', \FF')}, \phi)}
\\&=
\int_{\Sp_{r}(F) \bs \Sp_{r}(\A)}
\int_{\Sp_{n-1}(F) \bs \Sp_{n-1}(\A)} 
\FF'(\iota(g_{n-1}, g_r), \zeta_n\zeta_r) 
\overline{\varphi'(g_{n-1}, \zeta_n)} \overline{\varphi(g_r, \zeta_r)} 
\Theta_{\psi_\xi}^{\phi}(g_r, \zeta_r)
\\&\times
\left(
\int_{Z_{n-1}(\A)V_{n-1}(F) \bs V_{n-1}(\A)}
\Theta_{\psi_\xi}^{\phi'}(v (g_{n-1}, \zeta_n))
\overline{\Theta_{\psi_\xi}^{\phi'}(v(g_{n-1},\zeta_{n}))} dv
\right) dg_{n-1} dg_r
\\&=
\int_{Z_{n-1}(\A)V_{n-1}(F) \bs V_{n-1}(\A)}
\int_{\Sp_{n-1}(F) \bs \Sp_{n-1}(\A)} 
\\&\times
\left(
\int_{\Sp_{r}(F) \bs \Sp_{r}(\A)}
\FF'(\iota(g_{n-1}, g_r), \zeta_n\zeta_r) 
\Theta_{\psi_\xi}^{\phi' \otimes \phi}(\iota(v, 1) \cdot \iota(g_{n-1}, g_r), \zeta_n\zeta_r)
\overline{\varphi(g_r, \zeta_r)} dg_r
\right) 
\\&\times
\overline{\varphi'(g_{n-1}, \zeta_n)}
\overline{\Theta_{\psi_\xi}^{\phi'}(v(g_{n-1},\zeta_{n}))}
dg_{n-1} dv.
\end{align*}
If $n+r \geq 2$, then by Proposition \ref{FJcoeff2} (2), (3), 
there exists $\FF \in \Ik^{(n+r)}_{\psi}(\tau)$ such that
\[
\FF_{\psi_\xi}(\iota(v, 1) \cdot \iota(g_{n-1}, g_r), \zeta_n\zeta_r) 
= \FF'(\iota(g_{n-1}, g_r), \zeta_n\zeta_r)
\Theta_{\psi_\xi}^{\phi' \otimes \phi}(\iota(v, 1) \cdot \iota(g_{n-1}, g_r), \zeta_n\zeta_r).
\]
Then 
\begin{align*}
0 \not= (\phi', \phi') \overline{\PP_{r,r,\psi_\xi}(\varphi, \overline{\MM^{(r)}(\varphi', \FF')}, \phi)}
= \PP_{n,n-1, \psi_\xi}(\MM^{(n)}(\varphi, \FF), \overline{\varphi'}, \phi').
\end{align*}
Hence we obtain (2).
\end{proof}

As an application of this seesaw identity, 
we have a criterion for the non-vanishing of the Miyawaki liftings for the equal rank case.

\begin{prop}\label{n=r}
Let $\pi$ be an irreducible representation of $\Sp_r(\A_\fin)$
occurring in $\Sc_{k+r}(\Sp_r(F) \bs \Sp_r(\A))$ with $r \geq 1$.
Assume that $\pi$ has a tempered $A$-parameter, 
and that the $\xi$-th Fourier--Jacobi module $\FJ_{\psi_\xi}(\pi)$ is nonzero.
Then the following are equivalent: 
\begin{enumerate}
\item[(a)]
$\MM^{(r)}_{\psi, \tau}(\pi)$ is nonzero. 
\item[(b)]
For any irreducible representation $\pi'$ of $\Mp_{r-1}(\A_\fin)$
occurring in $\Sc_{k+r-1/2}(\Sp_{r-1}(F) \bs \Mp_{r-1}(\A))$, 
if $\PP_{r,r-1, \psi_\xi}$ is not identically zero on $\pi \times \overline{\pi'} \times \Sc(X_{r-1}(\A))_\xi$, 
then $\PP_{r,r, \psi_\xi}$ is not identically zero 
on $\pi \times \overline{\MM^{(r)}_{\psi, \tau\chi_\xi}(\pi')} \times \Sc(X_{r}(\A))_\xi$. 
\item[(c)]
There exists an irreducible representation $\pi'$ of $\Mp_{r-1}(\A_\fin)$
occurring in $\Sc_{k+r-1/2}(\Sp_{r-1}(F) \bs \Mp_{r-1}(\A))$
such that
$\PP_{r,r, \psi_\xi}$ is not identically zero 
on $\pi \times \overline{\MM^{(r)}_{\psi, \tau\chi_\xi}(\pi')} \times \Sc(X_{r}(\A))_\xi$.
\end{enumerate}
\end{prop}
\begin{proof}
First, we show that (a) implies (b). 
Suppose that $\MM^{(r)}_{\psi, \tau}(\pi)$ is nonzero.
Since $\pi$ has a tempered $A$-parameter, 
by Corollary \ref{pi} (1), we have $\MM^{(r)}_{\psi, \tau}(\pi) = \pi$.
Then by Proposition \ref{GSS} (1), we obtain the condition (b).
\par

Next, we show that (b) implies (c). 
Since $\FJ_{\psi_\xi}(\pi) \not= 0$, the map
\begin{align*}
\pi \otimes \overline{\omega_{\psi_\xi}} &\rightarrow \Sc_{k+r-1/2}(\Sp_{r-1}(F) \bs \Mp_{r-1}(\A)), \\
\varphi \otimes \overline{\phi} &\mapsto 
\int_{V_{r-1}(F) \bs V_{r-1}(\A)}\varphi(vg') \cdot \overline{\Theta_{\psi_\xi}^{\phi}(vg')} dv 
\end{align*}
is nonzero.
If we take an irreducible subrepresentation $\pi'$ in the image of this map, 
then $\PP_{r,r-1, \psi_\xi}$ is not identically zero on $\pi \times \overline{\pi'} \times \Sc(X_{r-1}(\A))_\xi$.
Applying (b) to $\pi'$, we see that $\pi'$ satisfies the condition of (c).
\par

Finally, we show that (c) implies (a).
If $\PP_{r,r, \psi_\xi}$ is not identically zero 
on $\pi \times \overline{\MM^{(r)}_{\psi, \tau\chi_\xi}(\pi')} \times \Sc(X_{r}(\A))_\xi$, 
then $\PP_{r,r-1, \psi_\xi}$ is not identically zero 
on $\MM^{(r)}_{\psi, \tau}(\pi) \times \overline{\pi'} \times \Sc(X_{r-1}(\A))_\xi$
by Proposition \ref{GSS} (2).
In particular, $\MM^{(r)}_{\psi, \tau}(\pi) \not= 0$.
This completes the proof.
\end{proof}

\section{A relation between Miyawaki liftings and the Gan--Gross--Prasad conjecture}\label{relation}
In this section, we formulate a conjecture on the non-vanishing of global Miyawaki liftings.
In addition, we relate this conjecture with 
the Gan--Gross--Prasad conjecture for the symplectic-metaplectic case.
\par

\subsection{Conjecture}
Fix $\tau$ and $\psi$ as in the previous section, 
and a totally positive element $\xi \in F^\times$. 
We denote the weight of $\otimes_{v \mid \infty}\tau_v$ 
by $\pm2k = (\pm2k_v)_{v}$ with $k_v > 0$.
Let $\pi$ be an irreducible representation of $\Mp_r(\A_\fin)$
occurring in $\Sc_{k+(n+r)/2}(\Sp_r(F) \bs \Mp_r(\A))$, 
and $\MM^{(n)}_{\psi, \tau}(\pi)$ be the Miyawaki lift of $\pi$.
The Rankin--Selberg $L$-function attached to $\pi \times \tau\chi_{-1}^r$
is denoted by $L(s, \pi \times \tau\chi_{-1}^r)$.
\par

\begin{conj}[$(M)_{r,n}$]\label{M0}
Suppose that $n \geq r$.
\begin{enumerate}
\item
When $n=r$, 
the Miyawaki lift $\MM^{(r)}_{\psi, \tau}(\pi)$ is nonzero 
if and only if $L(1/2, \pi \times \tau\chi_{-1}^r) \not= 0$.
\item
When $n > r$, 
the Miyawaki lift $\MM^{(n)}_{\psi, \tau}(\pi)$ is always nonzero.
\end{enumerate}
\end{conj}

Note that when $r=0$, Conjecture \ref{M0} is trivial since 
$\MM_{\psi, \tau}^{(n)}(\eta^n) = \Ik^{(n)}_{\psi}(\tau)$, 
where $\eta$ is the unique non-trivial character of $\Mp_0(\A_\fin) \cong \Z/2\Z$. 
When $n \geq r$ and $n \equiv r \bmod 2$, 
it is an extension of a part of Ikeda's conjecture \cite[Conjecture 5.1]{I2}.
The simplest case of this conjecture 
(the case where $F=\Q$, $n=r=1$, and both $\tau\chi_{-1}$ and $\pi$ are unramified everywhere) 
is proven by Ichino \cite{Ic} and Xue \cite{X2}.
When $n \geq r$ and $n \not\equiv r \bmod 2$, 
some examples are given by Hayashida \cite[Lemma 9.1]{Ha}.
\par

In the rest of this section, we explain a relation between Conjecture \ref{M0} 
and the Gan--Gross--Prasad conjecture.

\subsection{Gan--Gross--Prasad conjecture}
The Gan--Gross--Prasad conjecture (GGP) \cite{GGP} 
relates the non-vanishing of the Fourier--Jacobi periods
with the non-vanishing of the central values of the Rankin--Selberg $L$-functions.
In this subsection, we review the GGP conjecture for holomorphic cusp forms. 
For more precision, see also Appendix \ref{sec.GGP}.
\par

Suppose that $n' = n$ or $n' = n-1$.
Let $\pi$ and $\pi'$ be irreducible representations of $\Mp_n(\A_\fin)$ and $\Mp_{n'}(\A_\fin)$
occurring in $\Sc_{l/2}(\Sp_n(F) \bs \Mp_n(\A))$ and $\Sc_{(l-1)/2}(\Sp_{n'}(F) \bs \Mp_{n'}(\A))$, 
respectively.
Assume that $\pi$ and $\pi'$ have tempered $A$-parameters
$\Psi = \boxplus_{i=1}^{t}\tau_i$ and $\Psi' = \boxplus_{j=1}^{t'}\tau'_{j}$, respectively.
We define the Rankin--Selberg $L$-function $L(s, \pi \times \pi' \times \chi_{(-1)^{l-1}\xi})$ by
\[
L(s, \pi \times \pi' \times \chi_{(-1)^{l-1}\xi})
= \prod_{i=1}^{t} \prod_{j=1}^{t'} L(s, \tau_i \times \tau'_j \times\chi_{(-1)^{l-1}\xi}).
\]
Note that $\tau_j'$ is self-dual and unitary, so that $\overline{\tau_j'} \cong \tau'_j$ for any $j$.
\par

We state the Gan--Gross--Prasad conjecture \cite[Conjecture 24.1]{GGP} for holomorphic cusp forms.

\begin{conj}[Gan--Gross--Prasad conjecture]\label{GGP-S}
Let $\pi = \otimes'_{v < \infty} \pi_v$ and $\pi' = \otimes'_{v < \infty} \pi'_v$ be as above.
\begin{enumerate}
\item
If the Fourier--Jacobi period $\PP_{n,n', \psi_\xi}$ is not identically zero 
on $\pi \times \overline{\pi'} \times \Sc(X_{n'}(\A))_\xi$, 
then 
the central value $L(1/2, \pi \times \pi' \times \chi_{(-1)^{l-1}\xi})$ is nonzero and 
the local $\Hom$-space 
\[
\Hom_{\Sp_{n'}(F_v)}(\pi_v \otimes \overline{\pi'_v} \otimes \overline{\omega_{\psi_\xi}}, \C)
\]
is nonzero for any $v < \infty$.

\item
When $n' = n$, the converse of (1) holds.
\end{enumerate}
\end{conj}

This is not the usual formulation of the GGP conjecture.
For the usual statements and the relation with Conjecture \ref{GGP-S}, 
see Appendix \ref{sec.GGP} below.
When $n' = n$, Conjecture \ref{GGP-S} (1) was proven by Yamana \cite{Y}.

\subsection{Relation between Conjectures \ref{M0} and \ref{GGP-S}}
In this subsection, we show Conjecture \ref{M0} using the GGP conjecture
and the following hypothesis.

\begin{hypo}\label{hypo}
Let $\pi$ be an irreducible representation of $\Mp_r(\A_\fin)$
occurring in $\Sc_{l/2}(\Sp_r(F) \bs \Mp_r(\A))$ with tempered $A$-parameter.
Suppose that $\FJ_{\psi_\xi}(\pi) \not= 0$.
Then there exists an irreducible representation $\pi'$ of $\Mp_{r-1}(\A_\fin)$
occurring in $\Sc_{(l-1)/2}(\Sp_{r-1}(F) \bs \Mp_{r-1}(\A))$ with tempered $A$-parameter such that
$\PP_{r,r-1, \psi_\xi}$ is not identically zero on $\pi \times \overline{\pi'} \times \Sc(X_{r-1}(\A))_\xi$.
\end{hypo}

Hypothesis \ref{hypo} may be regarded as a global analogue of \cite[Lemma C.6]{AG}.
The main difficulty is the tempered-ness of the $A$-parameter for $\pi'$.
Since all representations occurring in $\Sc_{(l-1)/2}(\Sp_1(F) \bs \Mp_1(\A))$ 
have tempered $A$-parameters by Example \ref{n=1}, 
Hypothesis \ref{hypo} is true when $r \leq 2$.
\par

Using the GGP conjecture (Conjecture \ref{GGP-S}) and Hypothesis \ref{hypo},
we will show Conjecture \ref{M0}.
First, we consider the (almost) equal rank case, i.e., the case where $n=r$ or $n=r+1$.

\begin{thm}\label{rrr}
Assume the GGP conjecture (Conjecture \ref{GGP-S}) and Hypothesis \ref{hypo}.
Then for irreducible representations with tempered $A$-parameters, we have
\[
(M)_{r-1,r} \implies (M)_{r,r} \implies (M)_{r, r+1}.
\]
In particular, $(M)_{r,r}$ and $(M)_{r,r+1}$ for irreducible representations with tempered $A$-parameters 
are true for any $r \geq 0$.
\end{thm}
\begin{proof}
First, we show $(M)_{r-1,r} \implies (M)_{r,r}$.
Let $\pi$ be an irreducible representation of $\Mp_r(\A_\fin)$
occurring in $\Sc_{k+r}(\Sp_r(F) \bs \Sp_r(\A))$ with tempered $A$-parameter $\Psi$.
We choose a totally positive element $\xi \in F^\times$ such that $\FJ_{\psi_\xi}(\pi) \not= 0$.
By Hypothesis \ref{hypo}, 
we can find an irreducible representation $\pi'$ of $\Mp_{r-1}(\A_\fin)$
occurring in $\Sc_{k+r-1/2}(\Sp_{r-1}(F) \bs \Mp_{r-1}(\A))$ with tempered $A$-parameter $\Psi'$ such that
$\PP_{r,r-1, \psi_\xi}$ is not identically zero on $\pi \times \overline{\pi'} \times \Sc(X_{r-1}(\A))_\xi$.
\par

Suppose that $\MM_{\psi, \tau}^{(r)}(\pi) \not= 0$. 
Then by Proposition \ref{n=r} (b), we see that
$\PP_{r,r, \psi_\xi}$ is not identically zero 
on $\pi \times \overline{\MM_{\psi, \tau\chi_\xi}^{(r)}(\pi')} \times \Sc(X_{r}(\A))_\xi$.
Since $\pi'$ has a tempered $A$-parameter $\Psi'$, 
by Theorem \ref{irred} and Proposition \ref{Apara}, 
the Miyawaki lifting $\MM_{\psi, \tau\chi_\xi}^{(r)}(\pi')$ is irreducible 
and has a tempered $A$-parameter $\Psi' \boxplus \tau\chi_{(-1)^{r-1}\xi}$.
By the GGP conjecture (Conjecture \ref{GGP-S} (1)), 
we conclude that 
\[
L(1/2, \pi \times \pi' \times \chi_{-\xi}) L(1/2, \pi \times (\tau\chi_{(-1)^{r-1}\xi}) \times \chi_{-\xi}) \not= 0, 
\] 
so that $L(1/2, \pi \times \tau\chi_{(-1)^r}) \not= 0$.
\par

Conversely, suppose that $L(1/2, \pi \times \tau\chi_{(-1)^r}) \not= 0$.
By $(M)_{r-1,r}$, we have $\MM_{\psi, \tau\chi_\xi}^{(r)}(\pi') \not= 0$.
Hence by Theorem \ref{irred} and Proposition \ref{Apara}, 
it is irreducible and has a tempered $A$-parameter $\Psi' \boxplus \tau\chi_{(-1)^{r-1}\xi}$.
By the GGP conjecture (Conjecture \ref{GGP-S} (1), (2) and Theorem \ref{localGGP}), 
we see that $\PP_{r,r, \psi_\xi}$ is not identically zero 
on $\pi \times \overline{\MM_{\psi, \tau\chi_\xi}^{(r)}(\pi')} \times \Sc(X_{r}(\A))_\xi$.
Hence by Proposition \ref{n=r}, we conclude that $\MM_{\psi, \tau}^{(r)}(\pi) \not= 0$. 
This completes the proof of $(M)_{r-1,r} \implies (M)_{r,r}$.
\par

Next, we show $(M)_{r,r} \implies (M)_{r,r+1}$.
By Lemma \ref{lemA}, 
there exists an irreducible representation $\Pi = \otimes'_{v < \infty}\Pi_v$ of $\Mp_{r+1}(\A_{\fin})$
occurring in $\Sc_{k+r+1/2}(\Sp_{r+1}(F) \bs \Mp_{r+1}(\A))$ 
such that $\Pi_v = \MM^{(r+1)}_{\psi_v, \tau_v}(\pi_v)$ for each $v < \infty$, 
and the $A$-parameter for $\Pi$ is equal to $\Psi \boxtimes \tau\chi_{-1}^r$.
We choose a totally positive element $\xi \in F^\times$ such that $\FJ_{\psi_\xi}(\Pi) \not= 0$.
Using Hypothesis \ref{hypo}, 
we take an irreducible representation $\pi'$ of $\Sp_{r}(\A_\fin)$
occurring in $\Sc_{k+r}(\Sp_{r}(F) \bs \Mp_{r}(\A))$ with tempered $A$-parameter $\Psi'$ such that
$\PP_{r+1,r, \psi_\xi}$ is not identically zero on $\Pi \times \overline{\pi'} \times \Sc(X_{r}(\A))_\xi$.
By the GGP conjecture (Conjecture \ref{GGP-S} (1)), we see that 
\[
L(1/2, \Pi \times \pi' \times \chi_\xi) 
= L(1/2, \pi \times \pi' \times \chi_\xi) L(1/2, \pi' \times \tau\chi_{(-1)^r\xi})
\]
is nonzero.
Since $L(1/2, \pi' \times \tau\chi_{(-1)^r\xi}) \not= 0$, 
by $(M)_{r,r}$, 
we see that $\MM^{(r)}_{\psi, \tau\chi_\xi}(\pi')$ is nonzero. 
By Corollary \ref{pi} (1), it is equal to $\pi'$ itself.
Since $L(1/2, \pi \times \pi' \times \chi_\xi) \not= 0$, 
by the GGP conjecture (Conjecture \ref{GGP-S} (2) and Theorem \ref{localGGP}), we see that 
$\PP_{r,r, \psi_\xi}$ is not identically zero on 
\[
\pi \times \overline{\pi'} \times \Sc(X_{r}(\A))_\xi
= 
\pi \times \overline{\MM^{(r)}_{\psi, \tau\chi_\xi}(\pi')} \times \Sc(X_{r}(\A))_\xi.
\]
Then by the seesaw identity (Proposition \ref{GSS} (2)), 
we see that 
$\PP_{r+1,r, \psi_\xi}$ is not identically zero 
on $\MM^{(r+1)}_{\psi, \tau}(\pi) \times \overline{\pi'} \times \Sc(X_r(\A))_\xi$.
In particular we have $\MM^{(r+1)}_{\psi, \tau}(\pi) \not= 0$. 
This completes the proof of $(M)_{r,r} \implies (M)_{r,r+1}$.
\end{proof}

\begin{cor}\label{11}
The conjecture $(M)_{1,1}$ is true.
\end{cor}
\begin{proof}
Note that any irreducible representation occurring in $\Sc_{k+1}(\Sp_1(F) \bs \Sp_1(\A))$
has a tempered $A$-parameter (Remark \ref{n=1}).
The conjecture $(M)_{0,1}$ and Hypothesis \ref{hypo} for $r=1$ are trivial.
The GGP conjecture for $n = n' = 1$ is known
(see e.g., \cite[Proof of Theorem 7.1]{GG}, \cite[Proposition 4.1, Theorem 4.5]{Q} and \cite{X}).
Hence we have $(M)_{1,1}$.
\end{proof}

Next, we consider the going-up case, i.e., the case where $n > r+1$.
\begin{thm}\label{n>r}
For $n \geq r+2$, we have
\[
(M)_{r,n-1} \implies (M)_{r,n}.
\]
\end{thm}
\begin{proof}
Let $\pi$ be an irreducible representation of $\Mp_r(\A_\fin)$
occurring in $\Sc_{k+(n+r)/2}(\Sp_r(F) \bs \Mp_r(\A))$. 
Since $\DD^{(r)}_{k_v+(n+r-1)/2}$ is discrete series, 
by a result of Wallach \cite{Wa}, we have
\[
\varphi(g, \zeta) \overline{\Theta^{\phi}_{\psi_\xi}(g,\zeta)} \in \Sc_{k+(n+r-1)/2}(\Sp_{r}(F) \bs \Mp_{r}(\A))
\]
for any $\varphi \in \pi$ and $\phi \in \Sc(X_r(\A))_\xi$.
Take an irreducible representation $\pi'$ of $\Mp_{r}(\A_\fin)$ which appears in the space spanned by
\[
\left\{ \varphi(g, \zeta) \overline{\Theta^{\phi}_{\psi_\xi}(g,\zeta)}
\ |\ \varphi \in \pi,\ \phi \in \Sc(X_r(\A))_\xi \right\}.
\]
Then $\PP_{r,r, \psi_\xi}$ is not identically zero on $\pi \times \overline{\pi'} \times \Sc(X_{r}(\A))_\xi$.
Applying $(M)_{r,n-1}$ to $\pi'$, 
we have $\Pi' = \MM_{\psi, \tau\chi_\xi}^{(n-1)}(\pi') \not= 0$.
By Proposition \ref{dual}, we have $\pi' \subset \MM_{\psi, \tau\chi_\xi}^{(r)}(\Pi')$.
Hence $\PP_{r,r, \psi_\xi}$ is not identically zero 
on $\pi \times \overline{\MM_{\psi, \tau\chi_\xi}^{(r)}(\Pi')} \times \Sc(X_{r}(\A))_\xi$.
By the seesaw identity (Proposition \ref{GSS} (2)), 
we see that $\PP_{n,n-1, \psi_\xi}$ is not identically zero 
on $\MM_{\psi, \tau}^{(n)}(\pi) \times \overline{\Pi'} \times \Sc(X_{n-1}(\A))_\xi$.
In particular, we have $\MM_{\psi, \tau}^{(n)}(\pi) \not= 0$. 
\end{proof}

\appendix
\section{Jacquet modules of representations of metaplectic groups}\label{s.jac}
Let $F$ be a non-archimedean local field of characteristic zero.
In this appendix, 
we recall computations of Jacquet modules of 
induced representations of $\GL_k(F)$ or $\Mp_n(F)$.

\subsection{Induced representations of general linear groups}\label{GL}
Let $P(F) = M(F)N(F)$ be a parabolic subgroup of $\GL_k(F)$ 
containing the Borel subgroup consisting of upper triangular matrices.
Then the Levi part $M(F)$ is of the form $\GL_{k_1}(F) \times \dots \times \GL_{k_l}(F)$
with $k_1 + \dots + k_l = k$.
For representations $\tau_1, \dots, \tau_l$ of $\GL_{k_1}(F), \dots, \GL_{k_l}(F)$, respectively, 
we denote the normalized induced representation by
\[
\tau_1 \times \dots \times \tau_l = \Ind_{P(F)}^{\GL_k(F)}(\tau_1 \otimes \dots \otimes \tau_l).
\]
\par

A segment is a symbol $[x,y]$, 
where $x,y \in \R$ with $x-y \in \Z$.
We identify $[x,y]$ as the set 
$\{x, x-1, \dots, y\}$ if $x \geq y$, and $\{x, x+1, \dots, y\}$ if $x \leq y$.
Let $\rho$ be an irreducible unitary supercuspidal representation of $\GL_d(F)$.
Then the normalized induced representation 
\[
\rho|\cdot|^{x} \times \dots \times \rho|\cdot|^y
\]
has a unique irreducible subrepresentation, 
which is denoted by 
\[
\pair{\rho; x, \dots, y}.
\]
If $x \geq y$, this is called a Steinberg representation, 
which is a discrete series representation of $\GL_{d(|x-y|+1)}(F)$. 
If $x < y$, this is called a Speh representation.
For example, if $\rho = \mu$ be a unitary character (i.e., $d=1$) and $x<y$, 
then $\pair{\mu; x, \dots, y} = \mu|{\det}_{y-x+1}|^{(x+y)/2}$ is a character of $\GL_{y-x+1}(F)$, 
where we denote by ${\det}_k$ the determinant character of $\GL_k(F)$.
\par

\begin{defi}
Let $[x,y]$ and $[x',y']$ be two segments.
\begin{enumerate}
\item
When $(x-y)(x-y') \geq 0$, 
we say that $[x,y]$ and $[x',y']$ are linked if 
$[x,y] \not\subset [x',y']$, $[x',y'] \not\subset [x,y]$ as sets, 
and $[x,y] \cup [x,y']$ is also a segment.
\item
When $(x-y)(x'-y') < 0$, 
we say that $[x,y]$ and $[x',y']$ are linked if 
$[y,x]$ and $[x',y']$ are linked, 
and $x,y \not\in [x',y']$ and $x',y' \not\in [x,y]$. 
\end{enumerate}
\end{defi}

The linked-ness gives an irreducibility criterion for induced representations.
\begin{thm}[Zelevinsky {\cite[Theorems 4.2, 9.7]{Z}}, M{\oe}glin--Waldspurger \cite{MW}]\label{zel}
Let $[x,y]$ and $[x',y']$ be segments, 
and let $\rho$ and $\rho'$ be irreducible unitary supercuspidal representations 
of $\GL_d(F)$ and $\GL_{d'}(F)$, respectively.
Then the induced representation 
\[
\pair{\rho; x, \dots, y} \times \pair{\rho'; x', \dots, y'}
\]
is irreducible unless $[x, y]$ are $[x',y']$ are linked, and $\rho \cong \rho'$.
\end{thm}

For a partition $(k_1, k_2)$ of $k$, 
we denote by $R_{(k_1,k_2)}$ the normalized Jacquet functor of representations of $\GL_k(F)$
with respect to the standard maximal parabolic subgroup $P(F) = M(F)N(F)$ 
with $M(F) \cong \GL_{k_1}(F) \times \GL_{k_2}(F)$.
The Jacquet module of $\pair{\rho; x, \dots, y}$ is computed by Zelevinsky.

\begin{prop}[{\cite[Propositions 3.4, 9.5]{Z}}]
Let $\rho$ be an irreducible unitary supercuspidal representation of $\GL_d(F)$.
Suppose that $x \not= y$ and set $k = d(|x-y|+1)$.
Then $R_{(k_1, k_2)}(\pair{\rho; x, \dots, y}) = 0$ unless $k_1 \equiv 0 \bmod d$.
If $k_1 = dm$ with $1 \leq m \leq |x-y|$, we have
\[
R_{(k_1, k_2)}(\pair{\rho; x, \dots, y}) 
= \pair{\rho; x, \dots, x-\epsilon(m-1)} \otimes \pair{\rho; x-\epsilon m, \dots, y}, 
\]
where $\epsilon \in \{\pm1\}$ is defined so that $\epsilon(x-y) > 0$.
\end{prop}

\subsection{Representations of double covers of general linear groups}
Recall that for $a \in F^\times$, 
the Weil constant $\alpha_\psi(a)$
is an eighth root of unity, and satisfies that
\[
\frac{\alpha_{\psi}(a)\alpha_{\psi}(b)}{\alpha_{\psi}(1)\alpha_{\psi}(ab)} = \pair{a,b}
\]
for $a, b \in F$, 
where the right hand side is the Hilbert symbol.
In particular, 
\[
\left(\frac{\alpha_\psi(1)}{\alpha_\psi(a)}\right)^2 = \chi_{-1}(a)
\]
for $a \in F^\times$, 
where $\chi_{-1}$ is the quadratic character associated to $F(\I)/F$.
\par

A double cover of $\GL_k(F)$ is given by
\[
\cl\GL_k(F) = \GL_k(F) \times \{\pm1\}
\]
with group law
\[
(g_1,\epsilon_1) \cdot (g_2, \epsilon_2) = (g_1g_2, \epsilon_1\epsilon_2 \cdot \pair{\det g_1, \det g_2}).
\]
Let $\Irr(\GL_k(F))$ (\resp $\Irr(\cl\GL_k(F))$) 
be the set of equivalence classes of irreducible representations of $\GL_k(F)$
(\resp the set of equivalence classes of irreducible genuine representations of $\cl\GL_k(F)$).
For an irreducible representation $\tau$ of $\GL_k(F)$ and $(g,\epsilon) \in \cl\GL_k(F)$, 
we set
\[
\tau_\psi(g,\epsilon) = \epsilon \frac{\alpha_\psi(1)}{\alpha_\psi(\det g)} \tau(g).
\]
Then $\tau_\psi$ is irreducible and genuine, and
the map $\tau \mapsto \tau_\psi$ gives a bijection $\Irr(\GL_k(F)) \rightarrow \Irr(\cl\GL_k(F))$.
For a more precise representation theory for $\cl\GL_k(F)$, 
see \cite[\S 4.1]{HM}.
\par

\subsection{Induced representations of symplectic and metaplectic groups}
We denote by $R_{P_t(F)}$ the normalized Jacquet functor of 
representations of $\Mp_r(F)$ with respect to the maximal parabolic subgroup $\cl{P}_t(F)$.
For a smooth representation $\Pi$ of $\Mp_r(F)$, 
we write $\semi(\Pi)$ for the semisimplification of $\Pi$.

\begin{prop}\label{jacquet}
Let $\pi$ be a representation of $\Mp_r(F)$, $\mu$ be a unitary character of $F^\times$, 
and $\alpha \in \C$.
Then $\semi R_{P_t(F)}(\mu|{\det}_k|^{\alpha} \rtimes \pi)$ is isomorphic to the direct sum of 
\[
\semi\left(
\mu^{-1}|{\det}_{k-a}|^{-\alpha - \half{a}} \times \mu|{\det}_{b}|^{\alpha - \half{k-b}} \times \tau_\lam
\right)
\boxtimes 
\left(
\mu|{\det}_{a-b}|^{\alpha - \half{k-a-b}} \rtimes \pi_\lam
\right), 
\]
where $(a,b)$ runs over the pairs of integers such that $0 \leq b \leq a \leq k$ and $a-b \geq k-t$, 
and $\tau_\lam \boxtimes \pi_\lam$ runs over all irreducible subquotients of $R_{P_{a-b-k+t}(F)}(\pi)$
(with multiplicity).
\end{prop}
\begin{proof}
This follows from Tadi{\'c}'s formula \cite{T1}, \cite[Proposition 4,5]{HM}.
\end{proof}

\section{Local Langlands correspondence and Arthur's multiplicity formula}\label{langlands}
In this appendix, we summarize the local Langlands correspondence and Arthur's multiplicity formula.

\subsection{Local Langlands correspondence}
We recall the local Langlands correspondence (LLC)
for $\Sp_n(F)$ and $\Mp_n(F)$.
Let $F$ be a local field of characteristic zero. 
Fix a non-trivial unitary character $\psi$ of $F$.
We denote by $W_F$ and $\WD_F$ the Weil group and the Weil--Deligne group of $F$, respectively,
i.e., 
\[
\WD_F = \left\{
\begin{aligned}
&W_F \times \SL_2(\C) 	\iif \text{$F$ is non-archimedean}, \\
&W_F				\iif \text{$F$ is archimedean}.
\end{aligned}
\right.
\]
A representation of $\WD_F$ is a homomorphism $\phi \colon \WD_F \rightarrow \GL_N(\C)$
such that 
\begin{itemize}
\item
$\phi(\mathrm{Frob})$ is semi-simple if $F$ is non-archimedean; 
\item
$\phi|W_F$ is smooth if $F$ is non-archimedean, and $\phi$ is continuous if $F$ is archimedean; 
\item
$\phi|\SL_2(\C)$ is algebraic.
\end{itemize}
Here, $\mathrm{Frob} \in W_F$ is a (geometric) Frobenius element if $F$ is non-archimedean.
When $F$ is a non-archimedean local field of residue characteristic $p > 2$, 
we call a representation $\phi$ of $\WD_F$ unramified 
if $\phi$ is trivial on $I_F \times \SL_2(\C)$, 
where $I_F$ is the inertia subgroup of $W_F$.
\par

Set
\begin{align*}
\Phi(\Sp_n(F)) &= \{\phi \colon \WD_F \rightarrow \SO_{2n+1}(\C)\}/\cong, \\
\Phi(\Mp_n(F)) &= \{\phi \colon \WD_F \rightarrow \Sp_{n}(\C)\}/\cong.
\end{align*}
For $G_n = \Sp_n$ or $G_n = \Mp_n$, 
we call an element in $\Phi(G_n(F))$ an $L$-parameter for $G(F)$.
When $G_n = \Sp_n$ (\resp $G_n = \Mp_n$), 
any $\phi \in \Phi(G_n(F))$ can be decomposed into a direct sum
\[
\phi = m_1\phi_1 \oplus \dots \oplus m_t \phi_t \oplus \phi' \oplus (\phi')^\vee, 
\]
where 
$\phi_1, \dots, \phi_t$ are distinct irreducible orthogonal (\resp symplectic) representations of $\WD_F$, 
$m_i$ is the multiplicity of $\phi_i$ in $\phi$, 
and $\phi'$ is a sum of irreducible representations of $\WD_F$ which are not orthogonal (\resp symplectic).
We define the component group $A_\phi$ of $\phi$ by 
\[
A_\phi = \bigoplus_{i=1}^{t}(\Z/2\Z)a_i.
\]
Namely, $A_\phi$ is a free $\Z/2\Z$-module of rank $t$, 
and $\{a_1, \dots, a_t\}$ is a basis of $A_\phi$ with $a_i$ associated to $\phi_i$.
For $a = a_{i_1} + \dots + a_{i_k} \in A_\phi$ with $1 \leq i_1 < \dots < i_k \leq t$, 
put
\[
\phi^a = \phi_{i_1} \oplus \dots \oplus \phi_{i_k}.
\]
We call $z_\phi = \sum_{i=1}^{t}m_i a_i \in A_\phi$ the central element of $A_\phi$.
We denote the Pontryagin dual of $A_\phi$ 
by $\widehat{A_\phi} = \{\eta \colon A_\phi \rightarrow \{\pm1\}\}$.
\par

For $\phi \in \Phi(\Sp_n(F))$ (\resp $\phi \in \Phi(\Mp_n(F))$),
we say that:
\begin{itemize}
\item
$\phi$ is of good parity 
if $\phi$ is a direct sum of irreducible orthogonal (\resp symplectic) representations; 
\item
$\phi$ is tempered if $\phi(W_F)$ is bounded; 
\item
$\phi$ is almost tempered 
if each irreducible constituent $\phi_i$ of $\phi$ is of the form 
$\phi_i = \phi'_i|\cdot|^{s_i}$ such that $\phi'_i(W_F)$ is bounded and $-1/2 < s_i < 1/2$.
\end{itemize}
\par

The set of equivalence classes of irreducible representations of $\Mp_n(F)$ 
which are genuine (\resp not genuine) is denoted by $\Irr(\Mp_n(F))$ (\resp $\Irr(\Sp_n(F))$).
The LLC classifies $\Irr(G_n(F))$ by $\Phi(G_n(F))$ for $G_n = \Mp_n$ or $G_n = \Sp_n$.

\begin{thm}[\cite{L, AB, Ar, GS}]\label{LLC}
Let $G_n = \Mp_n$ or $G_n = \Sp_n$.
\begin{enumerate}
\item
There is a canonical surjective map
\begin{align*}
\Irr(G_n(F)) &\twoheadrightarrow \Phi(G_n(F)).
\end{align*}
For $\phi \in \Phi(G_n(F))$, we denote the inverse image of $\phi$ by $\Pi_\phi$, 
and call it the $L$-packet of $\phi$.

\item
There exists an injective map 
\[
\Pi_\phi \hookrightarrow \widehat{A_\phi}. 
\]
This is surjective if $G_n = \Mp_n$.
When $G_n = \Sp_n$, 
the image of this map is equal to 
\[
\{\eta \in \widehat{A_\phi}\ |\ \eta(z_\phi) = 1\}.
\]
When $\pi \in \Pi_\phi$ corresponds to $\eta \in \widehat{A_\phi}$, 
we call the pair $(\phi, \eta)$ the $L$-parameter for $\pi$.

\item
When $F$ is a non-archimedean local field of residue characteristic $p > 2$, 
an irreducible representation $\pi$ is unramified if and only if 
its $L$-parameter $(\phi, \eta)$ satisfies that $\phi$ is unramified and $\eta = \1$.

\item
$\pi \in \Pi_\phi$ is (almost) tempered if and only if $\phi$ is (almost) tempered. 

\item
If $\pi \in \Pi_\phi$ is discrete series, then $\phi \in \Phi(G_n(F))$ is of good parity.

\item
If $\phi$ is tempered, then one can decompose
\[
\phi = \phi_\tau \oplus \phi_0 \oplus \phi_\tau^\vee
\]
where
\begin{itemize}
\item
$\phi_0 \in \Phi(G_{n_0}(F))$ is of good parity; 
\item
$\phi_\tau$ is a sum of irreducible representations 
which are not orthogonal when $G_n = \Sp_n(F)$, 
and are not symplectic when $G_n = \Mp_n(F)$.
\end{itemize}
Let $\tau$ be the irreducible (tempered) representation of $\GL_k(F)$ corresponding to $\phi_\tau$.
Then for $\pi_0 \in \Pi_{\phi_0}$, the induced representation $\tau \rtimes \pi_0$ is irreducible, 
and the $L$-packet $\Pi_\phi$ is given by
\[
\Pi_\phi = \{ \tau \rtimes \pi_0\ |\ \pi_0 \in \Pi_{\phi_0}\}.
\]
If the $L$-parameter for $\pi_0$ is $(\phi_0, \eta_0)$, 
then the one for $\tau \rtimes \pi_0$ is $(\phi, \eta_0)$, 
where we regard $\eta_0$ as a character of $A_{\phi}$ 
via the canonical identification $A_{\phi} = A_{\phi_0}$.

\item
If $\pi = J(\tau_1|\cdot|^{s_1}, \dots, \tau_t|\cdot|^{s_t}, \pi_0)$, 
then the $L$-parameter $(\phi, \eta)$ for $\pi$ is given by
\[
\phi = 
\phi_1|\cdot|^{s_1} \oplus \dots \oplus \phi_t|\cdot|^{s_t}
\oplus \phi_0 \oplus 
\phi_t^\vee|\cdot|^{-s_t} \oplus \dots \oplus \phi_1^\vee|\cdot|^{-s_1}
\]
and $\eta = \eta_0$, 
where $(\phi_0, \eta_0)$ is the $L$-parameter for $\pi$, 
and $\phi_i$ is the (tempered) representation of $\WD_F$ corresponding to $\tau_i$ for $i = 1, \dots, t$. 

\end{enumerate}
\end{thm}

\subsection{Irreducibility criterion for standard modules}
In this subsection, we assume that $F$ is non-archimedean.
For each representation $\phi$ of $\WD_F$, 
one can consider the $L$-function $L(s, \phi)$ attached to $\phi$.
We recall a criterion for the irreducibility of standard modules 
in terms of analytic properties of $L$-functions.

\begin{prop}\label{GPR}
Let $\pi_0$ (\resp $\tau_i$) be an irreducible tempered representation 
of $\Mp_{r_0}(F)$ (\resp $\GL_{k_i}(F)$) with the $L$-parameter $(\phi_0, \eta_0)$ 
(\resp with associated representation $\phi_i$ of $\WD_F$), 
and $s_1, \dots, s_t$ be real numbers such that $s_1 > \dots > s_t > 0$.
Assume the following:
\begin{itemize}
\item
If $\pi_0$ is not genuine, then
\begin{align*}
&\left( \prod_{i=1}^{t} 
L(s-2s_i, \phi_i^\vee, \wedge^2) L(s-s_i, \phi_0 \otimes \phi_i^\vee) 
\right)
\\&\times
\left( \prod_{1 \leq i < j \leq t} 
L(s-s_i+s_j, \phi_i^\vee \otimes \phi_j) 
L(s-s_i-s_j, \phi_i^\vee \otimes \phi_j^\vee)
\right)
\end{align*}
is regular at $s=1$.
\item
If $\pi_0$ is genuine, then
\begin{align*}
&\left( \prod_{i=1}^{t} 
L(s-2s_i, \phi_i^\vee, \Sym^2) L(s-s_i, \phi_0 \otimes \phi_i^\vee) 
\right)
\\&\times
\left( \prod_{1 \leq i < j \leq t} 
L(s-s_i+s_j, \phi_i^\vee \otimes \phi_j) 
L(s-s_i-s_j, \phi_i^\vee \otimes \phi_j^\vee)
\right)
\end{align*}
is regular at $s=1$.
\end{itemize}
Then the standard module $\tau_1|\cdot|^{s_1} \times \dots \times \tau_t|\cdot|^{s_t} \rtimes \pi_0$
is irreducible.
\end{prop}
\begin{proof}
Note that our assumptions are equivalent that 
the adjoint $L$-function $L(s, \phi, \mathrm{Ad})$ is regular at $s=1$, 
where $\phi$ is the $L$-parameter for $J(\tau_1|\cdot|^{s_1}, \dots, \tau_t|\cdot|^{s_t}, \pi_0)$.
The non-genuine case is a result of Heiermann \cite{He} 
together with a conjecture of Gross--Prasad--Rallis (proven in \cite[Appendix B]{GI1}).
The genuine case is \cite[Theorem 3.13]{At1}.
\end{proof}

This proposition has several corollaries.

\begin{cor}\label{at}
If $\pi$ is irreducible and almost tempered, 
then $\pi$ is isomorphic to an irreducible standard module.
\end{cor}
\begin{proof}
This immediately follows from Proposition \ref{GPR}.
\end{proof}

\begin{cor}\label{cor2}
Let $\mu$ be a unitary character of $F^\times$.
Set $\delta = 0$ if $G_n = \Sp_n(F)$, and $\delta = 1$ if $G_n = \Mp_n(F)$.
Suppose that $\phi \in \Phi(G_n(F))$ is of good parity.
Then for any $\pi \in \Pi_\phi$ and for any positive integer $a$ with $a \not\equiv \delta \bmod 2$, 
the standard module
\[
\mu|\cdot|^{\half{a}} \rtimes \pi
\]
is irreducible.
\end{cor}
\begin{proof}
Since $\phi$ is of good parity and $a \not\equiv \delta \bmod 2$, 
we see that $L(s-a/2, \phi \otimes \mu^{-1})$ is regular at $s=1$.
If $\delta = 0$, then $L(s-a, \mu^{-1}, \wedge^2) = 1$, which is entire.
If $\delta = 1$, then $L(s-a, \mu^{-1}, \Sym^2) = L(s-a, \mu^{-2})$, 
which is regular at $s=1$ when $a \not= 1$.
By Proposition \ref{GPR}, we obtain the corollary.
\end{proof}

\begin{cor}\label{cor3}
Let $\pi \in \Pi_\phi$ with $\phi$ of good parity, 
If $\semi R_{P_1(F)}(\pi)$ contains $\chi|\cdot|^{\alpha} \boxtimes \pi_0$ 
for some unitary character $\chi$ and $\alpha \in \R$, 
then $2\alpha \in \Z$ and $2\alpha \equiv \delta \bmod 2$.
\end{cor}
\begin{proof}
First, we assume that $\pi$ is discrete series.
Replacing $\pi_0$ if necessary, 
we may assume that 
\[
\pi \hookrightarrow \chi|\cdot|^{\alpha} \rtimes \pi_0.
\]
By Casselman's criterion, we see that
\begin{itemize}
\item
$\alpha > 0$; 
\item
$\pi_0$ is discrete series; 
\item
the standard module $\chi|\cdot|^{\alpha} \rtimes \pi_0$ is reducible.
\end{itemize}
In particular, if $\pi_0 \in \Pi_{\phi_0}$, 
then $L(s-\alpha, \phi_0 \otimes \chi^{-1})$ has a pole at $s=1$, or 
$L(s-2\alpha, \chi^{-1}, \wedge^2)$ (\resp $L(s-2\alpha, \chi^{-1}, \Sym^2)$) has a pole at $s=1$ 
when $\delta = 0$ (\resp $\delta = 1$).
Since $\phi_0$ is of good parity, 
this condition implies that $2\alpha \in \Z$ and $2\alpha \equiv \delta \bmod 2$.
\par

In general, if $\pi \in \Pi_\phi$ with $\phi$ of good parity, 
then 
\[
\pi \hookrightarrow \tau_1 \times \dots \times \tau_r \rtimes \pi_0, 
\]
where 
\begin{itemize}
\item
$\tau_i$ is an irreducible discrete series representation of $\GL_{k_i}(F)$
to which the irreducible representation $\phi_i$ of $\WD_F$ corresponding 
is orthogonal if $\delta = 0$, and symplectic if $\delta = 1$; 
\item
$\pi_0$ is an irreducible discrete series representation of $\Mp_{n_0}(F)$.
\end{itemize}
Then the assertion follows from 
Tadi{\'c}'s formula \cite{T1} \cite[Proposition 4,5]{HM}.
\end{proof}

\subsection{Lowest weight modules}
In this subsection, we assume that $F = \R$ and $\psi(x) = \exp(2\pi a \I x)$ for some $a > 0$.
Recall that we denote the irreducible lowest weight representation of $\Mp_n(\R)$
with lowest $\cl{K}_\infty$-type ${\det}^{l/2}$ by $\DD^{(n)}_{l/2}$.
We determine the $L$-parameters for $\DD^{(n)}_{l/2}$
when it is discrete series, which is equivalent that $l > 2n$.
Note that the infinitesimal character of $\DD^{(n)}_{l/2}$ is equal to
\[
\left(\half{l}-1, \half{l}-2, \dots, \half{l}-n\right).
\]
\par

Recall that the Weil group $W_\R$ of $\R$ is of the form
\[
W_\R = \C^\times \cup \C^\times j
\]
with the group low
\[
j^2 = -1, 
\quad
jzj^{-1} = \overline{z}
\]
for $z \in \C^\times$.
Let $\sgn$ be the sign character of $W_\R$ 
defined by $\sgn(j) = -1$ and $\sgn(z) = 1$ for $z \in \C^\times$.
For each positive integer $k$, we define a $2$-dimensional irreducible representation 
$\rho_k \colon W_\R \rightarrow \GL_2(\C)$ by 
\[
\rho_k(j) = 
\begin{pmatrix}
0 & (-1)^k \\ 1 & 0
\end{pmatrix},
\quad
\rho_k(re^{\I \theta}) = 
\begin{pmatrix}
e^{k \I \theta} & 0 \\ 0 & e^{-k \I \theta}
\end{pmatrix}
\]
for $r > 0$ and $\theta \in \R/2\pi\Z$.
Note that (a conjugate of) the image of $\rho_k$ is contained in $\mathrm{O}(2,\C)$ if $k$ is even, 
and in $\SL_2(\C)$ if $k$ is odd.
\par

For an integer $l > 2n$, we put
\[
\phi_l^{(n)} = 
\left\{
\begin{aligned}
&\rho_{l-2} \oplus \rho_{l-4} \oplus \dots \oplus \rho_{l-2n} \oplus \sgn^n	\iif l \equiv 0 \bmod 2, \\
&\rho_{l-2} \oplus \rho_{l-4} \oplus \dots \oplus \rho_{l-2n}				\iif l \equiv 1 \bmod 2.
\end{aligned}
\right.
\]
Then $\phi_l^{(n)} \in \Phi(\Sp_n(\R))$ if $l$ is even, and $\phi_l^{(n)} \in \Phi(\Mp_n(\R))$ if $l$ is odd.
When $\phi = \phi_l^{(n)}$, the component group $A_{\phi}$ is given by
\[
A_{\phi} = 
\Big((\Z/2\Z) e_{l-2} \oplus (\Z/2\Z) e_{l-4} \oplus \dots \oplus (\Z/2\Z) e_{l-2n} \Big) + (\Z/2\Z)z_{\phi},  
\]
where $e_k \in A_{\phi}$ is the element associated to $\rho_k \subset \phi$.

\begin{prop}\label{Dl}
Suppose that $l > 2n$.
Then the $L$-parameter for $\DD^{(n)}_{l/2}$ is equal to $(\phi_l^{(n)}, \eta_l^{(n)})$, 
where $\eta_l^{(n)}$ is determined 
by $\eta_l^{(n)}(e_{l-2i}) = (-1)^{i-1}$ for $i = 1, \dots, n$.
\end{prop}
\begin{proof}
Set $\phi = \phi_l^{(n)}$.
Then the $L$-packet $\Pi_{\phi}$ is the set of discrete series representations with infinitesimal character 
\[
\left(\half{l}-1, \half{l}-2, \dots, \half{l}-n\right).
\]
Hence $\DD^{(n)}_{l/2} \in \Pi_{\phi}$.
The description of $\eta_l^{(n)}$ can be proven by a similar way to \cite[Appendix A]{At2}
using Schmid's character identity.
We omit the detail.
\end{proof}

\subsection{Arthur's multiplicity formula}\label{app.AMF}
In this subsection, 
we let $F$ be a totally real number field, and $\psi$ be a non-trivial unitary character of $\A/ F$.
A discrete global $A$-parameter for $\Sp_n(F)$ (\resp $\Mp_n(F)$) is a symbol
\[
\Psi = \tau_1[d_1] \boxplus \dots \boxplus \tau_t[d_t], 
\]
where
\begin{itemize}
\item
$\tau_i$ is an irreducible cuspidal unitary automorphic representation of $\GL_{m_i}(\A)$; 
\item
$d_i$ is a positive integer such that $\sum_{i=1}^{t}m_id_i$ 
is equal to $2n+1$ (\resp $2n$); 
\item
if $d_i$ is odd (\resp even), then $L(s, \tau_i, \Sym^2)$ has a pole at $s=1$; 
\item
if $d_i$ is even (\resp odd), then $L(s, \tau_i, \wedge^2)$ has a pole at $s=1$; 
\item
the central character $\omega_i$ of $\tau_i$ satisfies that $\omega_1^{d_1} \cdots \omega_l^{d_l} = \1$; 
\item
if $i \not= j$ and $\tau_i \cong \tau_j$, then $d_i \not= d_j$.
\end{itemize}
Two $A$-parameters $\Psi = \boxplus_{i=1}^{t}\tau_{i}[d_{i}]$ and $\Psi' = \boxplus_{i=1}^{t'}\tau'_{i}[d'_{i}]$
are said to be equivalent if $t = t'$ and there exists a permutation $\sigma \in \mathfrak{S}_t$ such that
$d'_i = d_{\sigma(i)}$ and $\tau'_i \cong \tau_{\sigma(i)}$ for each $i$.
We denote the set of equivalence classes of discrete global $A$-parameters 
for $\Sp_n(F)$ (\resp $\Mp_n(F)$) by $\Psi_2(\Sp_n(F))$ (\resp $\Psi_2(\Mp_n(F))$).
We call an $A$-parameter $\Psi = \boxplus_{i=1}^{t}\tau_{i}[d_{i}]$ tempered if $d_i = 1$ for any $i$.
In this case, we write $\Psi = \boxplus_{i=1}^{t}\tau_{i}$ for simplicity.
\par

Let $\Psi = \boxplus_{i=1}^{t}\tau_{i}$ be a tempered $A$-parameter for $\Sp_n(F)$ (\resp $\Mp_n(F)$).
For each place $v$ of $F$, we denote the representation of $\WD_{F_v}$ corresponding to $\tau_{i,v}$
by $\phi_{i,v}$, 
and put $\Psi_v = \oplus_{i=1}^{t} \phi_{i,v}$.
By \cite[Theorem 1.4.2]{Ar}, we have $\Psi_v \in \Phi(\Sp_n(F_v))$ (\resp $\Psi_v \in \Phi(\Mp_n(F_v))$).
We define a global $A$-packet $\Pi_{\Psi}$ by 
\[
\Pi_{\Psi} = \{\pi = \otimes'_v \pi_v\ |\ \pi_v \in \Pi_{\Psi_v},\ \text{$\pi_v$ is unramified for almost all $v$}\}.
\]
\par

For a tempered $A$-parameter $\Psi = \boxplus_{i=1}^{t}\tau_{i}$, 
the global component group $A_{\Psi}$ of $\Psi$ is defined by 
\[
A_\Psi = \bigoplus_{i=1}^{t}(\Z/2\Z)\alpha_i.
\]
Namely, $A_\Psi$ is a free $\Z/2\Z$-module of rank $t$, 
and $\{\alpha_1, \dots, \alpha_t\}$ is a basis of $A_\Psi$ with $\alpha_i$ associated to $\tau_i$.
There exists a localization map $A_{\Psi} \ni \alpha \mapsto \alpha_v \in A_{\Psi_v}$, so that
we have a diagonal map
\[
\Delta \colon A_{\Psi} \rightarrow \prod_{v}A_{\Psi_v}.
\]
\par

Let $\AA_2(\Mp_n(\A))$ (\resp $\AA_2(\Sp_n(\A))$) be the set of square-integrable automorphic forms
on $\Mp_n(\A)$ which are genuine (\resp not genuine).
This is an $\Mp_n(\A_\fin) \times \prod_{v \mid \infty}(\sp_n(\C), \cl{K}_\infty)$-module.
Arthur's multiplicity formula classifies $\AA_2(\Mp_n(\A))$ (\resp $\AA_2(\Sp_n(\A))$) 
in terms of $\Psi_2(\Mp_n(F))$ (\resp $\Psi_2(\Sp_n(F))$).
\begin{thm}[Arthur's multiplicity formula ({\cite[Theorem 1.5.2]{Ar}, \cite[Theorems 1.1, 1.3]{GI2})}]\label{AMF}
Let $G= \Sp_n$ or $G = \Mp_n$.
\begin{enumerate}
\item
For each $\Psi \in \Psi_2(G(F))$, there exists an $\Mp_n(\A_\fin) \times (\sp_n(\C), \cl{K}_\infty)$-submodule 
$\AA_{2, \Psi}$ of $\AA_2(G(\A))$ such that
\[
\AA_2(G(\A)) = \bigoplus_{\Psi \in \Psi_2(G(F))} \AA_{2, \Psi}.
\]

\item
Suppose that $\pi = \otimes_{v}' \pi_v$ is an irreducible subrepresentation of $\AA_{2, \Psi}$ 
with $\Psi = \boxplus_{i=1}^{t}\tau_{i}[d_{i}]$.
Then for almost all $v < \infty$, 
the local factors $\pi_v$ of $\pi$ and $\tau_{i,v}$ of $\tau_i$ are unramified for any $i$.
Moreover, if we denote the Satake parameter for $\tau_{i,v}$ by $\{c_{i,v,1}, \dots, c_{i,v,m_i}\}$, 
then the Satake parameter of $\pi_v$ is equal to
\[
\bigcup_{i=1}^{t} \bigcup_{j=1}^{m_i} 
\left\{ c_{i,v,j}q^{-\half{d_i-1}}, c_{i,v,j}q^{-\half{d_i-3}}, \dots, c_{i,v,j}q^{\half{d_i-1}} \right\}
\]
as multisets.

\item
When $\Psi = \boxplus_{i=1}^{t}\tau_{i}$ is a tempered $A$-parameter, 
$\AA_{2, \Psi}$ is the multiplicity-free direct sum of representations $\pi = \otimes'_v \pi_v \in \Pi_{\Psi}$ 
such that the character $\eta_v$ of $A_{\Psi_v}$ associated to $\pi_v$ satisfies the equation
\[
\left(\prod_{v} \eta_v\right) \circ \Delta(\alpha_i) = 
\left\{
\begin{aligned}
&\1			\iif G = \Sp_n, \\
&\ep(\tau_i)	\iif G = \Mp_n.
\end{aligned}
\right.
\]
Here, $\ep(\tau_i) = \ep(1/2, \tau_i, \psi) \in \{\pm1\}$ is the global root number 
attached to the symplectic representation $\tau_i$.

\end{enumerate}
\end{thm}

If a representation $\pi$ of $\Mp_n(\A)$ is contained in $\AA_{2, \Psi}$, 
we say that $\pi$ has an $A$-parameter $\Psi$, and $\Psi$ is the $A$-parameter for $\pi$.
We note that the following.

\begin{lem}\label{atu}
If $\Psi$ is a tempered $A$-parameter, and 
$\pi = \otimes_{v}' \pi_v$ is an irreducible subrepresentation of $\AA_{2, \Psi}$, 
then the local factors $\pi_v$ are almost tempered and unitary for all $v$.
\end{lem}
\begin{proof}
Since any representation $\pi$ contained in $\AA_{2, \Psi}$ is unitary, 
its local factors $\pi_v$ are unitary for all $v$.
By the toward Ramanujan conjecture (see \cite[(2.5) Corollary]{JS} and \cite[Appendix]{RS}), 
we see that $\pi_v$ are almost tempered for all $v$.
\end{proof}

Let $l = (l_v)_v \in \prod_{v \mid \infty}\Z$ with $l_v > 0$ 
and $l_v \equiv l_{v'} \bmod 2$ for any $v,v' \mid \infty$.
For $\Psi \in \Psi_2(\Sp_n(F))$ if all $l_v$ are even, and for $\Psi \in \Psi_2(\Mp_n(F))$ if all $l_v$ are odd, 
we set 
\[
\Sc_{l/2, \Psi} = \AA_{2, \Psi} \cap \Sc_{l/2}(\Sp_n(F) \bs \Mp_n(\A)).
\]
Hence
\[
\Sc_{l/2}(\Sp_n(F) \bs \Mp_n(\A)) \cong \bigoplus_{\Psi \in \Psi_2(G(F))} \Sc_{l/2, \Psi}.
\]
By Proposition \ref{Dl}, 
if $\Psi$ is tempered and $\Sc_{l/2, \Psi}$ is nonzero, 
then $\Psi_v = \phi_{l_v}^{(n)}$ for each $v \mid \infty$.
The subspace $\Sc_{l/2, \Psi}$ is often zero.

\begin{ex}\label{n=1}
Suppose that $n=1$.
There is a unique non-tempered $A$-parameter $\1[3]$ for $\Sp_1(F)$, 
where $\1$ is the trivial representation of $\GL_1(\A)$.
The associated subspace $\AA_{2, \1[3]}$ is the space of constant functions, 
i.e., $\AA_{2,\1[3]}$ is the trivial $\Sp_1(\A_\fin) \times (\sp_1(\C), K_\infty)$-module.
On the other hand, 
the non-tempered $A$-parameters for $\Mp_1(F)$ are of the form $\chi_\xi[2]$.
The associated space $\AA_{2, \chi_\xi[2]}$ is isomorphic to 
the even Weil representation $\omega_{\psi_\xi}^+$
as $\Mp_1(\A_\fin) \times (\sp_1(\C), \cl{K}_\infty)$-modules.
Therefore, for $\Psi = \1[3]$ or $\Psi = \chi_\xi[2]$, we have $\Sc_{l/2, \Psi} = 0$
since the trivial representation and the even Weil representations are not cuspidal.
In other words, if $n=1$ and $\Sc_{l/2, \Psi} \not= 0$, then $\Psi$ is tempered.
\end{ex}
\par

To show Theorem \ref{rrr}, we need the following lemma.
\begin{lem}\label{lemA}
Let $\tau = \otimes_v' \tau_v$ be an irreducible unitary cuspidal automorphic representation of $\GL_2(\A)$
satisfying the conditions (A1), (A2) and (A3) in \S \ref{sec.global}.
We denote the weight of $\otimes_{v \mid \infty}\tau_v$ by $\pm2k = (\pm2k_v)_{v}$ with $k_v > 0$.
Let $\Psi = \boxplus_{i=1}^{t}\tau_{i}$ be a tempered $A$-parameter for $\Mp_r(F)$. 
Assume that for each infinite place $v \mid \infty$, 
the local factor $\Psi_v$ is equal to $\phi_{l_v}^{(r)}$ with $l_v = 2k_v+2r+1$.
Then $\Psi' = \Psi \boxplus \tau\chi_{-1}^{r}$ is a tempered $A$-parameter for $\Mp_{r+1}(F)$.
Moreover, for any irreducible subrepresentation $\pi = \otimes'_{v < \infty}\pi_v$ of $\Sc_{l_v/2, \Psi}$, 
there exists an irreducible subrepresentation $\pi' = \otimes'_{v < \infty}\pi'_v$ of $\Sc_{l_v/2, \Psi'}$ 
such that for each finite place $v < \infty$, 
the $L$-parameters $(\Psi_v, \eta_v)$ and $(\Psi'_v, \eta'_v)$ for $\pi$ and $\pi'$, respectively, 
are related by $\eta'_v = \eta_v$ under the canonical identification $A_{\Psi'_v} = A_{\Psi_v}$.
\end{lem}
\begin{proof}
For $v \mid \infty$, 
the local factor $\Psi_v = \phi_{l_v}^{(r)}$ does not contain $\rho_{2k_v-1}$, 
which is the local factor of $\tau\chi_{-1}^{r}$ at $v$.
Since the central character of $\tau\chi_{-1}^{r}$ is trivial, 
the exterior square $L$-function $L(s, \tau\chi_{-1}^r, \wedge^2)$ has a pole at $s=1$
by \cite[Corollary 7.5]{KR}.
Hence $\Psi' = \Psi \boxplus \tau\chi_{-1}^r$ is a tempered $A$-parameter for $\Mp_{r+1}(F)$.
\par

Let $\pi = \otimes'_{v < \infty}\pi_v$ be an irreducible subrepresentation of $\Sc_{l/2, \Psi}$, 
and define $\pi' = \otimes'_{v < \infty}\pi'_v$ so that 
the $L$-parameter for $\pi'_v$ is equal to $(\Psi'_v, \eta_v)$ for each $v < \infty$, 
where $(\Psi_v, \eta_v)$ is the $L$-parameter for $\pi_v$.
To show that $\pi'$ is a subrepresentation of $\Sc_{l/2, \Psi'}$, 
we use Arthur's multiplicity formula (Theorem \ref{AMF} (3)).
Note that the $L$-parameters $(\phi_{l_v/2}^{(r)}, \eta_{l_v})$ and $(\phi_{l_v/2}^{(r+1)}, \eta'_{l_v})$ 
for $\DD^{(r)}_{l_v/2}$ and $\DD^{(r+1)}_{l_v/2}$ are related by
$\eta'_{l_v}| A_{\phi_{l_v/2}^{(r)}} = \eta_{l_v}$ and $\eta_{l_v}'(a_v) = (-1)^r$, 
where $a_v \in A_{\phi_{l_v/2}^{(r+1)}}$ is the element associated to $\rho_{2k_v -1}$.
Hence for the element $\alpha_i \in A_{\Psi} \subset A_{\Psi'}$ associated to $\tau_i$, 
we have
\[
\left(\prod_{v<\infty} \eta_v'(\alpha_{i,v}) \right) \left(\prod_{v \mid \infty} \eta'_{l_v}(\alpha_{i,v})\right)
=
\left(\prod_{v<\infty} \eta_v(\alpha_{i,v}) \right) \left(\prod_{v \mid \infty} \eta_{l_v}(\alpha_{i,v})\right)
= \ep(\tau_i), 
\]
and for the element $\alpha \in A_{\Psi'}$ associated to $\tau\chi_{-1}^r$, 
we have
\begin{align*}
&\left(\prod_{v<\infty} \eta_v'(\alpha_{v}) \right) \left(\prod_{v \mid \infty} \eta'_{l_v}(\alpha_{v})\right)
= \left(\prod_{v \mid \infty} (-1)^r\right)
= \left(\prod_{v \mid \infty} \chi_{-1,v}^r(-1)\right)
\\&=
\left(\prod_{v<\infty} (\mu_v\chi_{-1,v}^r)(-1) \right) 
\left(\prod_{v \mid \infty} (-1)^{k_v}\right)
= \ep(\tau\chi_{-1}^r).
\end{align*}
Here, we use the condition $(A3)$ on $\tau$.
By Arthur's multiplicity formula (Theorem \ref{AMF} (3)), 
we see that $\pi'$ occurs in $\Sc_{l/2, \Psi'}$.
\end{proof}

\section{Gan--Gross--Prasad conjecture}\label{sec.GGP}
The Gan--Gross--Prasad conjecture (GGP) predicts 
a relation between the non-vanishing of
the Fourier--Jacobi periods with the non-vanishing of the central values 
of the Rankin--Selberg $L$-functions.
In this appendix, 
we review the statements of the GGP conjecture, its refined version and its local version.
As an application of local Miyawaki liftings, 
we prove a new case of the local GGP conjecture in \S \ref{s.new}.

\subsection{Global GGP conjecture}
Let $F$ be a totally real number field, and $\psi$ be a non-trivial unitary character of $\A/ F$.
Recall that for $\xi \in F^\times$, 
the Weil representation $\omega_{\psi_\xi}$ of $\cl{J}_{n}(\A) = \Mp_{n}(\A) \ltimes V_{n}(\A)$
is realized on $\Sc(X_{n}(\A))$.
For $\phi \in \Sc(X_{n}(\A))$, we define the theta function by
\[
\Theta^{\phi}_{\psi_\xi}(vg') = \sum_{x \in X_{n}(\A)} \omega_{\psi_\xi}(vg') \phi(x)
\]
for $v \in V_n(\A)$ and $g' \in \Mp_n(\A)$.
This is a genuine automorphic form on $\cl{J}_{n}(\A)$. 
\par

Suppose that $n' = n$ or $n' = n-1$.
Let $\varphi$ and $\varphi'$ be cusp forms on $\Mp_n(\A)$ and $\Mp_{n'}(\A)$, 
respectively.
Assume that exactly one of $\varphi$ or $\varphi'$ is genuine.
Then for $\phi \in \Sc(X_{n'}(\A))$, 
we define the Fourier--Jacobi period $\PP_{n,n',\psi_\xi}(\varphi, \varphi', \phi)$ by 
\begin{align*}
\PP_{n,n,\psi_\xi}(\varphi, \varphi', \phi)
&=
\int_{\Sp_{n}(F) \bs \Sp_{n}(\A)} \varphi(g, \zeta) \varphi'(g,\zeta) 
\overline{\Theta_{\psi_\xi}^{\phi}(g,\zeta)} dg, \\
\PP_{n,n-1,\psi_\xi}(\varphi, \varphi', \phi) 
&=
\int_{V_{n-1}(F) \bs V_{n-1}(\A)} \int_{\Sp_{n-1}(F) \bs \Sp_{n-1}(\A)} \varphi(v(g, \zeta)) \varphi'(g,\zeta) 
\overline{\Theta_{\psi_\xi}^{\phi}(v(g,\zeta))} dgdv.
\end{align*}

The GGP conjecture is stated as follows:
\begin{conj}[Gan--Gross--Prasad conjecture {\cite[Conjecture 24.1]{GGP}}]\label{GGP}
Suppose that $n' = n$ or $n' = n-1$.
Let $\pi = \otimes'_{v} \pi_v$ and $\pi' = \otimes'_{v} \pi'_v$ be 
irreducible cuspidal automorphic representations of $\Mp_n(\A)$ and $\Mp_{n'}(\A)$, respectively.
Assume that exactly one of $\pi$ or $\pi'$ is genuine, 
and that the $A$-parameters for $\pi$ and $\pi'$ are tempered, 
Then the following are equivalent:
\begin{enumerate}
\item
The Fourier--Jacobi period $\PP_{n,n', \psi_\xi}$ is not identically zero 
on $\pi \times \pi' \times \Sc(X_{n'}(\A))$; 
\item
the central value $L(1/2, \pi \times \pi' \times \chi_{\xi})$ is nonzero and 
the local $\Hom$-space 
\[
\Hom_{\Sp_{n'}(F_v)}(\pi_v \otimes \pi'_v \otimes \overline{\omega_{\psi_{\xi,v}}}, \C)
\]
is nonzero for any place $v$ of $F$.
\end{enumerate}
\end{conj}

For the case where $n' = n$, 
Yamana \cite[Theorem 1.1]{Y} proved that (1) implies (2).
See also \cite[Remark 6.6]{Y}.
Conjecture \ref{GGP} implies Conjecture \ref{GGP-S} (1)
since the $A$-parameter for $\overline{\pi'}$ is given by
\[
\left\{
\begin{aligned}
&\Psi' = \boxplus_{j=1}^{t'} \tau'_j \iif \text{$\pi'$ is not genuine}, \\
&\Psi' \otimes \chi_{-1} = \boxplus_{j=1}^{t'} (\tau'_j \otimes \chi_{-1})
\iif \text{$\pi'$ is genuine},
\end{aligned}
\right.
\]
where $\Psi' = \boxplus_{j=1}^{t'} \tau'_j$ is the $A$-parameter for $\pi'$.
\par 

To obtain Conjecture \ref{GGP-S} (2), we need a refined version of Conjecture \ref{GGP}.
To state this, we define a local period integral.
For each place $v$, we fix a Haar measure on $\Sp_{n'}(F_v)$ such that 
$\vol(\Sp_{n'}(\oo_v); dg'_v) = 1$ for almost all $v$, 
and that $dg' = \prod_v dg'_v$ is the Tamagawa measure on $\Sp_{n'}(\A)$.
Fix an $\Sp_{n}(F_v)$-invariant inner product $\pair{,}_{\pi_v}$ of $\pi_v \times \pi_v$
such that
\[
\int_{\Sp_{n}(F) \bs \Sp_{n}(\A)} \varphi_1(g) \overline{\varphi_2(g)} dg
= \prod_{v}\pair{\varphi_{1,v}, \varphi_{2,v}}
\]
for $\varphi_1 = \otimes_v \varphi_{1,v}, \varphi_2 = \otimes_v \varphi_{2,v} \in \pi$. 
Similarly, fix an $\Sp_{n'}(F_v)$-invariant inner product $\pair{,}_{\pi'_v}$ of $\pi'_v \times \pi'_v$.
Recall that $\omega_{\psi_{\xi,v}}$ has a $\cl{J}_{n'}(F_v)$-invariant pairing 
\[
(\phi_{1,v}, \phi_{2,v})_v = \int_{X_{n'}(F_v)} \phi_{1,v}(x_v) \overline{\phi_{2,v}(x_v)}dx_v, 
\]
which satisfies that $(\phi_1, \phi_2) = \prod_{v}(\phi_{1,v}, \phi_{2,v})_v$
for $\phi_1 = \otimes_v \phi_{1,v}, \phi_2 = \otimes_v \phi_{2,v} \in \Sc(X_{n'}(\A))$.
For $\varphi_v \in \pi_v$, $\varphi'_v \in \pi'_v$ and $\phi_v \in \Sc(X_{n'}(F_v))$, 
we define a local period integral by
\begin{align*}
&\alpha_v(\varphi_v, \varphi'_v, \phi_v) 
\\&= 
\left\{
\begin{aligned}
&\int_{\Sp_{n}(F_v)} \pair{\pi_v(g)\varphi_v, \varphi_v}_{\pi_v} \pair{\pi'_v(g)\varphi'_v, \varphi'_v}_{\pi'_v} 
\overline{(\omega_{\psi_{\xi,v}}(g)\phi_v, \phi_v)} dg
\iif n'=n, \\
&\int_{V_{n-1}(F_v)}\int_{\Sp_{n-1}(F_v)} 
\pair{\pi_v(ug)\varphi_v, \varphi_v}_{\pi_v} \pair{\pi'_v(g)\varphi'_v, \varphi'_v}_{\pi'_v} 
\overline{(\omega_{\psi_{\xi,v}}(ug)\phi_v, \phi_v)} dgdu
\iif n'=n-1.
\end{aligned}
\right.
\end{align*}
This integral is absolutely convergent by \cite[Proposition 2.2.1]{X}.

The refined version of the GGP conjecture is stated as follows:
\begin{conj}[{Xue \cite[Conjecture 2.3.1]{X}}]\label{refined}
Let $\pi = \otimes'_{v} \pi_v$ and $\pi' = \otimes'_{v} \pi'_v$ be 
irreducible cuspidal automorphic representations of $\Mp_n(\A)$ and $\Mp_{n'}(\A)$
with tempered $A$-parameters $\Psi$ and $\Psi'$, respectively.
Assume that exactly one of $\pi$ or $\pi'$ is genuine.
Fix a sufficiently large finite set $S$ of places of $F$ containing all ``bad'' places.

\begin{enumerate}
\item
For each place $v$, 
the local $\Hom$-space 
$\Hom_{\Sp_{n'}(F_v)}(\pi_v \otimes \pi_v' \otimes \overline{\omega_{\psi_{\xi,v}}}, \C)$
is nonzero 
if and only if 
$\alpha_v(\varphi_v, \varphi'_v, \phi_v) \not= 0$ 
for some $\varphi_v \in \pi_v$, $\varphi'_v \in \pi'_v$ and $\phi_v \in \omega_{\psi_\xi}$.

\item
Let $\varphi = \otimes_v \varphi_v \in \pi$, $\varphi' = \otimes_v \varphi'_v$, 
and $\phi = \otimes_v \phi_v \in \Sc(X_{n'}(\A))$ be factorizable elements.
Then there exists a constant $\Delta^S$, which is an explicit product of partial zeta values, 
such that
\begin{align*}
&|\PP_{n,n', \psi_\xi}(\varphi, \varphi', \phi)|^2
\\&
= \frac{2\Delta^S}{|A_{\Psi}| |A_{\Psi'}|}
\left.\frac{L^S(s, \pi \times \pi' \times \chi_\xi)}
{L^S(s+\half{1}, \pi, \Ad)L^S(s+\half{1}, \pi, \Ad)}\right|_{s= \half{1}}
\times
\prod_{v \in S} \alpha_v(\varphi_v, \varphi'_v, \phi_v).
\end{align*}
Here, $L^S(s, \pi \times \pi' \times \chi_\xi)$ is the partial Rankin--Selberg $L$-function, 
and $L^S(s, \pi, \Ad)$ and $L^S(s, \pi', \Ad)$ are the partial adjoint $L$-functions of $\pi$ and $\pi'$, 
respectively.
\end{enumerate}
\end{conj}

Conjecture \ref{refined} together with the following lemma implies Conjecture \ref{GGP-S} (2). 

\begin{lem}\label{arch-period}
Assume that $n'=n$.
Let $v$ be a real place of $F$, and $\xi$ be a positive real number (in $F_v$).
Then for nonzero lowest weight vectors
$\varphi_{l/2} \in \DD_{l/2}^{(n)}$, $\varphi_{(l-1)/2} \in \DD_{(l-1)/2}^{(n)}$, 
and $\phi_{\xi}^0 \in \omega_{\psi_\xi}$, 
we have $\alpha_v(\varphi_{l/2}, \overline{\varphi_{(l-1)/2}}, \phi_{\xi}^0) \not= 0$.
\end{lem}
\begin{proof}
Set $\pi_v =  \DD_{l/2}^{(n)}$ and $\pi'_v =  \DD_{(l-1)/2}^{(n)}$.
Note that the subspace of $\DD_{(l-1)/2}^{(n)} \otimes \omega_{\psi_\xi}$ 
on which $\cl{K}_\infty$ acts by ${\det}^{l/2}$ is one dimension, and 
is spanned by $\varphi_{(l-1)/2} \otimes \phi_{\xi}^0$. 
Since $\Hom_{\Mp_n(\R)}(\DD_{l/2}^{(n)}, \DD_{(l-1)/2}^{(n)} \otimes \omega_{\psi_\xi}) \not=0$, 
we may regard $\varphi_{(l-1)/2} \otimes \phi_{\xi}^0$ as a lowest weight vector of 
$\DD_{l/2}^{(n)} \subset \DD_{(l-1)/2}^{(n)} \otimes \omega_{\psi_\xi}$.
Hence we can regard 
\[
\pair{\pi'_v(g)\varphi_{(l-1)/2}, \varphi_{(l-1)/2}}_{\pi_v'} (\omega_{\psi_{\xi}}(g)\phi_\xi^0, \phi_\xi^0) 
\]
as a (nonzero) matrix coefficient of $\DD_{l/2}^{(n)}$ 
on which the left-right translation of $\cl{K}_\infty \times \cl{K}_\infty$ 
is equal to ${\det}^{l/2} \boxtimes {\det}^{l/2}$.
Such a matrix coefficient is a scalar multiple of $\pair{\pi_v(g)\varphi_{l/2}, \varphi_{l/2}}_{\pi_v}$, 
so that $\alpha_v(\varphi_{l/2}, \overline{\varphi_{(l-1)/2}}, \phi_{\xi}^0) \not= 0$. 
\end{proof}

\subsection{Local GGP conjecture}
In this subsection, 
we let $F$ be a non-archimedean local field of characteristic zero.
Fix a non-trivial additive character $\psi$ of $F$ and an element $\xi \in F^\times$.
Let $\pi_1$ and $\pi_2$ be irreducible representations of $\Mp_r(F)$ and $\Mp_{r-1}(F)$, 
respectively. 
Assume that exactly one of them is genuine.
We set
\[
d_{r,r-1, \xi}(\pi_1, \pi_2) 
= \dim\Hom_{J_{r-1}(F)}(\pi_1 \otimes \pi_2 \otimes \overline{\omega_{\psi_\xi}^{(r-1)}}, \C).
\]
Similarly, let $\pi_1'$ and $\pi_2'$ be irreducible representations of $\Mp_n(F)$.
Assume that exactly one of them is genuine.
We set
\[
d_{n,n, \xi}(\pi_1', \pi_2') 
= \dim\Hom_{\Sp_{n}(F)}(\pi_1' \otimes \pi_2' \otimes \overline{\omega_{\psi_\xi}^{(n)}}, \C).
\]
The multiplicity one theorem 
proven by Sun \cite{S} and Gan--Gross--Prasad \cite[Corollary 16.2]{GGP}
asserts that 
\[
d_{r,r-1, \xi}(\pi_1, \pi_2) \leq 1, \quad 
d_{n,n, \xi}(\pi'_1, \pi'_2) \leq 1
\]
for any $\pi_1, \pi_2, \pi_1', \pi_2'$.
\par

For a symplectic representation $\phi$ of $\WD_F$, 
we denote the root number attached to $\phi$ by $\ep(\phi) = \ep(1/2, \phi, \psi)$.
This value does not depend on $\psi$, and is in $\{\pm1\}$.
The local Gan--Gross--Prasad conjecture (GGP) for the symplectic-metaplectic case 
(proven by the author \cite{At1})
gives complete descriptions of $d_{r,r-1, \xi}(\pi_1, \pi_2)$ and $d_{n,n, \xi}(\pi_1', \pi_2')$
in terms of internal structures of $L$-packets
when all $\pi_1, \pi_2, \pi_1', \pi_2'$ are almost tempered.

\begin{thm}[{\cite[Theorem 1.3, Corollary 1.4]{At1}, \cite[Proposition 18.1]{GGP}}]\label{localGGP}
Assume that all $\pi_1, \pi_2, \pi_1', \pi_2'$ are almost tempered.
Let $(\phi_1, \eta_1)$, $(\phi_2, \eta_2)$, $(\phi_1', \eta_1')$, $(\phi_2', \eta_2')$ 
be the $L$-parameters for $\pi_1, \pi_2, \pi_1', \pi_2'$, respectively.
Then $d_{r,r-1,\xi}(\pi_1, \pi_2) \not= 0$ if and only if 
\begin{align*}
\left\{
\begin{aligned}
&\eta_1(a) = 
\ep(\phi_1^a \otimes \phi_2 \otimes \chi_\xi) \ep(\phi_1 \otimes \phi_2 \otimes \chi_\xi)^{\det(a)}
\det(\phi_1^a)((-1)^{\half{1}\dim(\phi_2)}\xi),\\
&\eta_2(b) =
\ep(\phi_1 \otimes \phi_2^b \otimes \chi_\xi) \ep(\phi_2^b) \chi_\xi(-1)^{\half{1}\dim(\phi_2^b)}
\end{aligned}
\right.
\end{align*}
for $a \in A_{\phi_1}$ and $b \in A_{\phi_2}$, 
and 
$d_{n,n,\xi}(\pi'_1, \pi'_2) \not= 0$ if and only if 
\begin{align*}
\left\{
\begin{aligned}
&\eta'_1(a) = 
\ep({\phi'}_1^a \otimes \phi'_2 \otimes \chi_\xi) \ep(\phi'_1 \otimes \phi'_2 \otimes \chi_\xi)^{\det(a)}
\det({\phi'}_1^a)((-1)^{\half{1}\dim(\phi'_2)}\xi),\\
&\eta'_2(b) =
\ep(\phi'_1 \otimes {\phi'}_2^b \otimes \chi_\xi) \ep({\phi'}_2^b) \chi_\xi(-1)^{\half{1}\dim({\phi'}_2^b)}
\end{aligned}
\right.
\end{align*}
for $a \in A_{\phi'_1}$ and $b \in A_{\phi'_2}$.
\end{thm}

\subsection{Application of local Miyawaki liftings}\label{s.new}
As in Proposition \ref{LSS}, 
local Miyawaki liftings satisfy a seesaw identity.
Using this, one can describe $d_{n,n, \xi}(\pi'_1, \pi'_2)$ for a new case.
Recall that the unique irreducible algebraic representation of $\SL_2(\C)$ of dimension $d$ 
is denoted by $S_d$.

\begin{thm}\label{GGP-new}
Let $\pi_1$ and $\pi_2$ be irreducible almost tempered unitary representations 
of $\Mp_r(F)$ and $\Mp_{r-1}(F)$, 
on which $\{\pm1\}$ acts by $(\pm 1)^{n+r}$ and $(\pm 1)^{n+r-1}$, respectively.
Fix a unitary character $\mu$ of $F^\times$.
\begin{enumerate}
\item
Assume one of the following conditions:
\begin{itemize}
\item
The $L$-parameter $\phi_1$ for $\pi_1$ does not contain $\mu^{\pm1} S_d$ 
for any $d \geq n-r$ with $d \equiv n-r \bmod2$;

\item
$n=r$ or $n=r+1$.
\end{itemize}
Then we have 
\begin{align*}
d_{n,n,\xi}((\mu\chi_{-1}^{n+r-1} \circ {\det}_{n-r}) \rtimes \pi_1, 
(\mu\chi_{\xi} \circ {\det}_{n-r+1}) \rtimes \pi_2)
= d_{r,r-1, \xi}(\pi_1, \pi_2).
\end{align*}
In particular, $d_{n,n,\xi}((\mu\chi_{-1}^{n+r-1} \circ {\det}_{n-r}) \rtimes \pi_1, 
(\mu\chi_{\xi} \circ {\det}_{n-r+1}) \rtimes \pi_2)$
can be described in terms of internal structures of the $L$-packets for $\pi_1$ and $\pi_2$.

\item
Set $\mu' = \mu\chi_{-1}^{[(n+r-1)/2]}$ and $\tau' = \mu' \times \mu'^{-1}$.
If an irreducible representation $\pi_1'$ of $\Mp_n(F)$ satisfies 
that $d_{n,n, \xi}(\pi'_1, (\mu\chi_\xi \circ {\det}_{n-r+1}) \rtimes \pi_2)  \not= 0$, 
then $\MM_{\psi, \tau'}^{(r)}(\pi_1') \not= 0$.
\end{enumerate}
\end{thm}
\begin{proof}
Note that $\mu\chi_{-1}^{n+r-1} = \mu' \chi_{-1}^{[(n+r)/2]}$.
By Theorem \ref{howe} (2), (4) and the seesaw identity (Proposition \ref{LSS}), we have
\begin{align*}
&d_{n,n,\xi}((\mu\chi_{-1}^{n+r-1} \circ {\det}_{n-r}) \rtimes \pi_1, 
(\mu\chi_{\xi} \circ {\det}_{n-r+1}) \rtimes \pi_2) \not= 0
\\& \iff
d_{n,n,\xi}(\MM_{\psi, \tau'}^{(n)}(\pi_1), \MM_{\psi, \tau'\chi_\xi}^{(n)}(\pi_2)) \not= 0
\\& \iff
\Hom_{\Mp_n(F)}\left(
\MM^{(n)}_{\psi, \tau'\chi_\xi}(\pi_2^\vee) \otimes \omega_{\psi_\xi}^{(n)}, 
\MM^{(n)}_{\psi, \tau'}(\pi_1)
\right) \not= 0
\\& \iff 
\Hom_{\cl{J}_{r-1}}\left(
\MM^{(r)}_{\psi, \tau'}\left( \MM^{(n)}_{\psi, \tau'}(\pi_1) \right), 
\pi_2^\vee \otimes \omega_{\psi_\xi}^{(r-1)}
\right) \not= 0
\\& \iff
\Hom_{\cl{J}_{r-1}}\left(
\pi_1, \pi_2^\vee \otimes \omega_{\psi_\xi}^{(r-1)}
\right) \not= 0
\\& \iff
d_{r,r-1,\xi}(\pi_1, \pi_2) \not= 0.
\end{align*}
Hence we obtain (1).
\par

When $d_{n,n,\xi}(\pi'_1, (\mu\chi_\xi \circ {\det}_{n-r+1}) \rtimes \pi_2)  \not= 0$, 
the seesaw identity implies that
\[
\Hom_{\cl{J}_{r-1}}\left(
\MM^{(r)}_{\psi, \tau'}(\pi'_1), 
\pi_2^\vee \otimes \omega_{\psi}^{(r-1)}
\right) \not= 0.
\]
In particular, $\MM^{(r)}_{\psi, \tau'}(\pi'_1) \not= 0$.
Hence we have (2).
\end{proof}


\end{document}